\documentclass[12pt]{amsart}
\usepackage{amssymb,latexsym,amsmath,amsfonts}
\usepackage{supertabular}
\usepackage{enumerate}
\usepackage[mathscr]{eucal}
\usepackage{enumitem}
\usepackage{hyperref}

\usepackage[nameinlink]{cleveref}

\allowdisplaybreaks

\numberwithin{equation}{section}
\theoremstyle{definition}
\newtheorem{definition}{Definition}[section]
\newtheorem{example}[definition]{Example}

\theoremstyle{remark}
\newtheorem{remark}[definition]{Remark}

\theoremstyle{plain}
\newtheorem{proposition}[definition]{Proposition}
\newtheorem{theorem}[definition]{Theorem}
\newtheorem{lemma}[definition]{Lemma}
\newtheorem{result}[definition]{Result}
\newtheorem{corollary}[definition]{Corollary}
\newtheorem{obs}{Observation}

\usepackage[usenames]{color}

\definecolor{Red}{rgb}{1,0,0}
\definecolor{Magenta}{rgb}{0.69,0,0.83}
\definecolor{Green}{rgb}{0.117,0.706,0.314}
\definecolor{Grey}{rgb}{0.8,0.8,0.8}
\definecolor{Blue}{rgb}{0.1,0.1,1}

\newcommand{\unitdisk}{\mathbb{D}}
\newcommand{\ball}{\mathbb{B}}

\newcommand{\koba}{\mathsf{k}}
\newcommand{\dkoba}{\kappa}

\newcommand{\lk}{(\lambda, \kappa)}

\newcommand{\rprt}{\mathsf{Re}}
\newcommand{\iprt}{\mathsf{Im}}
\newcommand{\distance}{\mathrm{dist}}
\newcommand{\dtb}[1]{\delta_{#1}}
\newcommand{\clos}[1]{\overline{#1}}
\newcommand{\cl}{\mathrm{cl}}
\newcommand{\lraw}{\longrightarrow}
\newcommand{\wt}{\widetilde}
\newcommand{\wh}{\widehat}
\newcommand{\range}{\mathsf{ran}}
\newcommand{\diam}{\mathsf{diam}}

\newcommand*{\defeq}{\mathrel{\vcenter{\baselineskip0.5ex \lineskiplimit0pt \hbox{\scriptsize.}\hbox{\scriptsize.}}}=}
\newcommand*{\defines}{=\mathrel{\vcenter{\baselineskip0.5ex \lineskiplimit0pt \hbox{\scriptsize.}\hbox{\scriptsize.}}}}

\newcommand{\bdy}{\partial}
\newcommand{\OM}{\Omega}

\newcommand{\smoo}{\mathcal{C}}

\newcommand{\bcdot}{\boldsymbol{\cdot}}

\newcommand{\C}{\mathbb{C}} 
\newcommand{\R}{\mathbb{R}}

\newcommand{\Z}{\mathbb{Z}}
\newcommand{\posint}{\mathbb{Z}_{+}}

\begin{document}
\title[Visibility]{Visibility property in one and several variables and its applications}
\author{Vikramjeet Singh Chandel, Sushil Gorai, Anwoy Maitra and Amar Deep Sarkar}

\address{VSC: Department of Mathematics and Statistics, Indian Institute of Technology Kanpur,
Kanpur -- 208 016, India}
\email{vschandel@iitk.ac.in}
	
\address{SG: Department of Mathematics and Statistics, Indian Institute of Science Education 
and Research Kolkata, Mohanpur -- 741 246, India}
\email{sushil.gorai@iiserkol.ac.in}

\address{AM: Department of Mathematics and Statistics, Indian Institute of Technology Kanpur,
Kanpur -- 208 016, India}
\email{anwoym@iitk.ac.in}
    
\address{ADS: School of Basic Sciences, Indian Institute of Technology Bhubaneswar, Argul --
752 050, India}
\email{amar@iitbbs.ac.in}



\begin{abstract}
In this paper we report our investigations on visibility with respect to the Kobayashi 
distance and its applications, with a special focus on 
planar domains. We prove that totally disconnected subsets of the boundary are 
removable in the context of visibility. We also show that a domain in 
$\C^n$ is a local weak visibility domain if and only if it is a weak visibility 
domain. The above holds also for visibility. Along the way, we prove an intrinsic 
localization result for the Kobayashi distance. Moreover, we observe some 
interesting consequences of weak visibility; for example, weak visibility 
implies compactness of the end topology of the closure of the domain. For 
planar domains: 
\begin{itemize}
\item [(i)] We provide 
examples of visibility domains 
that are not locally Goldilocks at any boundary point.
	 
\item[(ii)] We provide certain general conditions on 
planar domains that yield 
the continuous extension of conformal maps, generalizing the 
Carath\'{e}odory extension theorem. Our conditions are quite general and assume 
{\em very little} regularity of the boundary. We demonstrate this 
through examples. 
	  
\item[(iii)] We also provide conditions for the homeomorphic extension of 
biholomorphic  maps up to the boundary. 
	  
\item[(iv)] We prove that 
a hyperbolic, simply connected domain possesses the visibility property if and only 
if its boundary is locally connected. This leads us to 
reformulate the MLC conjecture in terms of visibility. 
	  
	  
\item[(v)] We provide a characterization of visibility for a large class of planar domains including certain uncountably connected domains.
\end{itemize}

\end{abstract}

\maketitle

\section{Introduction and statement of results}\label{Sec:Intro}
The notion of visibility manifold was introduced by Eberlein and O'Neill \cite{EON} in the context of
non-positively curved simply connected Riemannian 
Manifolds. Very recently Bharali and Zimmer \cite{Bharali_Zimmer} considered this concept for domains in $\C^n$ with respect to the Kobayashi 
distance. When $n\geq 2$, given a hyperbolic domain $\OM\subset\C^n$, it is a very difficult problem to determine
if the distance space $(\OM, \koba_\OM)$ is complete\,---\,here $\koba_\OM$ denotes the Kobayashi
distance on $\OM$. Completeness is a natural condition that implies, in this case, 
that any two points can be joined by a geodesic. Bharali and Zimmer \cite{Bharali_Zimmer} introduced visibility property using almost-geodesics instead of
geodesics and showed that a class of domains, which they called  
{\em Goldilocks domains}, satisfy the visibility property.
For multidimensional domains, there are few general tools to prove visibility: showing that the domain is {\em locally a Goldilocks domain} 
(see \cite[Theorem~1.4]{BZ2023}) or that it satisfies a slightly weaker hypothesis 
\cite[Theorem~1.3]{CMS2023} (see also \cite[Remark~1.5]{BZ2023}). 
It turns out that, in the case of planar domains, things are very different and interesting; this was already pointed out in 
\cite{BNT2022}. Here we provide an example of a planar visibility domain which is not locally Goldilocks at any point of its boundary.
The construction 
uses a continuous but nowhere differentiable function introduced by Takagi \cite{Takagi}. This points towards the possibility that the visibility 
property may be the right tool to study the geometry of planar domains whose boundary is a fractal set. The question of understanding the 
geometry of fractal sets was mentioned in Gromov's book \cite{GromovBook}.
\smallskip

Continuous/homeomorphic extension of conformal one-one maps between planar domains up to 
the boundary is an old problem, going back to Carathéodory for 
simply connected domains. The question of homeomorphic extension on circle domains up to 
the boundary is also the first step towards the rigidity 
conjecture \cite{He-Schramm1994} and the rigidity conjecture is required to approach to the 
Koebe conjecture. The Koebe conjecture has been proved for countably 
connected domains (see \cite{He-Schramm1993} for details). Therefore, 
continuous/homeomorphic extension up to the boundary is a very important question 
in complex analysis in one variable. Using visibility as a tool, we provide 
conditions under which continuous/homeomorphic 
extensions of conformal maps up to the boundary exist (see {\Cref{thm:ext_biholo}} for a 
precise statement). Our result covers a broad class of domains for 
which recent results \cite{Luo-Yao2022} and \cite{Ntalampekos2023} do not apply. Note that 
in the case of hyperbolic planar domains, the            
Kobayashi distance coincides with the hyperbolic distance, which is the integrated form of
the hyperbolic metric, which in turn
is induced from the Poincar{\'e} metric of the unit disc through the covering map.
\smallskip

Visibility of planar domains has a very deep connection with local connectedness of the 
boundary. Bracci--Nikolov--Thomas \cite[Corollary~3.4]{BNT2022} 
proved that bounded simply connected domains satisfy the visibility property if and only if 
their boundaries are locally connected. 
We prove this for any (not necessarily bounded) simply connected planar domain. A general result about visibility in this paper says that if a domain $\Omega$ satisfies the visibility property for every two points in 
$\bdy\Omega\setminus S$, where $S$ is a totally disconnected set, then $\Omega$ is a visibility domain.
Therefore, a reformulation of the MLC conjecture is possible here. MLC conjecture was posed by Douady and Hubbard in 1980s.
\smallskip


\noindent {\bf A Reformulation of the MLC conjecture:} \\
{\em The complement of the Mandelbrot set in $\C$ satisfies the visibility property outside a totally disconnected subset of the boundary.}

\noindent This reformulation of the MLC conjecjure is interesting because it allows us
to check visibility outside a totally disconnected set which might be of full measure as there is no 
restriction on the measure of the totally disconnected set. Another general result about visibility in this paper says that a domain in $\C^n$ is a 
visibility domain if it is locally a visibility domain at each boundary point outside 
a totally disconnected set. This led us to believe, vaguely, that 
local visibility should imply local connectivity. We could actually prove a bit more: {\em the boundary of a simply connected planar domain is locally connected if it is 
locally connected outside a totally disconnected set.} This surprising rigidity of local connectedness of planar domains lead us to the following observation which might be known to experts. 

\begin{obs}
     MLC if and only if MLC except a totally disconnected subset of the boundary of the Mandelbrot set.
\end{obs}

\noindent There have been 
some advancements towards the proof of the MLC conjecture by Yoccoz, Lyubich and Kahn  others \cite{K, L1,DL, KL1, KL2, CS} by the method of different types of renormalizations. In certain 
classes of points in the boundary of the Mandelbrot set the local connectedness is already proved. The main problem that remains is whether the
complement of that set is totally disconnected.
\smallskip

We will now provide the statements of our main results. For that we need some definitions. 
In what follows, $\OM$ shall denote a Kobayashi hyperbolic domain with $\koba_\OM$, $\dkoba_\OM$
being the Kobayashi distance and Kobayashi--Royden metric, respectively, associated with $\OM$. 

\begin{definition}
Let $\Omega$ be a domain in $\C^d$ and let $\lambda\geq 1$ and $\kappa\geq 0$ be given. A map $\sigma: I\to\Omega$, where $I\subset\R$ is an interval, is 
said to {\em be a $(\lambda,\kappa)$-almost-geodesic (with respect to $\koba_\OM$)} if the following are satisfied:
\begin{itemize}
\item[(i)] $\forall\,s,t\in I,\;\frac{1}{\lambda}|s-t|-\kappa\leq \koba_\OM(\sigma(s),\sigma(t))\leq \lambda |s-t|+\kappa$;
\smallskip

\item[(ii)] $\sigma$ is absolutely continuous, whence $\sigma'(t)$ exists for a.e.~$t\in I$, and $\dkoba_\Omega(\sigma(t);\sigma'(t))\leq\lambda$ for 
almost every $t\in I$.
\end{itemize}
\end{definition}

The following definitions were introduced in \cite{Bharali_Zimmer} for bounded domains and in \cite{BZ2023} for unbounded domains. To introduce these definitions, we need to consider the end compactification $\clos{\OM}^{End}$
of the closure $\overline{\OM}$ of a domain $\OM\subset\C^d$. 
We refer the reader to \Cref{SS:endcpt} for the definition of the end compactification 
and the end topology introduced on the space $\clos{\OM}^{End}$. In what follows 
$\bdy\clos{\OM}^{End}:=\clos{\OM}^{End}\setminus\OM$; it consists of the ordinary
Euclidean boundary $\bdy\OM$ of $\OM$ plus the ends $\clos{\OM}^{End}\setminus\clos{\OM}$
of $\clos{\OM}$.

\begin{definition} \label{dfn:lk-visib-pair-points}
Given a hyperbolic domain $\OM\subset\C^d$, given $\lambda\geq 1$ and $\kappa>0$, and 
given a pair of distinct 
points $\zeta,\eta\in\bdy\clos{\OM}^{End}$,
the pair $\zeta,\eta$ is said to {\em satisfy ({\rm or} possess) the 
$(\lambda,\kappa)$-visibility property (with respect to $\koba_\OM$)} or to 
{\em satisfy ({\rm or} possess) the visibility property with respect to 
$(\lambda,\kappa)$-almost-geodesics (for $\koba_\OM$)} if there exist 
neighbourhoods $V_\zeta$ of $\zeta$ and $V_\eta$ of $\eta$ in $\clos{\OM}^{End}$ 
such that $\clos{V}_\zeta\cap \clos{V}_\eta=\emptyset$ and a compact $K\subset\OM$ 
such that for every $\lk$-almost-geodesic $\sigma:[0,T]\to\OM$ with respect to 
$\koba_\OM$ with $\sigma(0)\in V_\zeta$ and $\sigma(T)\in V_\eta$, 
$\sigma([0,T])\cap K\neq\emptyset$.
\end{definition}
	
\begin{definition} \label{dfn:lambda-vis-vis-wvis-pair-points}
Given a hyperbolic domain $\OM\subset\C^d$, given a pair of distinct 
points $\zeta,\eta\in\bdy\clos{\OM}^{End}$ and given $\lambda\geq 1$, the pair 
$\zeta,\eta$ is said to {\em satisfy ({\rm or} possess) the $\lambda$-visibility 
property (with respect to $\koba_\OM$)} if, for every $\kappa>0$, the 
pair $\zeta,\eta$ satisfies the $\lk$-visibility property. It is said to {\em satisfy 
({\rm or} possess) the visibility property (with respect to
$\koba_\OM$)} if, for every $\lambda\geq 1$, it satisfies the $\lambda$-visibility 
property. It is said to {\em satisfy ({\rm or} possess) the 
weak visibility property (with respect to $\koba_\OM$)} if it satisfies the $1$-
visibility property.
\end{definition}
	
\begin{definition} \label{dfn:lk-lambda-vis-dom}
Given a hyperbolic domain $\OM\subset\C^d$ and given $\lambda\geq 1$ and $\kappa>0$, 
$\OM$ is said to {\em be a $\lk$-visibility domain (with respect to 
$\koba_\OM$)} or to {\em satisfy ({\rm or} possess) the $\lk$-visibility property 
(with respect to $\koba_\OM$)} or to {\em satisfy ({\rm or} possess) 
the visibility property with respect to $\lk$-almost-geodesics (for $\koba_\OM$)} if 
every pair of distinct points $\zeta,\eta\in\bdy\clos{\OM}^{End}$ satisfies the 
$\lk$-visibility property. It is said to {\em be a $\lambda$-visibility domain (with 
respect to $\koba_\OM$)} or to {\em satisfy ({\rm or} possess) the 
$\lambda$-visibility property (with respect to $\koba_\OM$)} if every pair of distinct 
points $\zeta,\eta\in\bdy\clos{\OM}^{End}$ satisfies the $\lambda$-visibility
property. 
\end{definition}
	
\begin{definition} \label{dfn:vis-wvis-dom}
Given a hyperbolic domain $\OM\subset\C^d$, $\OM$ is said to {\em be a visibility domain (with respect to the Kobayashi distance)} or to 
{\em satisfy ({\rm or} possess) the visibility property (with respect to the Kobayashi distance)} if, for every $\lambda\geq 1$, $\OM$ satisfies the 
$\lambda$-visibility property. It is said to {\em be a weak visibility domain (with respect to the Kobayashi distance)} or to {\em satisfy ({\rm or}
possess) the weak visibility property (with respect to the Kobayashi distance)} if $\OM$ satisfies the $1$-visibility property.
\end{definition}
	




In what follows, we shall frequently use the notion of convergence in the 
Hausdorff distance. For the sake of clarity and completeness, we recall that, 
given a compact distance space $(X,d)$, the Hausdorff distance is a distance
defined on the set of all non-empty closed (equivalently, compact) subsets of 
$X$ that turns this set into a compact distance space in its own right. In 
particular, given any
sequence of non-empty closed subsets of $X$, it has a subsequence that 
converges in the Hausdorff distance to a non-empty closed subset of $X$. 
Further, if $(A_n)_{n\geq 1}$ is a sequence of non-empty closed subsets of 
$X$ that converges to some set $A$ in the Hausdorff distance, then $A$ 
consists precisely of the set of all subsequential limits of sequences 
$(x_n)_{n\geq 1}$ such that, for every $n$, $x_n\in A_n$. Finally, it is not 
difficult to show that if $(A_n)_{n\geq 1}$ is a sequence of non-empty, 
closed, {\em connected} sets converging in the Hausdorff distance to a 
(non-empty, closed) set $B$, then $B$ is necessarily connected.

\subsection{General results about visibility} \label{ss:gen_res_visib}
In this subsection, we present our main
results for general domains in $\C^d$ that satisfy
a version of visibility property . These results and
a few other are contained in Section~\ref{S:gen_loc_glob}. We begin with:

\begin{theorem}\label{thm-totally disconnected}
Let $\OM\subset\C^d$ be a hyperbolic domain
(not necessarily bounded). 
Suppose that for some
$\lambda\geq 1$, $\kappa\geq0$, and some totally disconnected set 
$S\subset\bdy\OM$, 
every pair of distinct points of $\bdy\OM\setminus S$
satisfies the visibility property with 
respect to $(\lambda,\kappa)$-almost-geodesics for $\koba_{\OM}$.
Then $\OM$ 
is a $(\lambda,\kappa)$-visibility domain. 
\end{theorem}
\noindent A consequence of the above result is that, in order to check if a domain $\OM$ 
possesses a particular type of visibility property, one only has to check whether pairs 
of distinct points of $\bdy\OM$ possess this property, i.e., there is no need to verify
the condition at the {\em ends} of $\clos{\OM}$.
\smallskip 

We now present our second result.

\begin{theorem}\label{thm:cont_surj_im_vis_dom_vis}
Let $\OM\subset\C^d$ be a hyperbolic domain (not necessarily bounded) and 
$\OM_0\subset\C^d$ be a visibility domain. If there exists a biholomorphism $\Phi:\OM_0\to\OM$ 
that extends to a continuous surjective map from $\clos{\OM_0}^{\text{End}}$ to 
$\clos{\OM}^{\text{End}}$, then $\OM$ is a visibility domain.
\end{theorem}
\noindent The above result is not only of independent interest but it turns out to be a useful tool in proving 
our results for planar domains.
\smallskip

Next we have a result that says that, under rather general 
conditions, local and global (weak)
visibility are equivalent. In the definition below, we 
specify what we mean by
``local (weak) visibility'' (``global (weak) visibility'' simply 
means (weak) visibility as defined in
\Cref{dfn:vis-wvis-dom}). We point out that, in the recent past, 
Nikolov--{\"O}kten--Thomas 
have proved very similar results (see 
\cite[Theorem~4.1, 4.2]{Nik_Okt_Tho}). Nevertheless, there are
important differences between their results and ours. Specifically,
they make a seemingly technical, but crucial, assumption about the
connectedness of certain domains obtained as intersections, whereas
we explicitly avoid making this assumption. This
is {\em not} merely a
technical matter: the specific definition of local visibility that
we use not only informs the theorems that we are able to prove
(such as \Cref{thm:glob_vis_iff_loc_vis} below, which has no
artificial connectedness-of-intersections-type assumption), but
is also inspired by the condition that the codomain in 
\Cref{thm:ext_biholo} is required to satisfy (compare the 
definition below with Condition~2 as defined in 
\Cref{S:Visibility_Thm}); and the latter theorem appears to be
{\em natural} in the sense that, given the evidence of
Carath{\'e}odory's result dealing with the continuous extension
of planar biholomorphisms (see, for instance, 
\cite[Theorem~4.3.1]{BCM}), it seems that local connectedness of
the boundary (which is essentially the condition imposed on the
codomain in \Cref{thm:ext_biholo}) is the most appropriate
one in so far as one is looking for a 
continuous-extension-type result.

\begin{definition}
A hyperbolic domain $\Omega\subset\C^d$ is said to be a {\em local (weak) 
visibility domain at $p\in\bdy\Omega$} if, for every sequence $(x_n)_{n\geq 1}$
in $\OM$ converging to $p$, there exist a subdomain $U$ of $\OM$ and a 
subsequence $(x_{k_n})_{n\geq 1}$ of $(x_n)_{n\geq 1}$ such that
$p\in (\bdy U\cap\bdy\OM)\setminus\clos{\bdy U\cap\OM}$, such that $x_{k_n}\in U$ 
for all $n$ and such that every two points of
$(\bdy U\cap\bdy\OM)\setminus\clos{\bdy U\cap\OM}$ possess the (weak) visibility 
property with respect to $\koba_U$. We say that $\Omega$ is a {\em local (weak) 
visibility domain} if it is a local (weak) visibility domain at $p$ for every 
$p\in\bdy\OM$. 
\end{definition}



\begin{theorem} \label{thm:glob_vis_iff_loc_vis}
Suppose that $\OM\subset\C^d$ is a hyperbolic domain that satisfies {\rm BSP} (in particular, it 
can be taken to be {\em any} bounded domain). Then 
$\OM$ is a (weak) visibility domain if and only if it is a local (weak) visibility domain.
\end{theorem}

\begin{remark}
{\rm BSP} above refers to the {\em boundary separation property} (see subsection~\ref{SS:MR} for a 
definition). It is a technical but essential assumption
in the {\em if} implication above (Result~\ref{lmm:refined_Roy_loc_lemm}, 
Lemma~\ref{lmm:kob_sep_intrnl_dom} and the results in 
subsection~\ref{SS:loc_implies_glob} show why the {\rm BSP} assumption is essential); note that it 
is automatically satisfied if $\OM$ is bounded. 
It is not needed for the {\em only if} implication. All of this will
become clear from the proof of Theorem~\ref{thm:glob_vis_iff_loc_vis} presented at the end of 
Section~\ref{S:gen_loc_glob}.
\end{remark}

\subsection{Visibility of planar domains and the continuous extension of conformal maps}
In this subsection, we present our results regarding the hyperbolic planar domains. 
These results and a few other are contained in Section~\ref{S:Visibility_Thm} of this article. 
The notion of visibility with respect to the Kobayashi distance was introduced 
mainly to study domains in higher-dimensional spaces. Not much is known in one 
dimension. (In this case, the Kobayashi distance coincides with the hyperbolic
distance, which comes from the hyperbolic metric.)
In this context, the following result by Bracci, Nikolov and Thomas 
\cite[Corollary~3.4]{BNT2022}, which is a corollary of a more general result
\cite[Theorem~3.3]{BNT2022}, is interesting. 

\begin{result}\label{R:simplyConn-BNT2022}
 Let $\Omega$ be a bounded simply connected domain in $\C$. Then $\Omega$ is a visibility domain if and only if $\bdy\Omega$ is locally connected. 
\end{result}

In this paper we generalize this result to any simply connected domain
by an application of Carath\'{e}odory extension theorem and Theorem~\ref{thm:cont_surj_im_vis_dom_vis}.
More precisely, we present:
\begin{theorem}[Corollary to {\Cref{thm:vis_simp_conn_loc_conn}}]
 Let $\Omega$ be a hyperbolic simply connected domain in $\C$. Then $\Omega$ 
 is a visibility domain if and only if $\bdy\Omega$ is locally connected. 
\end{theorem}
One other place where visibility for planar domains is considered is in the context of 
homeomorphic extension of biholomorphism between two planar domains by Bharali and Zimmer 
\cite[Theorem~1.10]{BZ2023}. Their approach to visibility is through the notion of locally 
Goldilocks domain. In this paper, in Section~\ref{Sec:Examples}, we construct a simply connected
domain $V_T$ in $\C$ that is 
not locally Goldilocks at {\em any} of its boundary points. 
This shows that the local Goldilocks property is considerably stronger than visibility.
More generally, we provide 
a much more general sufficient condition for the visibility of planar domains.

\begin{theorem}\label{T:Visibility_Thm_in_plane}
Let $\Omega \subset \C$ be a hyperbolic planar domain. Suppose
there exists a totally disconnected subset $S\subset\bdy\OM$ such that,
for every $p\in\bdy\OM\setminus S$, $\bdy\OM$ is 
locally connected at $p$ and the connected component of $\bdy\OM$ containing $p$ is not a singleton (i.e., is non-degenerate). Then $\OM$ is a visibility domain.
\end{theorem}

\noindent The proof of the above theorem is presented in Section~\ref{S:Visibility_Thm}. 
In this section, we introduce Condition~1 
and Condition~2 that a hyperbolic planar domain may satisfy. These conditions are defined analytically. One of 
the main results of this section is that any domain that satisfy
Condition~2 is a visibility domain; see Theorem~\ref{T:cond2ImplVisibility}.
Later we give a topological characterization of domains that satisfy Condition~2;
see Theorem~\ref{thm:loc_conn_apt_tot_disconn_cond2}.
Theorem~\ref{T:Visibility_Thm_in_plane} is a consequence of the aforementioned results. 
We conjecture that the 
converse of \Cref{T:Visibility_Thm_in_plane} is also true. We will give some evidence towards 
this.
\smallskip

We now discuss about continuous/homeomorphic extension of biholomorphic maps between planar 
domains. 
The question of the continuous extension of biholomorphisms 
between planar domains is an old one 
that goes back to the following result of Carath{\'e}odory. 

\begin{result}[Carath{\'e}odory; see, e.g., {\cite[Theorem~4.3.1]{BCM}}]\label{R:Cara}
Suppose that $\OM\subset\C_\infty$ is a hyperbolic simply connected domain and that 
$f:\unitdisk\to\OM$ is a biholomorphism. Then $f$ extends to a continuous map from 
$\clos{\unitdisk}$ to $\clos{\OM}$ if and only if $\bdy\OM$ is locally 
connected, where $\clos{\OM}$ and $\bdy\OM$ denote the closure and boundary of $\OM$,
respectively, in $\C_\infty$. The extension is a homeomorphism if and only if $\bdy\OM$ is a 
Jordan curve.
\end{result}

There are several generalizations of this result. The ones by Luo and Yao \cite{Luo-Yao2022} 
and Ntalampekos \cite{Ntalampekos2023} are quite general. We will state these results 
precisely in \Cref{sec:ext_biholo} while comparing them with our result in this direction:  

\begin{theorem}\label{thm:ext_biholo}
Let $\OM_1,\OM_2 \subset \C$ be hyperbolic domains and
 let $f :\OM_1 \lraw \OM_2$ be a 
biholomorphism. Suppose 
\begin{itemize}
    \item[(i)] there exists a totally disconnected set $S\subset\bdy \OM_2$ such that for all $p\in\bdy\OM_2\setminus S$ the boundary $\bdy\OM_2$ is locally connected at $p$ and the component of $\bdy\OM_2$ containing $p$ is not a singleton; and 
    \item[(ii)] every 
    component of $\bdy\OM_1$ is a Jordan curve and $\bdy\OM_1$ is locally connected. 
\end{itemize}
Then $f$ extends to a continuous
map from $\clos{\OM}_1^{End}$ onto $\clos{\OM}_2^{End}$. Moreover, if $\OM_2$ 
satisfies $(ii)$, then the extension is a homeomorphism.
\end{theorem}
\noindent We note that the conditions of \Cref{thm:ext_biholo} are 
purely topological and, as in \Cref{R:Cara}, local connectedness at the boundary points play a pivotal role.
\Cref{thm:ext_biholo} gives some classes of domains for which extension of biholomorphisms 
hold that can not be considered under the other known results. We will provide some 
examples in \Cref{Sec:Examples}. The proof of the above theorem is presented in \Cref{sec:ext_biholo}.
As mentioned earlier, the topological condition on $\OM_2$, as in $(i)$ above, is equivalent to the analytical 
condition Condition~2. Moreover the topological condition on $\OM_1$, as in $(ii)$ above, 
is equivalent to Condition~1, see Proposition~\ref{prp:cond1_jordanbdy} in Section~4. These analytical
conditions, together with Lemma~\ref{L:Unbounded_Gromov_prod}, help us to prove the above theorem. 

\section{Preliminaries}\label{Sec:Prelims}

In this section, we collect certain auxiliary results that are needed to prove the
main results of this article. We begin with a discussion about the end compactification.

\subsection{End compactification}\label{SS:endcpt}
In this article, we shall only consider the end compactifications of the closures of domains in Euclidean spaces. The notion of end compactification can be defined for
more general topological spaces, but we shall have no need to consider this.
\smallskip

Let $\OM$ be an unbounded domain in $\C^d$ (for bounded 
domains $\OM$, we have a natural compactification, namely the Euclidean
closure $\clos{\OM}$).
We choose and fix an exhaustion $(K_j)_{j\geq 1}$ of $\clos{\OM}$ such that, for all
$j$, $K_j\subset K^\circ_{j+1}$, where $K^\circ_{j+1}$ denotes the interior of $K_{j+1}$ in $\clos{\OM}$. We define an {\em end} $\mathfrak{e}$ of 
$\clos{\OM}$ to be 
a sequence $(F_j)_{j\geq 1}$ where, for every $j$, $F_j$ is a connected component of $\clos{\OM}\setminus K_j$ and where $F_{j+1}\subset F_j$. Note that 
if $\mathfrak{e}=(F_j)_{j\geq 1}$ is an end of $\clos{\OM}$ then, for every $j$, $F_j$ is necessarily an unbounded component of $\clos{\OM}\setminus K_j$. As
a {\em set}, $\clos{\OM}^{End}$ is defined to be 
$\clos{\OM}\cup\{\mathfrak{e}\mid\mathfrak{e}\text{ is an end of }\clos{\OM}\}$. We now define a topology on
$\clos{\OM}^{End}$ by prescribing a neighbourhood basis at each point of $\clos{\OM}^{End}$. If $x\in\clos{\OM}$, we take for a neighbourhood basis at $x$
all the ordinary Euclidean $\clos{\OM}$-neighbourhoods containing $x$; we take for a neighbourhood basis at an end $\mathfrak{e}=(F_j)_{j\geq 1}$ the family
$(\wh{F}_j)_{j\geq 1}$ of subsets of $\clos{\OM}^{End}$ where 
\[ \wh{F}_j \defeq F_j \cup \{ \mathfrak{f} \mid \mathfrak{f} = (G_\nu)_{\nu\geq 1} \text{ is an end of } \clos{\OM} \text{ such that } G_\nu=F_\nu \;
\forall\,\nu=1,\dots,j \}. \]
As we have defined it, both the
set $\clos{\OM}^{End}$ and the topology on it depend on the compact exhaustion $(K_j)_{j\geq 1}$ initially chosen; however, it is easy to see that if one
starts out with a different exhaustion $(L_j)_{j\geq 1}$ and proceeds to define the end compactification relative to this exhaustion then there exists a
natural homeomorphism between the two spaces thus obtained. Therefore, so far as topological questions are concerned (and we will deal only with such
questions), the end compactification is well defined, independent of the compact exhaustion chosen; we can work with whichever exhaustion is most 
convenient for us. At this point, we state a result that is 
very important in our future discussion. 

\begin{result}\label{res:end_seq_comp}
Suppose that $\OM\subset\C^d$ is an unbounded hyperbolic domain.
Suppose that, for every compact $K\subset\clos{\OM}$, there 
exists $R<\infty$ such
that $K\subset B(0;R)$ and such that there are only finitely many
connected components of $\clos{\OM}\setminus K$ that intersect 
$\C^d\setminus\clos{B}(0;R)$. Then $\clos{\OM}^{End}$ is
sequentially compact.	    
\end{result}

\noindent As stated, the above result has nothing to do with visibility 
and its proof is purely topological. Therefore we present the
proof of the above result in the appendix,
Section~\ref{S:end_seq_comp}, for completeness. Using 
the above result, we shall be able to conclude that for
a visibility domain $\OM$, $\clos{\OM}^{End}$ is squentially 
compact.

\subsection{Topological results}\label{SS:topo_res}
In this article, the word {\em path} will always denote a continuous map from some
(usually compact) interval in $\R$ into whatever topological space (usually a domain) 
is being considered. Given a domain $\OM\subset\C^d$ and a sequence of paths 
$\gamma_n:[a_n, b_n]\lraw\OM$, we shall say that {\em $(\gamma_n)_{n\geq 1}$ eventually 
avoids every compact set in $\OM$} if for any compact $K\subset\OM$, there exists $n(K)
\geq 1$, such that $\gamma_n([a_n, b_n])\cap K=\emptyset$ for 
all $n\geq n(K)$. 

\begin{lemma}\label{L:tr_1}
Let $\OM\subset\C^d$ be a domain and let  $p\neq q \in \bdy\clos{\OM}^{\text{End}}$. Let 
$\gamma_n:[a_n, b_n]\lraw\OM$ be a sequence of paths that eventually avoids every compact 
set in $\OM$ and such that $\gamma_n(a_n)\to p$ and $\gamma_n(b_n)\to q$ in the topology 
of $\clos{\OM}^{\text{End}}$. Then there exist a sequence $(s_n)_{n\geq 1}\subset[a_n, b_n]$, 
a point $\xi\in\bdy\OM$ and an $R>0$ such that $\gamma_n(s_n)\to\xi$ and such that each $\gamma_n$ 
intersects $\OM\setminus B(\xi, R)$.
\end{lemma}

\begin{proof}
We first observe that the case when one of the points $p,q$ is in $\bdy\OM$ is easy to
handle. First suppose that both $p$ and $q$ are in $\bdy\OM$. Then we can simply take
$s_n\defeq a_n$, $\xi\defeq p$ and $R\defeq \|p-q\|/2$. Next, suppose that exactly one
of $p,q$ is an end of $\clos{\OM}$. Suppose, without loss of generality, that $p\in\bdy\OM$
and that $q\in\clos{\OM}^{End}\setminus\clos{\OM}$. Then we can take $s_n\defeq a_n$, 
$\xi\defeq p$ and $R\defeq 1$.

So we may suppose that $p,q$ are two distinct ends. 
Choose a sequence $(K_j)_{j\geq 1}$ of compact subsets of $\clos{\OM}$ such that,
denoting by $K^{\circ}_j$ the interior of $K_j$ in $\clos{\OM}$, $K_j\subset
K^{\circ}_{j+1}$ for every $j$ and such that 
\[ \bigcup_{j=1}^{\infty} K^{\circ}_j=\clos{\OM}. \]
It is given that $(\gamma_n(a_n))_{n\geq 1}$ and $(\gamma_n(b_n))_{n\geq 1}$
converge to $p$ and $q$, respectively, in the topology of 
$\clos{\OM}^{\text{End}}$. By the definition of this topology, there
exist decreasing sequences $(U^{p}_j)_{j\geq 1}$ and $(U^{q}_j)_{j\geq 1}$ such that, 
for every $j$,
$U^p_j$ and $U^q_j$ are connected components of $\clos{\OM}\setminus K_j$, 
such that $p = (U^p_j)_{j\geq 1}$, $q = (U^q_j)_{j\geq 1}$.
Because $p\neq q$, there exists $j_0\in\posint$ such that $U^p_{j_0}\neq
U^q_{j_0}$. Then $U^p_{j_0}\cap U^q_{j_0}=\emptyset$ and, for all $j\geq j_0$, $U^p_j
\cap U^q_j = \emptyset$.
We may assume, without loss of generality, that for all $n$, $\gamma_n(a_n)\in
U^p_{j_0}$ and $\gamma_n(b_n)\in U^q_{j_0}$.
\smallskip

Now, for every $n$, $\gamma_n$ must intersect $K_{j_0}$; otherwise 
$\gamma_n([a_n, b_n])$, a connected subset of $\OM$, hence of $\clos{\OM}$, 
will intersect two distinct connected components of $\clos{\OM}\setminus K_{j_0}$, 
which will be a contradiction. Thus, for every $n$, there exists 
$s_n$, $a_n<s_n<b_n$, such that $\gamma_n(s_n)\in K_{j_0}$. Since $K_{j_0}$ is
compact, we may assume, without loss of generality, that 
$(\gamma_n(s_n))_{n\geq 1}$ converges to some point $\xi$ of $K_{j_0}$. Since
$(\gamma_n)_{n\geq 1}$ eventually avoids every compact set in $\OM$, it follows that
$\xi\in K_{j_0}\cap\bdy\OM$. Finally, as above, we can take $R\defeq 1$.

This finishes the consideration of all cases and completes the proof. 
\end{proof}

\subsection{Reparametrization of absolutely continuous curves with respect to the
 Kobayashi-Royden metric}\label{SS:Repar}
The first result in this subsection shows that any absolutely continuous curve that 
is almost everywhere {\em non-stationary} can be reparametrized so as to have 
unit speed with respect to the Kobayashi--Royden metric. 
\begin{proposition} \label{prp:drvtv_nwh_van_rprm_unt_spd}
Suppose that $\OM\subset\C^d$ is a hyperbolic domain and $\gamma: [a,b] \to \OM$ is 
an absolutely continuous map such that $\gamma'$ is almost everywhere non-vanishing. Then
there exists a $\dkoba_\OM$-unit-speed reparametrization
$\wt{\gamma}: [0,L] \to \OM$ of $\gamma$, i.e.,
\begin{equation}
\dkoba_\OM(\wt{\gamma}(t); \wt{\gamma}'(t))=1 \ \ \text{for a.e. $t\in[0, L]$}.
\end{equation}
\end{proposition}

\begin{proof}
Define the function $F:[a,b]\to[0,\infty)$ as follows:
\[
 F(s)\defeq \int_a^s \dkoba_\OM(\gamma(t);\gamma'(t))\,dt.
 \]
(The assumption that $\gamma$ is absolutely continuous implies that $F$
 is well-defined.) 
Write $L\defeq F(b)$; $L$ is the Kobayashi length of
$\gamma$ calculated using the Kobayashi metric. Note that $F$ itself is an absolutely
continuous function and
\begin{equation} \label{eqn:xpr_deriv_F}
 F'(t)=\dkoba_\OM(\gamma(t);\gamma'(t))\ \ \text{for a.e. } t\in [a,b].
  \end{equation} 
It also follows that $F$ is strictly increasing, for if
$a\leq s<t\leq b$,
\[ 
 F(t)-F(s) = \int_s^t \dkoba_\OM(\gamma(\tau);\gamma'(\tau))\,d\tau > 0,
  \]
since $\tau \mapsto \dkoba_\OM(\gamma(\tau);\gamma'(\tau))$ is a non-negative measurable
function that is positive almost everywhere. 
Therefore $F$ is a strictly increasing
(absolutely) continuous function that maps $[a,b]$ onto $[0,L]$. Let $G\defeq F^{-1}$; 
then $G$ is a strictly increasing continuous function from $[0,L]$ onto $[a,b]$. From
\eqref{eqn:xpr_deriv_F}, the hyperbolicity of $\OM$ and by hypothesis, $F'$ is 
almost everywhere non-vanishing. Consequently, by a standard result (see for instance
Exercise~13, Chapter~IX of \cite{Natanson}) $G$ is absolutely continuous. Consider the
mapping $\wt{\gamma}\defeq \gamma\circ G$. Since $\gamma$ is absolutely continuous and $G$
is strictly increasing and absolutely continuous, it follows that $\wt{\gamma}$ is
absolutely continuous and is a reparametrization of $\gamma$. We claim that it
is $\dkoba_\OM$-unit-speed. To prove this it suffices to show that, for a.e.~$t\in [0,L]$,
$\dkoba_\OM(\wt{\gamma}(t);\wt{\gamma}'(t))=1$. By the chain rule,
\begin{equation}
 \wt{\gamma}'(t) = \frac{1}{F'(G(t))}\gamma'(G(t)) =
  \frac{1}{\dkoba_\OM(\gamma(G(t));\gamma'(G(t)))}\gamma'(G(t)) \ \ \text{for a.e. } t\in [0,L].
\end{equation}
From the above it follows immediately that $\wt{\gamma}$ is $\dkoba_\OM$-unit-speed.
\end{proof}

In order to state our next result, we need a definition.
\begin{definition}[{\cite[Definition~3]{Nik_Okt_Tho}}] \label{D:geodesic}
Suppose that $\OM\subset\C^d$ is a hyperbolic domain, that $\lambda\geq 1$ and
$\kappa>0$ are parameters, and that $I\subset\R$ is an interval. An absolutely continuous curve $\gamma : I\to \OM$ is called a
{\em $(\lambda, \kappa)$-geodesic} for $\koba_\OM$ if for all $s\leq t\in I$ we have 
\[
l^{\dkoba}_\OM(\gamma|_{[s, t]})\leq \lambda\koba_\OM(\gamma(t), \gamma(s))+\kappa.
\]
Here $l^{\dkoba}_\OM(\gamma|_{[s, t]})$ denotes the length of the curve $\gamma|_{[s, t]}$
with respect to $\dkoba_\OM$. 
\end{definition}
\begin{remark} \label{rmk:geod_props}
Note that the definition of a $(\lambda, \kappa)$-geodesic is
devised in such a way that after allowable 
reparametrizations it remains a $(\lambda, \kappa)$-geodesic.
Lastly it is not difficult to see that a $(\lambda, \kappa)$-almost-geodesic
is a $(\lambda^2, \lambda^2\kappa)$-geodesic.
\end{remark}
\smallskip

Our next result says that any absolutely continuous
$(\lambda, \kappa)$-geodesic that is almost everywhere non-stationary can
be re-parametrized so as to become an almost-geodesic.
\begin{proposition} \label{prp:rprm_unt_spd_ag_ag}
Suppose that $\OM\subset\C^d$ is a hyperbolic domain 
and, given $\lambda\geq 1$ and $\kappa>0$, 
let $\gamma$ be a $(\lambda,\kappa)$-geodesic for 
$\koba_\OM$.
Suppose further that $\gamma'$ is almost 
everywhere non-vanishing. Let $\wt{\gamma}: [0,L]\to \OM$ be the $\dkoba_\OM$-unit-speed 
reparametrization of $\gamma$ as in the above proposition. Then $\wt{\gamma}$ is a 
$(\lambda,\kappa)$-almost-geodesic for $\koba_\OM$.
\end{proposition}

\begin{proof}
To show that $\wt{\gamma}$ is a $(\lambda,\kappa)$-almost-geodesic for $\koba_\OM$, 
we have to show the following: 
\begin{itemize}
\item for all $s,t\in [0,L]$, $\koba_\OM(\wt{\gamma}(s),\wt{\gamma}(t))\geq (1/\lambda)|s-t|-\kappa$, 
\smallskip

\item for all $s,t\in [0,L]$, $\koba_\OM(\wt{\gamma}(s),\wt{\gamma}(t))\leq \lambda|s-t|+\kappa$,
\smallskip 

\item for a.e.~$t\in [0,L]$, $\dkoba_\OM(\wt{\gamma}(t);\wt{\gamma}'(t))\leq\lambda$.
\end{itemize}

Since $\wt{\gamma}$ is $\dkoba_\OM$-unit-speed, for a.e.~$t\in [0,L]$, 
$\dkoba_\OM(\wt{\gamma}(t);\wt{\gamma}'(t))=1$ and hence the last property above
holds with stronger reason. Therefore, we also have, for all $s,t\in [0,L]$ with $s\leq t$, that
\begin{equation*}
 \koba_\OM(\wt{\gamma}(s);\wt{\gamma}(t)) \leq 
  \int_s^t \dkoba_\OM(\wt{\gamma}(\tau);\wt{\gamma}'(\tau)) d\tau = t-s \leq \lambda(t-s)+
   \kappa.
    \end{equation*}
Finally, for $s,t\in [0,L]$ with $s\leq t$ arbitrary,  
\begin{align*}
\koba_\OM(\wt{\gamma}(s),\wt{\gamma}(t)) &= \koba_\OM(\gamma(G(s)),\gamma(G(t))) \\
&\geq (1/\lambda)\ell^\dkoba_\OM(\gamma|_{[G(s),G(t)]})-(\kappa/\lambda), 
\end{align*}
since $\gamma$ is a $(\lambda,\kappa)$-geodesic for $\koba_\OM$.
Further, since the Kobayashi length of an
absolutely continuous path is independent of reparametrization,
\begin{align*}
\ell^\dkoba_\OM(\gamma|_{[G(s),G(t)]}) = \ell^\dkoba_\OM(\wt{\gamma}|_{[s,t]}) = t-s.
\end{align*}
Combining the two inequalities above, we get, for all $s,t\in [0,L]$ with $s\leq t$, that
\begin{equation*}
 \koba_\OM(\wt{\gamma}(s),\wt{\gamma}(t)) \geq
  (1/\lambda)(t-s)-(\kappa/\lambda) \geq (1/\lambda)(t-s)-\kappa.
   \end{equation*}
This establishes the result.
\end{proof}

The last result in this subsection says that an arbitrary $(\lambda, \kappa)$-geodesic can 
be approximated uniformly by another such curve that is non-stationary. More precisely: 

\begin{proposition} \label{prp:alm-geod_apprx_nonstnry}
Suppose that $\OM\subset\C^d$ is a hyperbolic domain and $x,y\in\OM$. Given $\lambda
\geq 1$ and $\kappa>0$, let $\gamma:[a,b]\to\OM$ be a 
$(\lambda,\kappa)$-geodesic joining $x$ and $y$ and let
$\epsilon>0$ be such that $Nbd_{\C^d}(\range(\gamma),\epsilon)\subset\OM$. Then there 
exists an 
absolutely continuous map $\wt{\gamma}:[a,b]\to\OM$
such that (1) $\wt{\gamma}$ is a $(\lambda,\lambda\epsilon+\kappa)$-geodesic,
(2) for every $t\in [a,b]$, $\|\gamma(t)-\wt{\gamma}(t)\|\leq 
\epsilon/2$, (3) $\wt{\gamma}'$ is almost-everywhere non-vanishing.
\end{proposition}

\begin{proof}
Let $E\defeq \{t\in [a,b]\mid \gamma'(t)=0\}$. Suppose $E$ has positive measure (otherwise 
the result is obvious). 
Let $\delta>0$ be an arbitrary number (whose values we should prescribe later)
and define $h$ almost everywhere on $[a,b]$ by
\[ h(t) \defeq \gamma'(t) + \delta \chi_E(t) \mathbf{e}, \]
where $\mathbf{e}\in\C^d$ is the vector
$(1,0,\dots,0)$ (any nonzero vector would work just as well). Clearly,
$h$ is integrable. Therefore, define
\[ \wt{\gamma}(s) \defeq \int_a^s h(t) dt, \ \ \forall\,s\in [a,b]. \]
Note that $\wt{\gamma}$ is absolutely continuous on $[a,b]$, differentiable almost everywhere
on $[a,b]$ and $\wt{\gamma}'(t) = h(t) = \gamma'(t)+\delta
\chi_E(t) \mathbf{e}$ outside a measure zero set $C$. Note that
\begin{equation} \label{eqn:gmm_tld_prm_non-van}
\wt{\gamma}'(t)\neq 0 \ \ \forall\, t\in [a,b]\setminus C. 
\end{equation}
(if $t\in [a,b]\setminus (C\cup E)$, $\wt{\gamma}'(t)=\gamma'(t)\neq 0$, whereas if
$t\in E\setminus C$, $\wt{\gamma}'(t)=\delta\mathbf{e}\neq 0$).
\smallskip

Now given $s\in [a,b]$ arbitrary, 
\begin{align*}
\wt{\gamma}(s)-\gamma(s) = \int_a^s h(t) dt - \int_a^s \gamma'(t) dt = \int_a^s \delta
\chi_E(t) \mathbf{e}\, dt,
\end{align*}
therefore, by choosing $\delta$ so small such that $3\delta\mathcal{L}^1(E)<\epsilon$, we get
\begin{equation} \label{eqn:gamma_tld_gamma_close}
\|\wt{\gamma}(s)-\gamma(s)\| \leq \delta \int_a^s \chi_E(t)\, dt
= \delta \mathcal{L}^1( E \cap [a,s]) \leq \epsilon/3, \ \forall\,s\in [a,b]. 
\end{equation}
Also for every $t\in [a,b]\setminus C$ we have 
\begin{equation} \label{eqn:gamma_tld_prime_gamma_prime_close}
\wt{\gamma}'(t)-\gamma'(t) = h(t)-\gamma'(t) = \delta 
\chi_E(t) \mathbf{e}, 
\implies \|\wt{\gamma}'(t)-\gamma'(t)\|\leq \delta \chi_E(t) \leq \delta.
\end{equation}
In particular,
\begin{equation} \label{eqn:nrm_gamma_tld_prm_sprmzd}
\|\wt{\gamma}'(t)\| \leq \|\gamma'(t)\|+\delta \ \ \forall\,t\in [a,b]\setminus C.
\end{equation}
We now make certain general observations. 
\begin{enumerate}
\item[$(a)$] Since $\koba_\OM$ is continuous on $\OM\times\OM$ (see, for example, 
\cite[Proposition~3.1.10]{JarPfl}), there exists $\delta_1>0$ such that 
\begin{align*}
|\koba_\OM(\gamma(s),\gamma(t))-\koba_\OM(z,w)|\leq \epsilon/3
\label{eqn:fund_kob_dists_apprx}
\end{align*} 
for all $s,t\in [a,b]$ and for all $z,w\in\OM$ such that $\|\gamma(s)-z\|\leq \delta_1 
\text{ and } \|\gamma(t)-w\|\leq \delta_1$.
\smallskip 

\item[$(b)$] Since $Nbd_{\C^d}\big(\range(\gamma),
\epsilon\big)\subset\OM$, it follows that 
$\clos{Nbd}_{\C^d}\big(\range(\gamma),\epsilon/2\big)\Subset\OM$. This implies that 
there exists $C<\infty$ such that 
\begin{equation*} \label{eqn:ub_kob_met_gamma}
\dkoba_\OM(z;v)\leq C\|v\| \ \ \forall z\in \clos{Nbd}_{\C^d}\big(\range(\gamma),\epsilon/2\big), \ \ \text{and} \ \ \forall v\in\C^d.
\end{equation*}
We may suppose, without loss of generality, that $\epsilon/(3C)<1$.
\smallskip

\item[$(c)$] Since $\gamma'$, 
defined a.e.~on $[a,b]$, is in $L^1([a,b],\C^d)$, there exists 
$\delta_2$, $0<\delta_2<\epsilon/(3C)$, such that,
\begin{equation*} \label{eqn:gamma_prime_integr_cond}
\forall \ \text{measurable} \ A \subset [a,b],\; \mathcal{L}^1(A)<\delta_2 \implies \int_A 
\|\gamma'(t)\| dt< \epsilon/(3C).
\end{equation*} 
\item[$(d)$] By Lusin's theorem, there exist a closed subset
$F\subset [a,b]$ and a continuous map $g:\R\to\C^d$ such that 
\begin{equation*} \label{eqn:Lusin_apprx}
\gamma'=g \text{ on } F \text{ and } \mathcal{L}^1([a,b]\setminus F)<\delta_2.
\end{equation*} 
As $F$ is a compact set and $g$ is a
continuous function on it, there exists $M<\infty$ such that
$\|g(t)\|\leq M$  for all $t\in F$. Since $\dkoba_\OM$ is upper semicontinuous on 
$\OM\times\C^d$ (see, for example, \cite[Proposition~3.5.13]{JarPfl}), there 
exists $\delta_3>0$ such that for all $t\in F$, all $z\in\OM$ such that
$\|z-\gamma(t)\|\leq\delta_3$, all $v\in\C^d$ with $\|v\|\leq M$, and all $w\in\C^d$ such
that $\|v-w\|\leq\delta_3$, $\dkoba_\OM(z;w)\leq\dkoba_\OM(\gamma(t);v)+(\epsilon/3(b-a))$.
In particular, for all $t\in F$, $z\in\OM$ such that $\|z-\gamma(t)\|\leq\delta_3$ and 
for all $v\in\C^d$ such that $\|v-\gamma'(t)\|\leq\delta_3$, we have
\begin{align*}
\dkoba_\OM(z;v) \leq \dkoba_\OM(\gamma(t);\gamma'(t))+(\epsilon/3(b-a)).
\label{eqn:fund_kob_met_apprx}
\end{align*} 
\end{enumerate}

Let $\delta\defeq\min\{1,\delta_1,\delta_2,\delta_3,\epsilon/3\}/\mathcal{L}^1(E)$ then the last inequality 
together with \eqref{eqn:gamma_tld_gamma_close} and
\eqref{eqn:gamma_tld_prime_gamma_prime_close} implies 
\begin{equation*}
\dkoba_\OM(\wt{\gamma}(t);\wt{\gamma}'(t)) \leq
\dkoba_\OM(\gamma(t);\gamma'(t))+(\epsilon/3(b-a)) \ \ \forall\,t\in F\setminus C.
\end{equation*}
This inequality implies, for all $s,t\in [a,b]$ with $s\leq t$, we have
\begin{align}
\ell^{\dkoba}_\OM(\wt{\gamma}|_{[s,t]}) &= \int_s^t \dkoba_\OM(\wt{\gamma}(\tau);
\wt{\gamma}'(\tau))\,d\tau \notag\\
&= \int_{[s,t]\cap F} 
\dkoba_\OM(\wt{\gamma}(\tau);\wt{\gamma}'(\tau))\,d\tau + \int_{[s,t]\setminus F}
\dkoba_\OM(\wt{\gamma}(\tau);\wt{\gamma}'(\tau))\,d\tau \notag\\
&\leq \int_{[s,t]\cap F} \dkoba_\OM(\gamma(\tau);\gamma'(\tau))\,d\tau +
\int_{[s,t]\cap F} (\epsilon/3(b-a))\,d\tau + \int_{[s,t]\setminus F} 
\dkoba_\OM(\wt{\gamma}(\tau);\wt{\gamma}'(t))\,d\tau\label{eqn:xpr_kob_length_to_be_mnmzd},
\end{align}
Note
\begin{equation*}
\int_{[s,t]\cap F}\dkoba_\OM(\gamma(\tau);\gamma'(\tau))\,d\tau \leq \int_s^t 
\dkoba_\OM(\gamma(\tau);\gamma'(\tau))\,d\tau = \ell^{\dkoba}_\OM(\gamma|_{[s,t]}) \leq
\lambda \koba_\OM(\gamma(s),\gamma(t))+\kappa,
\end{equation*}
where we used the fact that $\gamma$ is a $(\lambda,\kappa)$-almost-geodesic with respect
to $\koba_\OM$. By \eqref{eqn:gamma_tld_gamma_close}, and the inequality in item~$(a)$ above implies that
\begin{equation*}
\koba_\OM(\gamma(s),\gamma(t)) \leq \koba_\OM(\wt{\gamma}(s),\wt{\gamma}(t))+(\epsilon/3).
\end{equation*}
Therefore
\begin{equation} \label{eqn:ub_intrg_kob_met_gamma_mn_prt}
\int_{[s,t]\cap F}\dkoba_\OM(\gamma(\tau);\gamma'(\tau))\,d\tau \leq \lambda
\koba_\OM(\wt{\gamma}(s),\wt{\gamma}(t)) + (\lambda\epsilon/3) + \kappa.
\end{equation}
Obviously,
\begin{equation} \label{eqn:ub_trvl_intgr}
\int_{[s,t]\cap F} (\epsilon/3(b-a))\,d\tau \leq (\epsilon/3(b-a))(b-a) = \epsilon/3.
\end{equation} 
Finally
\begin{equation*} \label{eqn:bound_intgr_kob_met_rem_part_prelim}
\int_{[s,t]\setminus F} \dkoba_\OM(\wt{\gamma}(\tau);\wt{\gamma}'(t))\,d\tau \leq 
\int_{[s,t]\setminus F} C\|\wt{\gamma}'(\tau)\|\,d\tau \leq C \int_{[a,b]\setminus F} 
\|\wt{\gamma}'(\tau)\|\,d\tau,
\end{equation*}
where, to write the first inequality, we have used
\eqref{eqn:gamma_tld_gamma_close} and the inequality in item~$(b)$. Now
\begin{equation*} \label{eqn:bound_intgr_nrm_gmm_tld_prm_rem_part}
\int_{[a,b]\setminus F} \|\wt{\gamma}'(\tau)\|\,d\tau \leq \int_{[a,b]\setminus F}\big(
\|\gamma'(\tau)\|+\delta \big)\,d\tau \leq \frac{\epsilon}{3C} + \delta 
\mathcal{L}^1([a,b]\setminus F) \leq \frac{\epsilon}{3C}+\frac{\epsilon}{3C} = 
\frac{2\epsilon}{3C},
\end{equation*}
where, to write the first inequality, we have used \eqref{eqn:nrm_gamma_tld_prm_sprmzd};
to write the second, we have used the inequality in item~$(c)$; and to write the
third and fourth we have used the fact that $\delta\leq\delta_2<\epsilon/(3C)<1$. The last two inequality then imply
that
\begin{equation} \label{eqn:ub_intrg_kob_met_gmm_tld_rem_prt_fin}
\int_{[s,t]\setminus F} \dkoba_\OM(\wt{\gamma}(\tau);\wt{\gamma}'(t))\,d\tau \leq 
\frac{2\epsilon}{3}.
\end{equation}
Using \eqref{eqn:ub_intrg_kob_met_gamma_mn_prt}, \eqref{eqn:ub_trvl_intgr} and
\eqref{eqn:ub_intrg_kob_met_gmm_tld_rem_prt_fin} in \eqref{eqn:xpr_kob_length_to_be_mnmzd},
\begin{equation} \label{eqn:ineq_gmm_tld_ag}
\ell^{\dkoba}_\OM(\wt{\gamma}|_{[s,t]}) \leq \lambda
\koba_\OM(\wt{\gamma}(s),\wt{\gamma}(t)) + (\lambda\epsilon/3) + \kappa + (2\epsilon/3)
\leq \lambda\koba_\OM(\wt{\gamma}(s),\wt{\gamma}(t)) + \lambda\epsilon + \kappa.
\end{equation}
From \eqref{eqn:gmm_tld_prm_non-van}, \eqref{eqn:gamma_tld_gamma_close} and 
\eqref{eqn:ineq_gmm_tld_ag}, we see that the proposition has been proved.
\end{proof}

\subsection{Separation results with respect to the Kobayashi distance}\label{SS:MR}

In this subsection, we collect certain separation results with respect to the Kobayashi distance.
Given a hyperbolic domain $\OM\subset \C^d$ and two distinct points $p,q\in \bdy\clos{\OM}^{End}$, 
we shall say that $p, q$ satisfy the 
{\em boundary separation property} ${\rm BSP}$ with respect to $\koba_\OM$
if the following holds:
\begin{equation}\label{E:strong_sep_bound}
  \liminf_{(z,w)\to (p,q),\,z,w\in\OM}\koba_\OM(z,w)>0.
\end{equation}
It is not difficult to see that if the inequality above holds for 
all $p,q\in\bdy\OM$,
$p\neq q$ (i.e., for every pair of distinct ordinary boundary 
points), then it holds for all
$p,q\in\bdy\clos{\OM}^{End}$, $p\neq q$ (i.e., it extends 
automatically to pairs of distinct ends). 
Bearing this in mind, we say that {\em $\OM$ satisfies 
${\rm BSP}$ with respect to $\koba_\OM$} if, for all 
$p,q\in\bdy\OM$ with $p\neq q$,
\eqref{E:strong_sep_bound} holds.
Our first lemma says that a domain possesses this property
under a natural condition on the Kobayashi--Royden metric.

\begin{lemma} \label{lmm:if_kob_met_large_kob_sep}
Let $\OM\subset\C^d$ be a hyperbolic domain such that for every point $p\in\bdy\OM$ there exist a 
neighbourhood $U$ of $p$ and a constant $c>0$ 
such that $\dkoba_\OM(z;v)\geq c\|v\|$ for all $z\in U\cap\OM$ and $v\in\C^d$. 
Then $\OM$ satisfies ${\rm BSP}$ with respect to $\koba_\OM$.
\end{lemma}
\noindent The proof is not difficult and follows simply from the fact that the Kobayashi distance
is the integrated form of $\dkoba_\OM$. Every bounded domain satisfies the condition in the 
above lemma (see, e.g., Proposition~3.5 in \cite{Bharali_Zimmer}) and therefore every bounded domain satisfies 
${\rm BSP}$. 
For a planar hyperbolic domain $\OM\subset\C$, we know that 
$\lim_{z\to p,\,z\in\OM}\dkoba_\OM(z)=+\infty$ for any $p\in\bdy\OM$;
from this it follows that $\OM$ satisfies ${\rm BSP}$. We state this as a result to be used later:
\begin{result} \label{res:hyp_plan_dom_BSP}
	Every hyperbolic planar domain $\OM$ satisfies ${\rm BSP}$ with respect to $\koba_\OM$.
\end{result}
${\rm BSP}$ is intimately related to the notion of
{\em hyperbolicity at a boundary point} as presented, for example, in \cite{Nik_Okt_Tho}
(see \cite[Definition~1]{Nik_Okt_Tho}). In Section~\ref{S:gen_loc_glob}, we shall also 
see that a form of weak visibility for a hyperbolic
domain implies that it satisfies ${\rm BSP}$.
\smallskip

We shall also need a separation result of the type where two compact sets $K\subset L$ of a hyperbolic domain 
$\OM$ are given and we wish to conclude that $\koba_\OM(K,\OM\setminus L)>0$. The simplest case is given
by the following lemma from which the general case also follows. 
\begin{lemma}\label{lmm:Kob_dist_sep_ball}
	Suppose that $\OM\subset\C^d$ is a hyperbolic
	domain. Let $p\in\OM$ and $R>0$ be such that 
	$\overline{B(p,R)}\subset\OM$. Then there exists a $c>0$ such that
	$\koba_\OM(p,w)\geq c$ for all $w\in\OM\setminus\overline{B(p,R)}$.
\end{lemma}

\begin{proof}
	Since $\OM$ is hyperbolic, Theorem~2 in 
	\cite{RoyKobDist} implies that there exist $r,c'>0$ such that
	\begin{equation} \label{eqn:lb_Kob_met_ball}
		\dkoba_{\OM}(z,v)\geq c'\|v\| \ \ \ \forall\,z\in B(p,r)\;\text{and}\;\forall\,v\in\C^d.
	\end{equation} 
	Let $\rho=\min\{R/2, r\}$. Choose $w\in\OM\setminus\overline{B(p,R)}$ and let
	$\gamma:[0,1]\to\OM$ be {\em any} $\smoo^1$ path joining $p$
	and $w$. Writing
	\[ s\defeq \inf\{t\in[0,1]\mid \gamma(t)\notin B(p, \rho)\}, \] 
	it follows that $\gamma(s)\in\bdy B(p,\rho)$ and that 
	$\gamma\big([0,s)\big)\subset B(p,\rho)$. Then
	\begin{align*}
		l_{\OM}(\gamma) &\geq \int_0^s 
		\dkoba_{\OM}(\gamma(t),\gamma'(t))dt \\
		&\geq \int_0^s c'\|\gamma'(t)\|dt 
		\quad (\text{from \eqref{eqn:lb_Kob_met_ball}}) \\
		&\geq c'\|p-\gamma(s)\|=c'\rho \defines c.
	\end{align*}
	Since $\gamma$ was arbitrary, therefore, taking the infimum over
	all $\smoo^1$ paths $\gamma$ joining $p$ and $w$, it follows that 
	$\koba_\OM(p,w)\geq c$ .
\end{proof}

\begin{corollary}\label{crl:Kob_dist_sep_comp}
	Suppose that $\OM\subset\C^d$ is a hyperbolic
	domain. Let $K\subset\OM$ and $r>0$ be such that both $K$ and
	$K_r\defeq \{z\in\C^d\mid \distance_{\mathrm{Euc}}(z,K)\leq r\}$ are compact 
	subsets of
	$\OM$. Then $\koba_\OM(K,\OM\setminus K_r)>0$.
\end{corollary}

\subsection{Localization results for the Kobayashi metric}\label{SS:loc_res}
The following result due to Royden \cite[Lemma~2]{RoyKobDist} is an important tool in 
studying the localization properties of the Kobayashi distance and in studying the relation 
between local and global visibility and Gromov hyperbolicity (see, e.g., \cite{ADS2023, 
BGNT2022, BNT2022}). Its proof 
can also be found in \cite[Lemma~4]{Graham1975}. 

\begin{result}[Royden's Localization Lemma] \label{lmm:Roy_loc_lmm}
Suppose that $X$ is a hyperbolic complex manifold and that $U\subset X$ is a 
domain. Then for any $z\in U$ and $v\in T_z^{1, 0}X$, we have
\[ \dkoba_U(z;v) \leq\coth\big( \koba_X(z,X\setminus U) \big) \dkoba_X(z;v), \]
where $\koba_X(z,X\setminus U)\defeq \inf_{w\in X\setminus U}\koba_X(z,w)$. 
\end{result}

\begin{remark} \label{rmk:weak_loc_almost-geod}
To provide an illustration of how the result above can be used
to study the relation between global and local visibility, we remark the following: If
$\OM\subset\C^d$ is a hyperbolic domain, if $U\subset\C^d$ is a sub-domain, if 
$W\subset U$ is a set such that $\koba_\OM(W,\OM\setminus U)>0$ and if we write
$\lambda_0\defeq \coth\big( \koba_\OM(W,\OM\setminus U) \big)$, then, for every 
$\lambda\geq 1$ and every $\kappa\geq 0$, every $(\lambda,\kappa)$-almost-geodesic
$\gamma$ for $\koba_\OM$ that is contained in $W$ is a 
$(\lambda_0\lambda,\kappa)$-almost-geodesic for $\koba_U$. This follows immediately
from the definition of almost-geodesics by using \Cref{lmm:Roy_loc_lmm} and will
be used in Section~\ref{sec:ext_biholo}.
\end{remark}

The following lemma, recently presented and proved in \cite{ADS2023}, is a refinement of 
the one above, and is crucial in proving localization results for the Kobayashi distance, as 
in \cite{ADS2023}. Its usefulness will also become apparent in this paper.

\begin{result}[{\cite[Lemma~3.1]{ADS2023}}] \label{lmm:refined_Roy_loc_lemm}
Suppose that $\OM\subset\C^d$ is a hyperbolic domain and $U\subset\OM$ is a sub-domain. 
If $W\subset U$ is an open set with $\koba_\OM(W,\OM\setminus U)>0$, then there exists $L<\infty$ such that
\begin{equation*}
\dkoba_U(z;v) \leq (1 + L e^{-\koba_\OM(z,\OM\setminus U)})\dkoba_\OM(z;v) \ \ \forall\,z\in W,\,v\in\C^d. 
\end{equation*} 
\end{result}

\subsection{Preliminary results regarding visibility}\label{SS:prelim_vis}
The first result of this subsection says that, under the assumption of weak visibility, the Gromov 
product of points, converging to two distinct points of the boundary of the end compactification, 
remains finite in a strong sense. This behaviour is analogous to the one shown by pairs of distinct 
points of the ideal boundary of a Gromov hyperbolic distance space. It was first proved in 
\cite[Proposition~2.4]{BNT2022} and mildly generalized in \cite[Proposition~3.1]{CMS2023}. 
This is a further generalization.

\begin{lemma} \label{lmm:weak_visib_Grom_lim_fin}
	Suppose $\Omega \subset \C^d$ is a hyperbolic domain and that $\xi_1,\xi_2\in
	\bdy\clos{\OM}^{\text{End}}$, $\xi_1\neq\xi_2$, satisfy the visibility property
	with respect to $(1,\epsilon)$-almost-geodesics for some $\epsilon>0$. Then 
	\[
	\limsup_{(z,w)\to (\xi_1,\xi_2),\;z,w\,\in\,\OM}(z|w)_o < +\infty.
	\]
\end{lemma}

\begin{proof}
	This is a mild modification of the proof of \cite[Proposition~2.4]{BNT2022} 
	for the end compactification case, but, for the sake of completeness, we provide a
	full argument here. Note that, in order to prove the required result, it suffices
	to prove the following: if $(z_n)_{n\geq 1}$ and $(w_n)_{n\geq 1}$ are {\em any} 
	two sequences in $\OM$ converging to $\xi_1$ and $\xi_2$, respectively, then
	\begin{equation*}
		\limsup_{n \to \infty} (z_n|w_n)_o < +\infty.
	\end{equation*}
	So let $(z_n)_{n\geq 1}$ and $(w_n)_{n\geq 1}$ be such sequences. For every $n$,
	choose a $(1,\epsilon)$-almost-geodesic $\gamma_n$ for $\koba_{\OM}$ joining $z_n$
	and $w_n$. Note that, by the assumption that the pair $\{\xi_1,\xi_2\}$ possesses
	the visibility property, there 
	exists a compact subset $K$ of $\OM$ such that, for every $n$, $\gamma_n\cap 
	K\neq\emptyset$. Choose, for every $n$, $x_n\in\gamma_n\cap K$. Since $z_n,x_n$ 
	and $w_n$ lie on a $(1,\epsilon)$-almost-geodesic, for all $n$, 
	\begin{align*}
		\koba_{\OM}(z_n,w_n) &\geq \koba_{\OM}(z_n,x_n) + \koba_{\OM}(x_n,w_n) - 3\epsilon
		\\
		&\geq \koba_{\OM}(z_n,o)+\koba_{\OM}(w_n,o)-2\koba_{\OM}(x_n,o)-3\epsilon.
	\end{align*}
	From this it follows that 
	\[ (z_n|w_n)_o\leq \koba_{\OM}(x_n,o)+2\epsilon. \] 
	Since $x_n\in K$, it follows that there exists $L<\infty$ such that, for every 
	$n$, the right-hand side above is $\leq L$. This, of course, shows that
	\begin{equation*}
		\limsup_{n \to \infty} (z_n|w_n)_o < \infty
	\end{equation*}
	and completes the proof.
\end{proof}

        The last lemma of this section is about the convergence of a 
        sequence of geodesics to a geodesic ray in visibility domains. 
        The proof of this lemma is substantially similar to that of
	\cite[Proposition~5.4]{BZ2023} and its main ideas are also to be found in
	\cite[Lemma~3.1]{BNT2022}. The latter result, in fact, is 
	precisely this lemma stated for {\em bounded} complete hyperbolic visibility
	domains. We present a proof of this result here since the hypotheses and conclusion of
	this lemma are slightly different from those of either of the two results
	quoted.	

\begin{lemma} \label{lmm:seq_geods_conv_geod_ray}
	Suppose $\OM\subset\C^d$ is a complete hyperbolic domain that possesses 
	the geodesic  visibility property. Suppose $p\in \bdy\clos{\OM}^{End}$, that 
	$(q_n)_{n\geq 1}$ is a sequence of points of $\OM$ converging to $p$, that
	$o\in\OM$ is a fixed point, and that, for every $n$, $\gamma_n$ is a
	$\koba_{\OM}$-geodesic joining $o$ and $q_n$. Then, up to a subsequence,
	$(\gamma_n)_{n\geq 1}$ converges locally uniformly on $[0,\infty)$ to a
	geodesic ray that emanates from $o$ and lands at $p$.		
\end{lemma}

\begin{proof}	
	Let $\gamma_n : [0,L_n] \to \OM$; then $L_n = \koba_{\OM}(o,q_n)$. Since
	$(\OM,\koba_{\OM})$ is by assumption complete, $\koba_{\OM}(o,q_n)\to\infty$
	as $n\to\infty$. Therefore, we may
	suppose, without loss of generality, that $(L_n)_{n\geq 1}$ is a strictly
	increasing sequence of positive real numbers tending to $\infty$. Extend each
	$\gamma_n$ to all of $\R$ by defining $\gamma_n|_{[L_n,\infty]}$ to be 
	constantly equal to $q_n$; call the extended mapping $\wt{\gamma}_n$. Now 
	consider the sequence of mappings $(\wt{\gamma}_n)_{n\geq 1}$ between the 
	distance spaces $([0,\infty),|\cdot|)$ and $(\OM,\koba_{\OM})$. It is now easy
	to see that one may invoke the Arzel{\`a}--Ascoli theorem to conclude that 
	$(\wt{\gamma}_n)_{n\geq 1}$ has a subsequence that converges locally uniformly
	on $[0,\infty)$ to a continuous mapping $\wt{\gamma}$ from $[0,\infty)$ to
	$\OM$. It also follows immediately that $\wt{\gamma}$ is a 
	$\koba_{\OM}$-geodesic. 
	\smallskip

	Now we will show that $\wt{\gamma}$ lands at $p$. First we show that 
	$\wt{\gamma}$ lands at some point of $\bdy\clos{\OM}^{End}$ (actually, our
	argument will show that {\em any} $\koba_{\OM}$-geodesic ray lands at some 
	point of $\bdy\clos{\OM}^{End}$). Consider the set
	\[S\defeq \clos{\wt{\gamma}([0,\infty))} \setminus \wt{\gamma}([0,\infty)),\]
	where we take the closure in $\clos{\OM}^{End}$.
	By the properness of the distance space $(\OM,\koba_{\OM})$, $S\subset
	\bdy\clos{\OM}^{End}$. By the compactness of $\clos{\OM}^{End}$, $S\neq 
	\emptyset$. We need to show that $S$ is a singleton. Assume, to get a
	contradiction, that $S$ contains at least two points, say $\xi$ and $\zeta$,
	$\xi\neq\zeta$. 
	Since $\xi\in S$, there exists a
	sequence $(s_n)_{n\geq 1}$ such that $s_n \nearrow \infty$ and such that 
	$\wt{\gamma}(s_n)\to\xi$ as $n\to\infty$. Similarly, we get a sequence 
	$(t_n)_{n\geq 1}$ corresponding to $\zeta$. By passing to subsequences, we may
	assume that, for all $n$, $s_n<t_n<s_{n+1}$. Now consider the sequence
	$\big(\wt{\gamma}|_{[s_n,t_n]}\big)_{n\geq 1}$ of $\koba_{\OM}$-geodesics. The
	sequences of its end-points converge to $\xi$ and $\zeta$, respectively, 
	which are distinct. Therefore, by the visibility assumption, there exists a
	compact subset $K$ of $\OM$ such that, for every $n$, 
	$\wt{\gamma}|_{[s_n,t_n]}\cap K\neq\emptyset$. Therefore, for every $n$, there
	exists $u_n\in [s_n,t_n]$ such that $\wt{\gamma}(u_n)\in K$. Consequently,
	there exists $M<\infty$ such that, for every $n$, 
	$\koba_{\OM}(\wt{\gamma}(0),\wt{\gamma}(u_n))\leq M$. But, at the same time,
	$\koba_{\OM}(\wt{\gamma}(0),\wt{\gamma}(u_n))=u_n\geq s_n$ and, since 
	$s_n\to\infty$ as $n\to\infty$, we have a contradiction. Therefore, our 
	starting assumption must be wrong, and it must be that $S$ is a singleton.
	
	Now we will show that $S=\{p\}$; it is enough to assume that $p\notin S$ and
	obtain a contradiction. So we assume $p\notin S$. Since we know $S$ is a
	singleton, let $S\defines \{\xi\}$. Then our assumption means that $p\neq\xi$.
	Choose $(U_j)_{j\geq 1}$ and $(V_j)_{j\geq 1}$ representing $p$ and $\xi$
        respectively as ends.
	Since $\lim_{t\to\infty} \wt{\gamma}(t)=\xi$, for every $j$, there exists $s_j\in
	[0,\infty)$ such that, for all $t\geq s_j$, $\wt{\gamma}(t)\in V_j$. We may
	also suppose that $s_j\to\infty$. Now
	$(\wt{\gamma}_n(s_j))_{n\geq 1}$ converges to $\wt{\gamma}(s_j)$. Therefore,
	there exists $n_j\in\posint$ such that, for all $n\geq n_j$, 
	$\wt{\gamma}_n(s_j)\in V_j$. We may also take $n_j$ to be so large that 
	$L_{n_j}>s_j$, so that, for all $n\geq n_j$, 
	$\wt{\gamma}_n(s_j)=\gamma_n(s_j)$.
	So, finally, for all $j$, $\gamma_{n_j}(s_j)\in V_j$. But, at the same time,
	$(\gamma_{n_j}(L_{n_j}))_{j\geq 1}$ converges to $p$. Therefore, for every 
	$k\in\posint$, there exists $j_k\in\posint$ such that for all $j\geq j_k$,
	$\gamma_{n_j}(L_{n_j})\in U_k$. In particular, for every $k$,
	$\gamma_{n_{j_k}}(L_{n_{j_k}}) \in U_k$. Now, for every $k$, consider the
	geodesic $\gamma_{n_{j_k}}|_{[s_{j_k},L_{n_{j_k}}]}$. The sequences of its end 
	points, $(\gamma_{n_{j_k}}(s_{j_k}))_{k\geq 1}$ and 
	$(\gamma_{n_{j_k}}(L_{n_{j_k}}))_{k\geq 1}$, converge to the distinct points 
	$\xi$ and $p$ of $\bdy\clos{\OM}^{End}$, respectively. By visibility, there 
	exist $t_k \in [s_{j_k},L_{n_{j_k}}]$ and $M<\infty$ such that, for all $k$,
	$\koba_{\OM}(\gamma_{n_{j_k}}(t_k),o)\leq M$. But 
	$\koba_{\OM}(\gamma_{n_{j_k}}(t_k),o)=
	\koba_{\OM}(\gamma_{n_{j_k}}(t_k),\gamma_{n_{j_k}}(0))=t_k\geq s_{j_k}$, which
	is a contradiction because $s_{j_k}\to\infty$. This contradiction shows that
	$p\in S$, hence that $S=\{p\}$, and we are done.		
\end{proof}

\section{Some general results about visibility}\label{S:gen_loc_glob}

\subsection{Removability of a totally disconnected set and 
other topological consequences of visibility property}
In this subsection, we present certain results that show that if a domain $\Omega$
satisfies either the visibility property or the weak visibility property then it must 
satisfy certain topological conditions. In this direction, our first result is
Theorem~\ref{thm-totally disconnected} whose proof we present here.

\begin{proof}[Proof of {\Cref{thm-totally disconnected}}]
	Assume, to get a contradiction, that $\OM$ does not possess the visibility property with 
	respect to $(\lambda,\kappa)$-almost-geodesics for $\koba_\OM$. Then we know that there exist
	$p,q\in\bdy\clos{\OM}^{End}$, $p\neq q$, and a sequence $(\gamma_n)_{n\geq 1}$ of 
	$(\lambda,\kappa)$-almost-geodesics with respect to $\koba_\OM$, $\gamma_n: [a'_n,b'_n]\to\OM$,
	such that $\gamma_n(a'_n)\to p$, $\gamma_n(b'_n)\to q$ and such that $(\gamma_n)_{n\geq 1}$ eventually
	avoids every compact set in $\OM$. By \Cref{L:tr_1}, there exist $\xi\in\bdy\OM$, $R>0$ and a
	sequence $(a_n)_{n\geq 1}\subset [a'_n,b'_n]$ such that $\gamma_n(a_n)\to\xi$ and such that every
	$\gamma_n$ intersects $\OM\setminus B(\xi;R)$. We may suppose, without loss of generality, that,
	for every $n$, $\gamma_n([a_n,b'_n])\cap(\OM\setminus B(\xi;R))\neq\emptyset$.
	For every $n$, write
	\begin{align*}
		b_n &\defeq \inf\{t>a_n\mid \gamma_n(t)\notin B(\xi;R)\}. 
	\end{align*}
	Note that, for every $n$, there exist $c_n,d_n$, $a_n<c_n<d_n<b_n$, such that 
	$\gamma_n\big((a_n,c_n)\big)\subset B(\xi;R/3)$, $\gamma_n(c_n)\in\bdy B(\xi;R/3)$, $\gamma_n(d_n)
	\in\bdy B(\xi;2R/3)$ and $\gamma_n\big((d_n,b_n)\big)\subset B(\xi;R)\setminus \clos{B}(\xi;2R/3)$.
	Note that, by passing to subsequences successively, we may assume that 
	$\big(\gamma_n\big([a_n,c_n]\big)\big)_{n\geq 1}$ converges in the Hausdorff distance to some
	compact, connected set $L_1\subset\clos{B}(\xi;R/3)$ and that
	$\big(\gamma_n\big([d_n,b_n]\big)\big)_{n\geq 1}$ converges in the Hausdorff distance to some
	compact, connected set $L_2\subset\clos{B}(\xi;R)\setminus B(\xi;2R/3)$ (see the discussion following
	\Cref{dfn:vis-wvis-dom}).
	Clearly, $L_1$ contains
	$p$ and {\em some} point of $\bdy B(\xi;R/3)$ and $L_2$ contains {\em some} point of $\bdy B(\xi;2R/3)$ 
	and {\em some} point of $\bdy B(\xi;R)$. By the assumption that $(\gamma_n)_{n\geq 1}$
	eventually avoids every compact subset of $\OM$, it follows that $L_1\subset\clos{B}(\xi;R/3)
	\cap\bdy\OM$ and that $L_2\subset\big(\clos{B}(\xi;R)\setminus B(\xi;2R/3)\big)\cap\bdy\OM$. $L_1$
	and $L_2$ are closed, connected subsets of $\bdy\OM$ each of which contains at least 2 points.
	Therefore, neither of them can be included in $S$, i.e., there exist points $q_1\in L_1
	\setminus S$ and $q_2\in L_2\setminus S$. Since $q_1$ and $q_2$ are in $L_1$ and $L_2$,
	respectively, there exist sequences $c'_n\in [a_n,c_n]$ and $d'_n\in [d_n,b_n]$ such that
	$\gamma_n(c'_n)\to q_1$ and $\gamma_n(d'_n)\to q_2$, respectively. Now consider the sequence
	$\gamma_n|_{[c'_n,d'_n]}$ of $(\lambda,\kappa)$-almost-geodesics. Its end points converge,
	respectively, to $q_1\in \bdy\OM\setminus S$ and $q_2\in\bdy\OM\setminus S$; and $q_1,q_2$ are
	distinct. Thus, by the assumed visibility property, there exists a compact subset of $\OM$ that 
	all the $\gamma_n|_{[c'_n,d'_n]}$ intersect. But this is an immediate contradiction to the 
	assumption that $(\gamma_n)_{n\geq 1}$ eventually avoids every compact set. This contradiction 
	proves the result.
\end{proof}

\begin{remark} \label{rmk:cnvg_outside_tot_disc}
	We extract from the preceding proof an important fact that we shall use later. Namely, if
	$\OM\subset\C^d$ is a hyperbolic domain, if $S\subset\bdy\OM$ is a totally disconnected set,
	and if $\gamma_n : [a_n,b_n] \to \OM$ is a sequence of paths whose end points $\gamma_n(a_n)$
	and $\gamma_n(b_n)$ converge to distinct points, say $p$ and $q$, of $\bdy\clos{\OM}^{End}$, 
	then there exist $p_0\in\bdy\OM\setminus S$, $R>0$ and a sequence $t_n$ of parameter values such that
	$\gamma_n(t_n)$ converges to $p_0$ and such that each $\gamma_n$ intersects $\OM\setminus B(p_0;R)$.
	We may also suppose that $p_0\neq p,q$.
\end{remark}

As mentioned in Subsection~\ref{SS:endcpt}, 
we now present the proof of the
result that for a visibility domain, its end compactification is 
sequentially compact. First, we need the following lemma. 
\begin{lemma} \label{lemma:inf_big_comp_not_weak_visib}
Suppose that $\OM\subset\C^d$ is an unbounded hyperbolic domain. 
Given a compact subset $K\subset\clos{\OM}$  and an $R<\infty$ such that
$K\subset B(0;R)$, suppose that infinitely many connected components of $\clos{\OM}\setminus K$ intersect $\C^d\setminus\clos{B}(0;R)$. 
Then, for every $\kappa>0$, there exist two distinct boundary points of $\OM$ that do not satisfy the visibility property with respect to 
$(1,\kappa)$-almost-geodesics for $\koba_\OM$. In particular, $\OM$ is not a weak visibility domain.
\end{lemma}

\begin{proof}
To begin with, choose $R'<R$ such that $K\subset B(0;R')$.
\smallskip
		
\noindent {\bf Claim~1.} $\OM\setminus\clos{B}(0;R')$ has infinitely many components that intersect $\C^d\setminus\clos{B}(0;R)$.
\smallskip
	
\noindent Assume, to get a contradiction, that there are only finitely many components of 
$\OM\setminus\clos{B}(0;R')$ that intersect
$\C^d\setminus\clos{B}(0;R)$, say $W_1,\dots,W_n$
(there must be at least one, otherwise $\OM$ will be bounded).
Note $\overline{W_i}$ is connected and is contained in
$\overline{\OM}\setminus K$. Therefore, for each $i$, 
$\overline{W_i}\subset F$ for some connected component $F$ of 
$\overline{\OM}\setminus K$ that intersects
$\C^d\setminus\overline{B}(0, R)$.
On the other hand, given $F$ a connected component of $\overline{\OM}\setminus K$ that intersects
$\C^d\setminus\overline{B(0, R)}$, clearly $F$ intersects $\overline{\OM}\setminus\overline{B}(0, R)
\subset\overline{\OM\setminus\overline{B}(0, R)}$. 
Note that
\begin{equation*}
 \OM\setminus\clos{B}(0;R) = \bigcup_{i=1}^n W_i\setminus\clos{B}(0;R) \ \ \  \implies \ \ \
   \clos{\OM\setminus\clos{B}(0;R)} = \clos{\bigcup_{i=1}^n W_i\setminus\clos{B}(0;R)} \subset\bigcup_{i=1}^n \clos{W}_i.
     \end{equation*}
As $F$ intersects $\overline{\OM\setminus\overline{B}(0, R)}$, $F$ clearly intersects some
$\overline{W_i}$. But then $\overline{W_i}\subset F$, i.e., each connected component $F$ of 
$\overline{\OM}\setminus K$ that intersects $\C^d\setminus\overline{B(0, R)}$ contains some $W_i$. As the 
collection of such $F$'s is infinite, and $W_i$'s are finite this leads to contradiction whence the claim.\hfill $\blacktriangleleft$
\smallskip
	
	Using Claim~1, choose a sequence $(W_n)_{n\geq 1}$ of pairwise distinct connected components of $\OM\setminus\clos{B}(0;R')$ that intersect
	$\C^d\setminus\clos{B}(0;R)$. Choose and fix a point $o\in\OM\cap B(0;R')$ and, for every $n$, choose $z_n\in W_n\setminus\clos{B}(0;R)$. Also, for every $n$, choose a $(1,\kappa)$-almost-geodesic $\gamma_n$ for
	$\koba_\OM$ that joins $o$ with $z_n$. Suppose that every $\gamma_n$ is defined on $[0,1]$; then, for every $n$, there exists $\wh{t}_n\in (0,1)$ 
	such that $\gamma_n((\wh{t}_n,1])\subset\OM\setminus\clos{B}(0;R')$. Therefore, every $\gamma_n((\wh{t}_n,1])$ is included in some connected
	component of $\OM\setminus\clos{B}(0;R')$; since $\gamma_n(1)=z_n\in W_n$, it follows that $\gamma_n((\wh{t}_n,1])\subset W_n$. If we choose and fix 
	$R_1\in (R',R)$, it is clear that each $\gamma_n$ intersects $\bdy B(0;R_1)$. Indeed,
	assuming that each $\gamma_n$ is defined on $[0,1]$, there exist, for each $n$, $s_n,t_n\in [0,1]$ with $0<s_n<t_n<1$ such that $\gamma_n([s_n,t_n])
	\subset \clos{B}(0;R)\setminus B(0;R_1)$, such that $\gamma_n(s_n)\in \bdy B(0;R_1)$ and $\gamma_n(t_n)\in \bdy B(0;R)$. By the compactness of 
	$\clos{B}(0;R)$, there exists a subsequence of $(\gamma_n([s_n,t_n]))_{n\geq 1}$,
	which we will denote without changing subscripts, that converges in the Hausdorff distance to a compact, connected subset $L$ of 
	$\clos{\OM \cap (\clos{B}(0;R)\setminus B(0;R_1))}$.
	\smallskip
	
	Clearly, $L\subset\clos{\OM}$. First suppose that $L\cap\OM\neq\emptyset$. This means that there is a subsequence $(k_n)_{n\geq 1}\subset\posint$ and
	a corresponding sequence of parameter values $u_{k_n}\in [s_{k_n},t_{k_n}]$ such that $\gamma_{k_n}(u_{k_n})\to x$, where $x \in \OM \cap
	(\clos{B}(0;R)\setminus B(0;R_1))$. Now, $x$ is in some connected component $W$ of $\OM\setminus\clos{B}(0;R')$; this $W$ is an open set, and
	$(\gamma_{k_n}(u_{k_n}))_{n\geq 1}$ is a sequence converging to $x$; therefore, for large $n$, say for $n\geq N$, $\gamma_{k_n}(u_{k_n})\in W$. 
	In particular, $\gamma_{k_N}(u_{k_N}), \gamma_{k_{N+1}}(u_{k_{N+1}})\in W$. But $\gamma_{k_N}(u_{k_N})\in W_N$ and $\gamma_{k_{N+1}}(u_{k_{N+1}})\in 
	W_{N+1}$. Therefore the {\em distinct} connected components $W_N$ and $W_{N+1}$ of $\OM\setminus\clos{B}(0;R')$ both intersect the connected component
	$W$ of $\OM\setminus\clos{B}(0;R')$, which is a contradiction.
	\smallskip
	
	Thus it cannot be that $L\cap\OM\neq\emptyset$; hence it must be that $L\subset\bdy\OM$. Furthermore, since $\gamma_n(s_n)\in\bdy B(0;R_1)$ and
	$\gamma_n(t_n)\in\bdy B(0;R)$, we may assume, by passing to a subsequence and relabelling, that $(\gamma_n(s_n))_{n\geq 1}$ converges to a point
	$p\in\bdy B(0;R_1)\cap\bdy\OM$ and that $(\gamma_n(t_n))_{n\geq 1}$ converges to a point $q\in\bdy B(0;R)\cap\bdy\OM$. Since $L\subset\bdy\OM$, it
	follows that $\left(\gamma_n|_{[s_n,t_n]}\right)_{n\geq 1}$, which is a sequence of $(1,\kappa)$-almost-geodesics with respect to $\koba_\OM$, avoids
	every compact subset of $\OM$. Moreover, as we just remarked, the sequences of its endpoints converge to two distinct boundary points of $\OM$.
	Consequently, $\OM$ does not have the weak visibility property, which completes the proof.	 
\end{proof}

\begin{theorem}\label{thm:vis_seqcompact}
Suppose that $\OM\subset\C^d$ is an unbounded hyperbolic domain that satisifes the visibility property with respect to $(1,\kappa)$-almost-geodesics
for some $\kappa>0$. Then the topological space $\clos{\OM}^{End}$is sequentially compact.    
\end{theorem}
\begin{proof}
Let $K\subset\overline{\OM}$ be a compact set and let 
$K\subset B(0, R)$. Then it follows from the above lemma 
that there are only finitely many connected components of 
$\overline{\OM}\setminus K$ that intersect
$\C^d\setminus\overline{B}(0, R)$. Therefore by
Result~\ref{res:end_seq_comp}
$\clos{\OM}^{End}$is sequentially compact.
\end{proof}

We shall now present another topological consequence of
visibility property. We begin with a lemma. 
\begin{lemma} \label{prp:vsb_nbd_ntrsc_fin}
Suppose $\OM\subset\C^d$ is a weak visibility domain, and given $p\in
\bdy\OM$, let $U$ be a neighbourhood of $p$ in $\C^d$. Then, given a sequence 
$(x_n)_{n\geq 1}$ in $U\cap\OM$
converging to $p$, the set 
\[ F_{p,\,U,\,(x_n)_{n\geq 1}} \defeq \big\{V\mid V \text{ is a connected component of } 
U\cap\OM \text{ and } \exists\, n\in\posint \text{ such that } x_n\in V \big\} \]
is finite.
\end{lemma}

\begin{proof}
Assume that $F_{p,\,U,\,(x_n)_{n\geq 1}}$ is infinite; then, in particular, $U\cap\OM$ 
has infinitely many connected components.  Relabelling, we may assume, without loss of 
generality, that$(V_j)_{j\geq 1}$ is a sequence of {\em distinct} connected components 
of $U\cap\OM$ such that $x_j\in V_j$ and $(x_j)_{j\geq 1}$ converges to $p$. Let $W$ 
be a neighbourhood of $p$ such that $W\Subset U$ and $\OM\setminus\clos{W}\neq\emptyset$.
Note that, for large $j$, $x_j\in W$. Fix a point $o\in\OM\setminus\clos{W}$ and choose 
a $(1,1)$-almost-geodesic $\gamma_j: [a_j,b_j]\to\OM$ for $\koba_\OM$ joining $o$ and 
$x_j$. Note, for large $j$, $\gamma_j(a_j)\in\OM\setminus\clos{W}$ and $\gamma_j(b_j)
\in W$, therefore, there exists $t_j\in[a_j,b_j]$ such that $\gamma_j(t_j)\in \bdy W_j$ 
and $\gamma_j((t_j,b_j]) \subset W$. By the compactness of $\bdy W$, we may assume, 
without loss of generality, that $(\gamma_j(t_j))_{j\geq 1}$ converges to some point 
$q\in\bdy W$. 
\smallskip

\noindent{\bf Case~1.} $q\in\OM$. 
\smallskip

\noindent Note that $q\in U\cap\OM$, and therefore, $q$ must belong to some connected 
component of $U\cap\OM$, say $\widehat{V}$. 
For every $j$,
$\gamma_j([t_j,b_j])$ is a connected subset of $U\cap\OM$, and consequently, it is
contained in a connected component of $U\cap\OM$. Since $\gamma_j(b_j)=x_j
\in V_j$, it follows that $\gamma_j([t_j,b_j])\subset V_j$. Since $q\in
\widehat{V}$ and $(\gamma_j(t_j))_{j\geq 1}$ converges to $q$, therefore, for
all $j$ large enough, $\gamma_j(t_j)\in\widehat{V}$. But then, for {\em all}
such $j$, $\gamma_j(t_j)$ is contained in {\em both} the connected components
$\widehat{V}$ and $V_j$ of $U\cap\OM$. In particular, this means (since distinct 
components are disjoint) that all the components $V_j$ for large enough $j$
coincide, which is a contradiction. 
\smallskip

\noindent{\bf Case~2.} $q\in\bdy\OM$. 
\smallskip

\noindent Note that $q\in \bdy W \cap \bdy\OM
\subset U\cap\bdy\OM$ and
$\left(\gamma_j|_{[t_j,b_j]}\right)_{j\geq 1}$ is a sequence of 
$(1,1)$-almost-geodesics for $\koba_\OM$ the sequences of whose end-points,
$(\gamma_j(t_j))_{j\geq 1}$ and $(x_j)_{j\geq 1}$, converge to $q$ and $p$,
respectively, which are {\em distinct} points of $\bdy\OM$. Therefore, by the
assumed weak visibility property, there exists a sequence $(s_j)_{j\geq 1}$,
$t_j\leq s_j\leq b_j$, such that $\{\gamma_j(s_j)\mid j\geq 1\}$ is relatively
compact in $\OM$ (hence relatively compact in $U\cap\OM$). Consequently, we may
suppose, without loss of generality, that $(\gamma_j(s_j))_{j\geq 1}$ converges
to some point $u\in U\cap\OM$. Now we appeal to the argument as in Case~1 above
to obtain a contradiction, whence the result. 
\end{proof}

\begin{theorem}\label{Th:vis_bound_reg}
Suppose $\OM\subset\C^d$ is a weak visibility domain, and given $p\in
\bdy\OM$, let $U$ be a neighbourhood of $p$ in $\C^d$. 
Then there exists a connected component $V$ of $U\cap\OM$ such that $p\in\bdy V$;
moreover, the number of such components are finite. 
\end{theorem}

\begin{proof}
If $U\cap\OM$ has finitely many components, then the proof is trivial. So we assume
that $U\cap\OM$ has infinitely many components.
Choose a sequence $(x_n)_{n\geq 1}$ of points of $U\cap\OM$ converging to $p$.
Let $(V_j)_{j\geq 1}$ be an enumeration of the connected components of 
$U\cap\OM$. By \Cref{prp:vsb_nbd_ntrsc_fin}, 
$F_{p,U,(x_n)_{n\geq 1}}$ is finite, which in this case means that
\[ \{j\in\posint\mid \exists\,n\in\posint \text{ such that } x_n\in V_j\} \]
is finite. Thus there exists $j_0\in\posint$ such that $x_n\in V_{j_0}$ for infinitely
many $n$. Clearly, then, $V_{j_0}$ is a connected component of $U\cap\OM$ that has
$p$ as a boundary point. 
\smallskip

Now suppose there are infinitely many components $V$ of $U\cap\OM$
such that $p\in \bdy V$. Then there exists a sequence $(V_j)_{j\geq 1}$ of
distinct components of $U\cap\OM$ such that, for each $j$, $p\in\bdy V_j$. For
each $j$, choose $x_j\in V_j$ such that $\|x_j-p\|<1/j$. Then $(x_j)_{j\geq 1}$
is a sequence in $U\cap\OM$ converging to $p$ that contradicts
\Cref{prp:vsb_nbd_ntrsc_fin}, whence the result. 
\end{proof}

We remark that it is still possible that there exists a neighbourhood
basis at $p$, the intersection of each of whose elements with $\OM$ has 
infinitely many components; see \Cref{xmp:nbd_bss_intrsc_infnt}.

\subsection{Proof of \Cref{thm:cont_surj_im_vis_dom_vis}}
\label{SS:cont_surj_im_vis_dom_vis}

Now we move on to provide a proof of \Cref{thm:cont_surj_im_vis_dom_vis}. We shall first prove a lemma.

\begin{lemma} \label{lmm:dsj_cpt_sets_visib_gen}
Suppose $\OM\subset\C^d$ is a hyperbolic
domain that is a $(\lambda,\kappa)$-visibility domain for some $\lambda\geq 1$, $\kappa>0$.
Then, given two disjoint compact subsets 
$H_1$, $H_2$ of $\bdy\clos{\OM}^{End}$, there exists a compact subset $K$ of $\OM$ such that 
the following is satisfied. 
\begin{itemize}
\item[$(*)$] Given a sequence $\gamma_n:[a_n,b_n]\to\OM$ of $(\lambda,\kappa)$-almost-geodesics with respect to 
$\koba_{\OM}$ such that all the limit points of 
$(\gamma_n(a_n))_{n\geq 1}$ (resp. $(\gamma_n(b_n))_{n\geq 1}$) lie in $H_1$ and
$H_2$ respectively, there exists $n_0\in \posint$ such that
$\gamma_n\cap K\neq\emptyset$ for every $n\geq n_0$.
\end{itemize}
\end{lemma}

\begin{proof}
Assume that there does not exist any compact set
$K\subset\OM$ satisfying the stated property. Choose a compact exhaustion
$(K_{\nu})_{\nu\geq 1}$ of $\OM$. Then, for each fixed $\nu$, 
there exists a sequence $\gamma^\nu_n:[a^\nu_n,b^\nu_n]\to\OM$ of
$(\lambda,\kappa)$-almost-geodesics such
that all the limit points of $\big(\gamma^\nu_n(a^\nu_n)\big)_{n\geq 1}$
(resp. $\big(\gamma^\nu_n(b^\nu_n)\big)_{n\geq 1}$) lie in $H_1$ (resp. 
$H_2$) and (without loss of generality) such that every $\gamma^\nu_n$ avoids
$K_{\nu}$. By the compactness of $\clos{\OM}^{End}$, 
$\big(\gamma^\nu_n(a^\nu_n)\big)_{n\geq 1}$ and 
$\big(\gamma^\nu_n(b^\nu_n)\big)_{n\geq 1}$ both have limit points in
$\clos{\OM}^{End}$ that lie, by hypothesis, in $H_1$ and $H_2$,
respectively. Passing to a subsequence and relabelling, we may assume that
$\big(\gamma^\nu_n(a^\nu_n)\big)_{n\geq 1}$ and
$\big(\gamma^\nu_n(b^\nu_n)\big)_{n\geq 1}$ converge to points $p_{\nu}$ of
$H_1$ and $q_{\nu}$ of $H_2$, respectively. By the compactness of $H_1$ and $H_2$
we may assume, by passing to a subsequence and relabelling, that
$(p_{\nu})_{\nu\geq 1}$ and $(q_{\nu})_{\nu\geq 1}$ converge to points $\check{p}
\in H_1$ and $\check{q}\in H_2$, respectively. Choose countable neighbourhood bases
$(U_j)_{j\geq 1}$ and $(V_j)_{j\geq 1}$ for the topology of $\clos{\OM}^{End}$ at $\check{p}$ and
$\check{q}$, respectively. For every $j$, choose $\nu(j)\in\posint$ such that, for every 
$\nu\geq\nu(j)$, $p_{\nu}\in U_j$ and $q_{\nu}\in V_j$; in particular,
$p_{\nu(j)}\in U_j$ and $q_{\nu(j)} \in V_j$. Since $U_j$ (resp., $V_j$) is an open subset
of $\clos{\OM}^{End}$ and $p_{\nu(j)}$ (resp., $q_{\nu(j)}$) is a point of $\OM$ that is in 
$U_j$ (resp., $V_j$), and since, for every $n$, $(\gamma^{\nu(j)}_n(a^{\nu(j)}_n))_{n\geq 1}$
(resp., $(\gamma^{\nu(j)}_n(b^{\nu(j)}_n))_{n\geq 1}$) converges to $p_{\nu(j)}$ (resp.,
$q_{\nu(j)}$), therefore, for every $j$, we may choose $n(j)\in\posint$ so large that, for
every $n\geq n(j)$, $\gamma^{\nu(j)}_n(a^{\nu(j)}_n)\in U_j$ (resp.,
$\gamma^{\nu(j)}_n(b^{\nu(j)}_n)\in V_j$); in particular,
$\gamma^{\nu(j)}_{n(j)}(a^{\nu(j)}_{n(j)})\in U_j$ (resp., 
$\gamma^{\nu(j)}_{n(j)}(b^{\nu(j)}_{n(j)}) \in V_j$).
Now simply rename 
$\gamma^{\nu(j)}_{n(j)}$ as $\theta_j$, $a^{\nu(j)}_{n(j)}$ as $c_j$ and $b^{\nu(j)}_{n(j)}$
as $d_j$. Then $\theta_j : [c_j,d_j] \to \OM$ is a sequence of 
$(\lambda,\kappa)$-almost-geodesics for $\koba_\OM$ whose sequences of end points converge
to the {\em distinct} points $\check{p}$ and $\check{q}$ of $\bdy\clos{\OM}^{End}$. 
Consequently, by the $(\lambda,\kappa)$-visibility property of $\OM$, there exists a compact
subset of $\OM$ that all the $\theta_j$ intersect. This clearly is a contradiction to our assumption. 
\end{proof}
		
                         \begin{proof}[Proof of \Cref{thm:cont_surj_im_vis_dom_vis}]
			Choose $p,q\in\bdy\clos{\OM}^{\text{End}}$ with $p\neq q$ and $\lambda\geq 1$
			and $\kappa>0$. Consider $S_p\defeq \Phi^{-1}\{p\}$ and $S_q\defeq
			\Phi^{-1}\{q\}$. Then $S_p$ and $S_q$ are disjoint, compact subsets of $\bdy\clos{\OM}^{End}_0$.
			By Lemma~\ref{lmm:dsj_cpt_sets_visib_gen}, there exists a compact subset $K$ of $\OM_0$
			such that the property $(*)$ therein is satisfied. 
			Now suppose $(x_n)_{n\geq 1}$ 
			and $(y_n)_{n\geq 1}$ are sequences in $\OM$ converging to $p$ and $q$, 
			respectively, and $\gamma_n: [a_n,b_n]\to \OM$ is a sequence of $(\lambda,\kappa)$-almost-geodesics with 
			respect to $\koba_{\OM}$ joining $x_n$ to $y_n$. Now let $\check{\gamma}_n\defeq 
			\Phi^{-1}\circ\gamma_n$. Then, since $\Phi$ is a biholomorphism, each 
			$\check{\gamma}_n$ is a $(\lambda,\kappa)$-almost-geodesic with respect to 
			$\koba_{\OM_0}$. Since $\Phi$ extends continuously to $\clos{\OM}^{End}_0$ and 
			$\gamma_n(a_n)=x_n\to p$ and $\gamma_n(b_n)=y_n\to q$, it
			follows that all 
			the limit points of $(\check{\gamma}_n(a_n))_{n\geq 1}$ 
			(resp. $(\check{\gamma}_n(b_n))_{n\geq 1}$) lie in $S_p$ (resp. $S_q$). Therefore,
			by property $(*)$ in Lemma~\ref{lmm:dsj_cpt_sets_visib_gen}, there exists $n_0\in\posint$ such that, for all
			$n\geq n_0$, $\check{\gamma}_n\cap K\neq\emptyset$. Consequently
			\[
			 \Phi\big( \check{\gamma}_n\cap K \big) \neq \emptyset \ \ \forall\,n\geq n_0
			 \]
			which implies 
			\[
			 \gamma_n\cap\Phi(K) \neq \emptyset \ \ \forall\,n\geq n_0. 
			\] 
			By the continuity of $\Phi$, $\Phi(K)$ is a compact subset of $\OM$. Since $p, q$ and $\lambda\geq 1, \kappa>0$ were
			chosen arbitrarily, the result follows.
\end{proof}

\subsection{On Boundary separation property ${\rm BSP}$}\label{SS:BSP}

Now we prove that if a pair of points of the boundary of the end-compactification
satisfies (a property weaker than) the weak visibility property, then they satisfy
 ${\rm BSP}$ (as defined in subsection~\ref{SS:MR}) 
relative to the Kobayashi distance.
\begin{lemma} \label{lmm:w_visib_kob_sep_bdy_pts}
	Suppose that $\Omega \subset \C^d$ is a hyperbolic domain and that $\xi_1,\xi_2\in
	\bdy\clos{\OM}^{\text{End}}$, $\xi_1\neq\xi_2$, satisfy the visibility property
	with respect to $(1,\epsilon)$-almost-geodesics for some $\epsilon>0$. Then
	$\xi_1,\xi_2$ satisfy ${\rm BSP}$.	
\end{lemma}

\begin{proof}
	To get a contradiction assume that there exist sequences $(z_n)_{n \geq 1}$ and 
	$(w_n)_{n \geq 1}$ with $z_n \to \xi_1$ and $w_n \to \xi_2$ as $n \to \infty$ and 
	$\koba_{\Omega} (z_n, w _n ) \leq 1/n$ for all $n \in \posint$. For every $n$ so 
	large that $1/n<\epsilon$, choose a $(1,1/n)$-almost-geodesic $\gamma_n:[a_n,b_n]
	\to \OM$ joining $z_n$ and $w_n$. By the assumption that $\xi_1$ and $\xi_2$ 
	satisfy the visibility property with respect to $(1,\epsilon)$-almost-geodesics,
	there exists a compact subset $K \subset \Omega$ such that, for every $n$, 
	$\gamma_n \cap K  \neq \emptyset$. Choose, for every $n$, $x_n\in \gamma_n\cap K$. 
	Since $z_n,x_n,w_n$ lie on the $(1,1/n)$-almost-geodesic $\gamma_n$ and since, by
	assumption, $\koba_{\OM}(z_n,w_n)\leq 1/n$, we have, as before,
	\begin{align*}
		1/n \geq \koba_{\OM}(z_n,w_n) &\geq \koba_{\OM}(z_n,x_n)+\koba_{\OM}(x_n,w_n)-3/n\\
		&\geq \koba_{\OM}(z_n,x_n)-3/n.
	\end{align*}
	Thus $\koba_{\OM}(x_n,z_n)\leq 4/n$ and so 
	$\lim_{n \to \infty}\koba_{\OM}(x_n,z_n)=0$. But this is a contradiction to
	Corollary~\ref{crl:Kob_dist_sep_comp} because as $z_n$ converges either to a 
	boundary point or to an end, it eventually escapes from all compact subsets of 
	$\OM$.
\end{proof}

The boundary separation property ${\rm BSP}$, at the level of points, leads
to a similar property at the level of sets. This latter property is essential for all 
localization results in this article. 

\begin{lemma} \label{lmm:kob_sep_intrnl_dom}
Suppose that $\OM\subset\C^d$ is a hyperbolic domain and $U\subset\OM$ is a sub-domain 
such that given two distinct points $\xi_1$ and $\xi_2$ of $\bdy U\cap\bdy\OM$ we have 
\[
  \liminf_{(z,w)\to (\xi_1,\xi_2),\,z,w\in\OM} \koba_{\OM}(z,w)>0. 
  \]
Let $W\subset U$ be a non-empty open subset such that
$W\Subset (U\cup\bdy\OM)\setminus \clos{\bdy U\cap\OM}$.
Then $\koba_\OM(W,\OM\setminus U)>0$.
\end{lemma}

\begin{proof}
Assume, to get a contradiction, that $\koba_\OM(W,\OM\setminus U)=0$. 
Using the assumption that $W\Subset (U\cup\bdy \OM) \setminus \clos{\bdy U\cap\OM}$,
choose $r>0$ such that $\text{Nbd}(\clos{W};r)\cap\clos{\bdy U\cap\OM}=\emptyset$. Put
$W_1 \defeq \{z\in U\mid d_{\text{Euc}}(z,W)<r/2\}$. Then $W_1$ is a non-empty open 
subset of $U$, $W_1 \Subset (U \cup \bdy\OM)\setminus\clos{\bdy U\cap\OM}$, and
\begin{equation} \label{eqn:rel_btw_W_W1}
	W \Subset (W_1\cup\bdy\OM)\setminus\clos{\bdy W_1\cap\OM}.
\end{equation}	
By hypothesis, for every
$n\in\posint$, there exist $x_n\in W$ and $y_n\in \OM\setminus U$ such that
$\koba_\OM(x_n,y_n)<1/n$. For every $n$, choose a $(1,1/n)$-almost-geodesic $\gamma_n:
[a_n,b_n]\to\OM$ for $\koba_\OM$ joining $x_n$ and $y_n$. Consider 
\[ 
 t_n\defeq \sup\{t\in[a_n,b_n]\mid\gamma_n([a_n,t])\subset W_1\} \ \ \text{for every $n$}.
 \] 
It is easy to see 
that, for every $n$, $a_n<t_n<b_n$, $\gamma_n([a_n,t_n))\subset W_1$, and 
$\gamma_n(t_n)\in \bdy W_1 \cap \OM$. We may assume, without loss of generality, that
$(x_n)_{n\geq 1}$ converges to a point $\xi\in \clos{W}$ and $(\gamma_n(t_n))_{n\geq 1}$
converges to a point $\zeta\in \bdy W_1$. In case either $\xi$ or $\zeta$ is in $\OM$, we
use Corollary~\ref{crl:Kob_dist_sep_comp} to get a contradiction.
So we may assume that $\xi,\zeta\in\bdy\OM$. In this
case, it follows from \eqref{eqn:rel_btw_W_W1} that $\xi\neq\zeta$. Also by our hypothesis,
there exists $c>0$ such that
\[
 \koba_\OM(x_n,\gamma_n(t_n))\geq c \ \ \text{for every $n$}.
 \]
So $\xi$ and $\zeta$ are two distinct points of $\bdy U \cap \bdy \OM$, 
$(\gamma_n(a_n))_{n\geq 1}$ and $(\gamma_n(t_n))_{n\geq 1}$ converge to $\xi$ and $\zeta$,
respectively, and $\left( \gamma_n|_{[a_n,t_n]} \right)_{n\geq 1}$ is a sequence of
$(1,1/n)$-almost-geodesics for $\koba_\OM$. 
Therefore
\begin{align*}
\koba_\OM(x_n,y_n) = \koba_\OM(\gamma_n(a_n),\gamma_n(b_n)) &\geq \koba_\OM(\gamma_n(a_n),\gamma_n(t_n)) +
\koba_\OM(\gamma_n(t_n),\gamma_n(b_n))-(3/n) \\
&\geq \koba_\OM(\gamma_n(a_n),\gamma_n(t_n))-(3/n) \geq c-(3/n).
\end{align*}
So, for all $n$ large enough, we have $\koba_\OM(x_n,y_n)\geq (c/2)$. This is an 
immediate contradiction, which proves the required result.
\end{proof}

At the end of this subsection, we state and prove a result about complete hyperbolic visibility 
domains which will be used in the proof of \Cref{thm:ext_biholo} and is of independent 
interest.
	
\begin{lemma} \label{lmm:diff_ends_geod_rays_Haus_infty}
	Suppose that $\OM\subset\C^d$ is a complete hyperbolic 
	visibility 
	domain, that $\alpha$ and $\beta$ are distinct ends of $\clos{\OM}$ and 
	that 
	$\sigma$ and $\gamma$ are $\koba_{\Omega}$-geodesic rays in $\OM$ that land at 
	$\alpha$ and $\beta$, respectively (i.e., $\lim_{t\to\infty}\sigma(t)=\alpha$ 
	and $\lim_{t \to \infty}\gamma(t)=\beta$ in the end-compactification 
	topology). Then $d^{\koba_{\Omega}}_H(\sigma,\gamma)=\infty$. 
\end{lemma}

\begin{proof}
	We proceed by contradiction. Assume, to get a contradiction, that
	$d^{\koba_{\Omega}}_H(\sigma,\gamma)<\infty$. Then, by definition, there exists
	$M<\infty$ such that, for every $n\in\posint$, there exists $y_n\in\gamma$ such
	that $\koba_{\Omega}(\sigma(n),y_n)<M$. For each $n$, choose a 
	$\koba_{\Omega}$-geodesic $\theta_n:[0,T_n]\to\OM$ such that $\theta_n(0)
	= \sigma(n)$ and $\theta_n(T_n)=y_n$. Then $(\theta_n)_{n\geq 1}$ is a sequence
	of $\koba_{\Omega}$-geodesics the sequences of whose end points converge to the
	distinct ends $\alpha$ and $\beta$, respectively, of $\clos{\OM}$. By our 
	assumption that $\OM$ is a visibility domain, there exists a compact 
	$K\subset\OM$ such that, for every $n$, $\theta_n\cap K\neq \emptyset$. For
	every $n$, choose $z_n\in\theta_n\cap K$. Since $(z_n)_{n\geq 1}$ is trapped in
	$K$, since $\sigma(n)\to \infty$ as $n\to\infty$ (i.e., $\lim_{t \to \infty} 
	\|\sigma(t)\|=\infty$) and since $\OM$ is complete hyperbolic, it follows that
	$\lim_{n \to \infty}\koba_{\Omega}(\sigma(n),z_n)=\infty$. By the fact that 
	each $\theta_n$ is a geodesic, it follows that, for each $n$, 
	$\koba_{\Omega}(\sigma(n),y_n) = \koba_{\Omega}(\sigma(n),z_n)+
	\koba_{\Omega}(z_n,y_n)$. So, since $\koba_{\Omega}(\sigma(n),z_n)\to\infty$ as
	$n\to\infty$, it follows that $\koba_{\Omega}(\sigma(n),y_n)\to\infty$ as
	$n\to\infty$. But this is an immediate contradiction.
\end{proof}

\subsection{Local visibility implies global visibility}\label{SS:loc_implies_glob}

In this subsection, we shall prove that, given domains $U\subset\OM$, under 
certain conditions on $U$, visibility with respect to 
$\koba_U$ implies visibility with respect to $\koba_\OM$. The precise result is stated in 
Theorem~\ref{thm:loc_vis_imp_glob_vis_ntrnlzd} below, which follows from a more
refined result, namely Theorem~\ref{thm:loc_vis_imp_glob_vis_ntrnlzd_gen}, below.
At the heart of this latter result is Lemma~\ref{lmm:ag_wrt_glob_ag_wrt_loc_upld} below. 
This Lemma and two others that follow, and their proofs, 
are substantially similar to Proposition~3.2, Lemma~3.3 and Lemma~3.5 in \cite{ADS2023}.
However, there is an important difference. The results in \cite{ADS2023} quoted above have not
been presented in the form that will be convenient for us, and which is important in order to 
get to Theorem~\ref{thm:loc_vis_imp_glob_vis_ntrnlzd}, 
which deals with a situation more general than the one considered in (one of) the (two) main
result(s) in \cite{ADS2023}, namely \cite[Theorem~1.1]{ADS2023}. (It must be noted, however, that 
\cite[Theorem~1.1]{ADS2023} provides a {\em localization result} for the Kobayashi distance,
whereas our Theorem~\ref{thm:loc_vis_imp_glob_vis_ntrnlzd} is concerned with proving that
visibility with respect to the local Kobayashi distance implies visibility with respect to
the global one.)
What makes \cite[Theorem~1.1]{ADS2023} have a narrower range of applicability than 
Theorem~\ref{thm:loc_vis_imp_glob_vis_ntrnlzd} is the approach, in the former result, of 
dealing with intersections of suitable open sets with the domain under consideration. This 
necessitates the technical, but crucial, assumption that the intersection  is connected.
Since we wish to make no such assumption in Theorem~\ref{thm:loc_vis_imp_glob_vis_ntrnlzd},
we must handle a situation that is topologically quite distinct.

\medskip 

A second point that we must make (and we made it also in 
\Cref{ss:gen_res_visib}) is that recently Nikolov--{\"O}kten--Thomas 
\cite{Nik_Okt_Tho} have proved results very similar to ours (see 
\cite[Theorem~4.1, Theorem~4.2]{Nik_Okt_Tho}). Nonetheless, there are
important differences. Firstly, Nikolov--{\"O}kten--Thomas's work is focused 
entirely on a {\em single} boundary point (see the results cited above, as 
well as \cite[Theorem~3.1, Theorem~3.6, Lemma~3.2]{Nik_Okt_Tho}).
Secondly, it is implicitly assumed in \cite{Nik_Okt_Tho} that when one 
intersects the given domain with suitable neighbourhoods of the boundary 
point of interest, the intersection is connected. Thirdly, the 
statements of their results are qualitatively quite different from ours 
(compare Theorems~\ref{thm:loc_vis_imp_glob_vis_ntrnlzd} and 
\ref{thm:glob_vis_loc_vis_gen} in this paper with 
Theorems~4.1 and 4.2 in \cite{Nik_Okt_Tho}). Finally, it appears that some 
facts pertaining to the differences between what Nikolov--{\"O}kten--Thomas 
call $(\lambda,\kappa)$-geodesics (which we use in our paper) 
and $(\lambda,\kappa)$-{\em almost}-geodesics in the sense of 
Bharali--Zimmer (which we also use)
are not made explicit in \cite{Nik_Okt_Tho}.

\medskip 

Bearing all this in mind, we present full statements and proofs.
	
\begin{lemma} \label{lmm:vsb_sets_loc_met_upld}
Suppose that $\OM\subset\C^d$ is a hyperbolic domain and $U\subset\OM$ is a sub-domain
such that every pair of distinct points of $(\bdy U\cap\bdy\OM)\setminus
\clos{\bdy U\cap\OM}$ possesses the $\lambda$-visibility property with respect to  
$\koba_U$ for some $\lambda\geq 1$.
Let $W_1$ and $W_2$ be two open sets such that $W_1\Subset (U\cup\bdy\OM)\setminus
\clos{\bdy U\cap\OM}$ and $W_2\Subset (W_1\cup\bdy\OM)\setminus \clos{\bdy W_1\cap\OM}$.
Then, for every $\kappa>0$, there exists a compact subset $K$ of $U$ such that every
$(\lambda,\kappa)$-almost-geodesic with respect to $\koba_U$ starting in $W_2$ and ending 
in $U\setminus W_1$ intersects $K$.
\end{lemma}

\begin{proof}
Suppose the above is not true, i.e., there exists a $\kappa_0>0$ such
that for every compact subset $K$ of $U$ there exists a 
$(\lambda,\kappa_0)$-almost-geodesic
with respect to $\koba_U$ that originates in $W_2$, terminates in $U\setminus W_1$ and
avoids $K$. Choose a compact exhaustion $(K_n)_{n\geq 1}$ of $U$. Then for every 
$n$ there exists a $(\lambda,\kappa_0)$-almost-geodesic $\gamma_n : [a_n,b_n] \to U$ for $\koba_U$ 
that originates in $W_2$, terminates in $U\setminus W_1$ and that avoids $K_n$.
Define, for each $n$, 
\[ 
 t_n\defeq \sup\{t\in [a_n,b_n]\mid \gamma_n([a_n,t])\subset W_1\}. 
 \]
Then: $\gamma_n([a_n,t_n))\subset W_1$ and 
$\gamma_n(t_n)\in \bdy W_1 \cap \OM$. As $W_2$ and $W_1$ are relatively
compact, we may
assume that $(\gamma_n(a_n))_{n\geq 1}$ and $(\gamma_n(t_n))_{n\geq 1}$ both converge to
points $\xi$ and $\zeta$ of $\clos{W}_2$ and $\clos{\bdy W_1\cap\OM}$, respectively. From
the fact that the $\gamma_n$ eventually avoid every compact subset of $U$ it follows
that $\xi,\zeta\in\bdy\OM$. 
It is also clear from our hypothesis on $W_1$ and $W_2$ that $\xi\neq\zeta.$
Thus 
$\big(\gamma_n|_{[a_n,t_n]}\big)_{n\geq 1}$ is a sequence of 
$(\lambda,\kappa_0)$-almost-geodesics with respect to $\koba_U$ whose end points
converge to {\em distinct} points of $(\bdy U \cap \bdy \OM)\setminus 
\clos{\bdy U\cap\OM}$. Therefore, by hypothesis, there exists a compact subset $K$ of $U$
that all the $\gamma_n|_{[a_n,t_n]}$ intersect. But that is an immediate contradiction 
whence the result follows.
\end{proof}

Now we prove a lemma to the effect that, under the assumption of visibility, the Kobayashi
distance to a set can be approximated from below by the distance to a single point.

\begin{lemma} \label{lmm:dst_set_rdcd_pt_loc_upld}
Under the same assumptions as in Lemma~\ref{lmm:vsb_sets_loc_met_upld}, given
$o\in U$, there exists $L<\infty$ such that for any $z\in W_2$
\[
 \koba_U(z,o)\leq \koba_U(z,U\setminus W_1)+L. 
 \]
\end{lemma}

\begin{proof}
By Lemma~\ref{lmm:vsb_sets_loc_met_upld}, there exists a compact subset $K$ of $U$ such
that every $(1,1)$-almost-geodesic for $\koba_U$ joining a point of $W_2$ to one of
$U\setminus W_1$ intersects $K$. Let $L\defeq \sup_{w\in K}\koba_U(o,w)$. Fix $z\in W_2$,
and for each $n$, choose $w_n\in U\setminus W_1$ such that 
$\koba_U(z,w_n)\leq \koba_U(z,U\setminus W_1)+(1/n)$. 
Now choose a 
$(1,1/n)$-almost-geodesic $\gamma_n: [a_n,b_n]\to U$ for $\koba_U$ joining $z$ and $w_n$.
Since $z\in W_2$ and $w_n\in U\setminus W_1$, each $\gamma_n$ intersects $K$. For each $n$,
choose $c_n\in [a_n,b_n]$ such that $\gamma_n(c_n)\in K$. For every $n$, $\koba_U(z,w_n) 
\geq \ell_U(\gamma_n) - (1/n)$ because $\gamma_n$, being a $(1,1/n)$-almost-geodesic with 
respect to $\koba_U$, is also (see \Cref{rmk:geod_props}) a $(1,1/n)$-geodesic. Also, for 
every $n$, 
\[ \ell_U(\gamma_n) = \ell_U\big(\gamma_n|_{[a_n,c_n]}\big) + 
\ell_U\big(\gamma_n|_{[c_n,b_n]}\big). \]
Of course, for every $n$, $\ell_U\big(\gamma_n|_{[a_n,c_n]}\big)\geq 
\koba_U(z,\gamma_n(c_n))$. Therefore, for every $n$,
\begin{align*} 
\koba_U(z,U\setminus W_1) &\geq \koba_U(z,w_n)-(1/n) \\
&\geq \ell_U(\gamma_n)-(2/n) \\
&\geq\ell_U\big(\gamma_n|_{[a_n,c_n]}\big)-(2/n) \\
&\geq \koba_U(z,\gamma_n(c_n))-(2/n).
\end{align*}
Now, for every $n$, $\koba_U(z,\gamma_n(c_n))\geq \koba_U(z,o)-\koba_U(o,\gamma_n(c_n))
\geq \koba_U(z,o)-L$ since $\koba_U(o,\gamma_n(c_n))\leq L$. Therefore, from the above,
\[ \koba_U(z,U\setminus W_1) \geq \koba_U(z,o)-L-(2/n). \]
This inequality holds for every $n$. Taking $n\to\infty$, we obtain the required result.
\end{proof}
	
\begin{lemma} \label{lmm:dst_to_set_cmpr_glob_loc}
Suppose that $\OM\subset\C^d$ is a hyperbolic domain and $U\subset\OM$ is a sub-domain.
Let $W\subset U$ be an open set such that $W \Subset (U\cup\bdy\OM) \setminus
\clos{\bdy U \cap \OM}$ and $\koba_\OM(W,\OM\setminus U)>0$. Then there 
exists $C<\infty$ such that
\[
 \koba_U(z,U\setminus W) \leq C \koba_\OM(z,U\setminus W)\ \  \forall\,z\in W.
 \]	
\end{lemma}

\begin{proof}
Let $C\defeq \coth\big(\koba_\OM(W,\OM\setminus U)\big)$; by hypothesis, $C<\infty$.
Suppose that $z\in W$. Let
\[ A_{z,W} \defeq \{ \gamma: [0,1]\overset{{\smoo^1}}{\to}\clos{W} \mid \gamma([0,1))\subset W,\,\gamma(0)=z,\,
\gamma(1)\in \bdy W \cap U \}, \]
\[ B_{z,W} \defeq \{ \ell_U(\gamma) \mid \gamma\in A_{z,W} \}, \;\text{and } 
C_{z,W} \defeq \{ \ell_\OM(\gamma) \mid \gamma\in A_{z,W} \}. \]
Then, since the Kobayashi distance is given by taking the infimum over the lengths of 
curves computed according to the Kobayashi--Royden metric, it follows easily that 
$\koba_U(z,U\setminus W) = \inf B_{z,W}$ and $\koba_\OM(z,U\setminus W)=\inf C_{z,W}$.
Now suppose that $\gamma\in A_{z,W}$. Then
\[ \ell_U(\gamma) = \int_0^1 \dkoba_U(\gamma(t);\gamma'(t))dt \text{ and } 
\ell_\OM(\gamma) = \int_0^1 \dkoba_\OM(\gamma(t);\gamma'(t))dt. \]
A direct consequence of Royden's localization lemma is that
\[ \text{For a.e. } t\in [0,1], \; \dkoba_U(\gamma(t);\gamma'(t)) \leq 
C\dkoba_\OM(\gamma(t);\gamma'(t)). \]
Consequently, it is immediate that $\ell_U(\gamma) \leq C \ell_\OM(\gamma)$. This holds for
every $\gamma$. From this it follows that $\inf B_{z,W} \leq C \inf C_{z,W}$, i.e.,
$\koba_U(z,U\setminus W) \leq C \koba_\OM(z,U\setminus W)$. Finally, $z\in W$ was 
arbitrary, so we are done.
\end{proof}

When we combine Lemmas~\ref{lmm:dst_set_rdcd_pt_loc_upld} and 
\ref{lmm:dst_to_set_cmpr_glob_loc}, we obtain

\begin{lemma} \label{lmm:lb_dist_set_OM_upld}
Under the assumptions as in Lemma~\ref{lmm:vsb_sets_loc_met_upld} and, in addition,
that $\koba_\OM(W_1,\OM\setminus U)>0$, there exists $C<\infty$ such
that, for every $o\in U$, there exists $L<\infty$ such that for all $z\in W_2$, we have
\[ 
 \koba_\OM(z,U\setminus W_1) \geq (1/C)\koba_U(z,o)-
 (L/C) \geq (1/C)\koba_\OM(z,o)-(L/C)\ \  
 \]
\end{lemma}

\begin{proof}
By Lemma~\ref{lmm:dst_to_set_cmpr_glob_loc}, there exists $C<\infty$ such that for all $z\in W_1$
\[
 \koba_U(z,U\setminus W_1) \leq C \koba_\OM(z,U\setminus W_1). 
 \]
Now fix $o\in U$. By Lemma~\ref{lmm:dst_set_rdcd_pt_loc_upld}, there exists $L<\infty$ 
such that for all $z\in W_2$
\[ 
 \koba_U(z,U\setminus W_1)\geq \koba_U(z,o)-L. 
 \]
From the two equations above, it follows that if $z\in W_2$
\[ 
 \koba_\OM(z,U\setminus W_1) \geq (1/C)\koba_U(z,o)-(L/C). 
 \]
Since $\koba_U(z,o)\geq\koba_\OM(z,o)$, the result now follows.
\end{proof}

Our next lemma is the most crucial tool, which says that 
under suitable conditions, non-stationary $(\lambda,\kappa)$-almost-geodesics 
with respect to the global metric are $(\lambda^2,\wt{\kappa})$-almost-geodesics with respect to the 
local metric.

\begin{lemma} \label{lmm:ag_wrt_glob_ag_wrt_loc_upld}
Under the same assumptions as in Lemma~\ref{lmm:lb_dist_set_OM_upld}, given
$\lambda\geq 1$, $\kappa>0$, there exists $\wt{\kappa}>0$ such that every 
$(\lambda,\kappa)$-almost-geodesic $\gamma$ for $\koba_\OM$ included 
in $W_2$\,---\,that is almost everywhere
non-stationary\,---\,is, upto a reparametrization, a 
$(\lambda^2,\wt{\kappa})$-almost-geodesic for $\koba_U$.	
\end{lemma}

\begin{proof}
Fix $o\in U$. 
By Lemma~\ref{lmm:lb_dist_set_OM_upld}, there exists $C<\infty$ such that
for every $z\in W_2$ 
\begin{equation} \label{eqn:lb_kob_dist_to_set}
 \koba_\OM(z,U\setminus W_1) \geq (1/C)\koba_\OM(z,o)-C.
  \end{equation}
Also, by the fact that $\koba_\OM(W_1,\OM\setminus U)>0$ and 
Result~\ref{lmm:refined_Roy_loc_lemm}, we may suppose that the $C$ above is so large that
for every $z\in W_1$ and $v\in\C^d$ 
\begin{equation} \label{eqn:loc_glob_kob_met_ineq_sharp}
 \dkoba_U(z;v)\leq (1+Ce^{-\koba_\OM(z,\OM\setminus U)})
  \dkoba_\OM(z;v).
   \end{equation}

Note (recall \Cref{rmk:geod_props}) that $\gamma$ is a $(\lambda^2, \lambda^2\kappa)$-geodesic with respect to $\koba_\OM$.
Using Proposition~\ref{prp:drvtv_nwh_van_rprm_unt_spd}, we can reparametrize $\gamma$ such that
it is $\dkoba_\OM$-unit-speed. We continue to call this reparametrization $\gamma$ and observe that it may not be a 
$(\lambda,\kappa)$-almost-geodesic with respect to $\koba_\OM$ but it remains a $(\lambda^2, \lambda^2\kappa)$-geodesic
with respect to $\koba_\OM$.
Therefore, for all $s, t\in [a, b]$, we have
\begin{equation} \label{eqn:param_dsts_ineq_koba_gamma}
  (1/\lambda^2)|s-t|-\lambda^2\kappa \leq 
  \koba_\OM(\gamma(s),\gamma(t))\leq |s-t| \ \text{ and}
\end{equation}
\begin{equation} \label{eqn:kob_met_eq_gamma}
  \dkoba_\OM(\gamma(t);\gamma'(t))=1  \   \text{for a.e. } t\in [a,b].
\end{equation}
     
\noindent{\bf Claim.} There exists a $\wt{\kappa}>0$ such that 
$\gamma$ is a $(\lambda^2, \wt{\kappa})$-geodesic for $\koba_U$, i.e.,
for all $s,t\in [a,b]$ with $s\leq t$ we have
\begin{equation*}
 \ell_U(\gamma|_{[s,t]}) \leq 
  \lambda^2\koba_U(\gamma(s),\gamma(t)) + \wt{\kappa}.
   \end{equation*}
   
\noindent Note
\begin{equation*}
\ell_U(\gamma|_{[s,t]}) = \int_s^t \dkoba_U(\gamma(\tau);\gamma'(\tau)) d\tau \leq 
\int_s^t \big( 1 + C e^{-\koba_\OM(\gamma(\tau),\OM\setminus U)} \big) d\tau,
\end{equation*}
where we have used the fact that $\gamma$ is contained in $W_2$, 
\eqref{eqn:loc_glob_kob_met_ineq_sharp}, and \eqref{eqn:kob_met_eq_gamma}.
Therefore
\begin{align} \label{eqn:ub_ell_U_gamma_all_but_int}
\ell_U(\gamma|_{[s,t]}) &\leq (t-s) + C \int_s^t e^{-\koba_\OM(\gamma(\tau),\OM\setminus U)} 
d\tau \notag\\
&\leq \lambda^2\koba_\OM(\gamma(s),\gamma(t))+\lambda^4\kappa+C\int_s^t 
e^{-\koba_\OM(\gamma(\tau),\OM\setminus U)}d\tau,
\end{align}
where, to write the last inequality, we have used \eqref{eqn:param_dsts_ineq_koba_gamma}.
Now, using the fact that the Kobayashi distance is given by infimizing over the
lengths of paths, for all $z\in W_2$, we have 
$\koba_\OM(z,\OM\setminus W_1) = \koba_\OM(z,U\setminus W_1)$.
Furthermore, obviously, $\koba_\OM(z,\OM\setminus U) \geq \koba_\OM(z,\OM\setminus W_1)$.
Therefore, for all $z\in W_2$, we have
\begin{equation*}
 \koba_\OM(z,\OM\setminus U) \geq \koba_\OM(z,U\setminus W_1).
  \end{equation*}
 As $\gamma$ is contained in $W_2$,
we use \eqref{eqn:lb_kob_dist_to_set} to write
\begin{equation*}
 \koba_\OM(\gamma(\tau),U\setminus W_1) \geq (1/C) 
  \koba_\OM(\gamma(\tau),o) - C \ \ \forall\,\tau\in [a,b].
   \end{equation*}
Combining the last two inequalities above, for all $\tau\in [a,b]$, we have
\begin{equation} \label{eqn:flb_kob_dist_gammapoint_comp}
 \koba_\OM(\gamma(\tau),\OM\setminus U) \geq (1/C) 
  \koba_\OM(\gamma(\tau),o)-C.
   \end{equation}
Choose $\tau_0\in [a,b]$ such that $\koba_\OM(\gamma(\tau_0),o)=\min_{\tau\in [a,b]}
\koba_\OM(\gamma(\tau),o)$. Now note, by the choice of $\tau_0$, that
\begin{equation*}
 \koba_\OM(\gamma(\tau),o) \geq 
  \koba_\OM(\gamma(\tau),\gamma(\tau_0))-\koba_\OM(\gamma(\tau_0),o) \geq
   \koba_\OM(\gamma(\tau),\gamma(\tau_0))-\koba_\OM(\gamma(\tau),o)\ \ \forall\,\tau\in [a,b]. 
     \end{equation*}
Thus for all $\tau\in [a,b]$, by \eqref{eqn:param_dsts_ineq_koba_gamma}, we have
\begin{equation*}
2\koba_\OM(\gamma(\tau),o) \geq 
\koba_\OM(\gamma(\tau),\gamma(\tau_0)) \geq (1/\lambda^2)|\tau-\tau_0|-\lambda^2\kappa.
\end{equation*}
Therefore, by 
\eqref{eqn:flb_kob_dist_gammapoint_comp} and the above,
\begin{equation*}
 \koba_\OM(\gamma(\tau),\OM\setminus U) \geq 
  (1/2C\lambda^2)|\tau-\tau_0|-(\lambda^2\kappa/2C)-C \ \ \forall\,\tau\in [a,b].
   \end{equation*}
Note that in the inequality above it is $\tau_0$ that depends on $\gamma$; $\lambda$, $\kappa$ and $C$
are completely independent of $\gamma$. Therefore
\begin{align*}
\int_s^t e^{-\koba_\OM(\gamma(\tau),\OM\setminus U)} d\tau &\leq 
\int_s^t e^{-\frac{1}{2C\lambda^2}|\tau-\tau_0|+\frac{\lambda^2\kappa}{2C}+C} d\tau  \\
&\leq \int_{-\infty}^{\infty}
e^{-\frac{1}{2C\lambda^2}|\tau-\tau_0|+\frac{\lambda^2\kappa}{2C}+C} d\tau = 4C\lambda^2
e^{\frac{\lambda^2\kappa}{2C}+C}.
\end{align*}
Hence, from \eqref{eqn:ub_ell_U_gamma_all_but_int} we have
\begin{align*}
\ell_U(\gamma|_{[s,t]}) &\leq \lambda^2\koba_\OM(\gamma(s),\gamma(t)) + \lambda^2\kappa + 
4C^2 \lambda^2 e^{\frac{\lambda^2\kappa}{2C}+C} \\
&= \lambda^2 \koba_\OM(\gamma(s),\gamma(t)) + \wt{\kappa} \\
&\leq \lambda^2 \koba_U(\gamma(s),\gamma(t)) + \wt{\kappa},
\end{align*}
where the last inequality follows from the Kobayashi-distance-decreasing property of the
inclusion map. This establishes the claim. Appealing to Proposition~\ref{prp:rprm_unt_spd_ag_ag}
we can again reparametrize $\gamma$ so that $\gamma$ is a 
$(\lambda^2,\wt{\kappa})$-almost-geodesic with respect to $\koba_U$. Since $\gamma$ was 
arbitrary, the proof is complete.
\end{proof}

Now we prove a result that shows that local visibility implies global visibility. Roughly
speaking, the result says that if a hyperbolic domain is such that every sequence
converging to every boundary point that is not in a small exceptional subset of the 
boundary has a {\em subsequence} that is contained in a {\em sub-domain} (depending on the
original sequence) that satisfies the $\lambda^4$-visibility property on the ``relative 
boundary'' (see below), then the domain itself satisfies the $\lambda$-visibility property. 
It is noteworthy that we do not need to assume that the boundary points have neighbourhoods 
whose intersections with the domain are connected.

\begin{theorem} \label{thm:loc_vis_imp_glob_vis_ntrnlzd_gen}
Suppose that $\OM\subset\C^d$ is a hyperbolic domain
and let $S\subset \bdy\OM$ be a totally disconnected set. 
Given $\lambda\geq 1$, suppose for every $p\in\bdy\OM \setminus S$ and 
every sequence $(x_n)_{n\geq 1}$ in $\OM$ converging
to $p$, there exist a sub-domain $U\subset\OM$,
a subsequence $(x_{k_n})_{n\geq 1}$ of $(x_n)_{n\geq 1}$
and an $r>0$ such that
\begin{itemize}
\item  $p\in (\bdy U\cap\bdy\OM) \setminus \clos{\bdy U\cap\OM}$;
\smallskip

\item every two distinct points of $(\bdy U \cap \bdy \OM) \setminus \clos{\bdy U \cap \OM}$ possess 
the $\lambda^4$-visibility property with respect to $\koba_U$;
\smallskip

\item for every $n$, $x_{k_n}\in U$ and 
\smallskip

\item $\koba_\OM(B(p;r) \cap U,\OM\setminus U)>0$. 
\end{itemize}
Then $\OM$ is a $\lambda$-visibility domain.
\end{theorem}

\begin{proof}
Assume, to get a contradiction, that $\OM$ is not a $\lambda$-visibility domain.
Then there exist $\kappa>0$, points $p,q\in \bdy\clos{\OM}^{End}$, $p\neq q$,
and a sequence $\gamma_n : [a_n,b_n] \to \OM$ of $(\lambda,\kappa)$-almost-geodesics 
with respect to $\koba_\OM$, such that $(\gamma_n(a_n))_{n\geq 1}$, 
$(\gamma_n(b_n))_{n\geq 1}$ converge to $p$, $q$ respectively, and such that 
$(\gamma_n)_{n\geq 1}$ eventually avoids every compact subset of $\OM$.
By \Cref{rmk:cnvg_outside_tot_disc}, there exist a sequence 
$(t_n)_{n\geq 1}$ with $a_n\leq t_n\leq b_n$ and a point $p_0\in \bdy\OM\setminus S$ 
such that $\gamma_n(t_n)\to p_0$ as $n\to\infty$. We may also suppose that $p_0\neq p,q$.
By hypothesis, there exist a domain $U\subset\OM$, a subsequence 
$\big(\gamma_{k_n}(t_{k_n})\big)_{n\geq 1}$, and an $r>0$, such that
\begin{itemize}
\item $p_0\in(\bdy U\cap\bdy\OM) \setminus \clos{\bdy U \cap \OM}$;
\smallskip

\item any two distinct points of $(\bdy U \cap \bdy \OM) \setminus \clos{\bdy U \cap \OM}$ possess the 
$\lambda^4$-visibility property with respect to $\koba_U$;
\smallskip

\item for every $n$, $\gamma_{k_n}(t_{k_n})\in U$; and 
\smallskip

\item $\koba_\OM\big(B(p_0;r)\cap U,\OM\setminus U\big)>0$.
\end{itemize}
Since $p_0\notin\clos{\bdy U\cap\OM}$, we may choose $r_1>0$ such that $\clos{B}(p_0;r_1)\cap
\clos{\bdy U \cap \OM} = \emptyset$. 
We may also suppose that $r_1$ is so small
that $\koba_\OM\big(B(p_0;r_1)\cap U,\OM\setminus U\big)>0$ 
and that $p,q\notin \clos{B}(p_0;r_1)$. Write $W_1\defeq B(p_0;r_1)\cap U$,
$W_2\defeq B(p_0;r_1/2)\cap U$ and $W_3\defeq B(p_0;r_1/3)\cap U$. 
Then $p_0 \in \bdy W_3 \cap \bdy \OM$, and 
$\clos{\wt{V}}\subset (\wt{U}\cup\bdy\OM)\setminus\clos{\bdy\wt{U}\cap\OM}$,
where $(\wt{V},\wt{U})$ is
any one of the pairs $(W_3,W_2)$, $(W_2,W_1)$ or $(W_1,U)$.
\smallskip
	
Note that, for every $n$ large enough, and hence, without loss of generality, for every
$n$, $\gamma_{k_n}(t_{k_n})\in W_3$.
Note also that, for all $n$, there exists $s<t_{k_n}$ such that 
$\gamma_{k_n}(s)\notin W_3$. This is because we may assume that
$\gamma_{k_n}(a_{k_n})$ is very close to $p$, and consequently 
$\gamma_{k_n}(a_{k_n}) \notin \clos{B}(p_0;r_1)$, whence, a fortiori,
$\gamma_{k_n}(a_{k_n})\notin W_3$.	
For every $n$, write
\begin{equation*}
s_n \defeq \inf\{t\in [a_{k_n},b_{k_n}] \mid 
\gamma_{k_n}\big((t,t_{k_n}]\big) \subset W_3\}.
\end{equation*}
Then the infimum above is actually a minimum, we have 
$\gamma_{k_n}\big((s_n,t_{k_n}]\big)\subset W_3$, and $\gamma_{k_n}(s_n)\in
\bdy W_3 \cap \OM \subset \bdy B(p_0;r_1/3)\cap U$.	
By the compactness of $\bdy B(p_0;r_1/3)$, we may suppose, without loss of generality, 
that $(\gamma_{k_n}(s_n))_{n\geq 1}$ converges to some
point $x_0\in \clos{\bdy B(p_0;r_1/3)\cap U}$. 
Since $(\gamma_n)_{n\geq 1}$ eventually avoids every compact subset of $\OM$, 
it follows that $x_0\in\bdy\OM$. 
Consequently, $x_0\in \bdy B(p_0;r_1/3)\cap
\bdy U\cap\bdy\OM\subset (\bdy U\cap\bdy\OM) \setminus \clos{\bdy U\cap\OM}$. In 
particular, $x_0\neq p_0$.
\smallskip

For every $n$ choose $\epsilon_n>0$ so small that (1) $\epsilon_n \searrow 0$, (2) 
$\lambda^2\epsilon_1<1$, and 
\[(3)\; Nbd_{\C^d}(\gamma_{k_n}([s_n,t_{k_n}]),\epsilon_n)\Subset B(p_0;r_1/2)\cap U=W_2.\]
Note for each $n$, $\gamma_{k_n}|_{[s_n,t_{k_n}]}$ is a $(\lambda^2, \lambda^2\kappa)$-geodesic
for $\koba_\OM$. 
Using Proposition~\ref{prp:alm-geod_apprx_nonstnry}, choose for every $n\in\posint$ an
absolutely continuous curve $\wt{\gamma}_n:[s_n,t_{k_n}]\to\OM$
such that $\wt{\gamma}_n$ is a $(\lambda^2,\lambda^2\epsilon_n+\lambda^2\kappa)$-geodesic for 
$\koba_\OM$, such that $\wt{\gamma}_n$ is almost everywhere non-stationary, and such
that 
\[ 
 \sup_{t\in [s_n,t_{k_n}]} \|\wt{\gamma}_n(t)-\gamma_{k_n}(t)\| < \epsilon_n. 
  \]
Note each $\wt{\gamma}_n$ is a $(\lambda^2,\lambda^2\kappa+1)$-geodesic for
$\koba_\OM$ such that $\wt{\gamma}_n(s_n)\to x_0$ and
$\wt{\gamma}_n(t_{k_n})\to p_0$.
Using Proposition~\ref{prp:rprm_unt_spd_ag_ag}, we reparametrize each $\wt{\gamma}_n$ to
get an absolutely continuous, $\dkoba_\OM$-unit-speed curve $\theta_n:[0,L_n]\to\OM$ that
is a $(\lambda^2,\lambda^2\kappa+1)$-almost-geodesic for $\koba_\OM$. One also has
$\theta_n(0)=\wt{\gamma}_n(s_n)\to x_0$ and $\theta_n(L_n)=\wt{\gamma}_{n}(t_{k_n})\to p_0$.	
Now consider the sequence $(\theta_n)_{n\geq 1}$ of $(\lambda^2,\lambda^2\kappa+1)$-almost-geodesics 
with respect to $\koba_\OM$. Each of them is almost everywhere non-stationary and included 
in $W_2$. By the properties that $W_1$ and $W_2$ satisfy, we may invoke 
Lemma~\ref{lmm:ag_wrt_glob_ag_wrt_loc_upld} to conclude that there exists $\wt{\kappa}>0$ 
such that, for every $n$, $\theta_n$ is a $(\lambda^4,\wt{\kappa})$-almost-geodesic with 
respect to $\koba_U$. ($\wt{\kappa}$ depends only on $\lambda, \kappa$, $U$ and $r_1$.)
\smallskip
	
So consider the sequence $\theta_n$ of $(\lambda^4,\wt{\kappa})$-almost-geodesics with 
respect to $\koba_U$. The two sequences of its endpoints, $(\theta_n(0))_{n\geq 1}$ and 
$(\theta_n(L_n))_{n\geq 1}$ converge to $x_0$ and $p_0$, respectively,
which are {\em distinct} points of $(\bdy U \cap \bdy\OM) \setminus \clos{\bdy U \cap \OM}$.
Therefore, by hypothesis, there exists a compact subset $K$ of $U$ such that,
for every $n$, $\theta_n\big([0,L_n]\big)\cap K\neq\emptyset$. But then it is obvious, by
how the $\theta_n$ have been constructed, that there also exists a compact subset $\wt{K}$
of $\OM$ such that, for every $n$, $\mathsf{ran}(\gamma_{k_n})\cap \wt{K}\neq \emptyset$, 
which is a contradiction to our starting assumption. This contradiction completes the proof.
\end{proof}

\begin{theorem} \label{thm:loc_vis_imp_glob_vis_ntrnlzd}
Let {\bf (V)} stand for either ``the weak visibility property'' or ``the visibility property''. 
Suppose that $\OM\subset\C^d$ is a hyperbolic domain and let $S\subset \bdy\OM$ be a  
totally disconnected set. Suppose that for every $p\in\bdy\OM \setminus S$ 
and every sequence $(x_n)_{n\geq 1}$ in $\OM$ converging to $p$, there exist a sub-domain 
$U\subset\OM$, a subsequence $(x_{k_n})_{n\geq 1}$ of $(x_n)_{n\geq 1}$ and an $r>0$ such that
\begin{itemize}
\item  $p\in (\bdy U\cap\bdy\OM) \setminus \clos{\bdy U\cap\OM}$;
\smallskip
	
\item every two distinct points of $(\bdy U \cap \bdy \OM) \setminus \clos{\bdy U \cap \OM}$ 
possess {\bf (V)} with respect to 
$\koba_U$;
\smallskip
		
\item for every $n$, $x_{k_n}\in U$ and 
\smallskip
		
\item $\koba_\OM(B(p;r) \cap U,\OM\setminus U)>0$. 
\end{itemize}
Then $\OM$ possesses {\bf (V)} with respect to $\koba_\OM$, i.e., $\OM$ is either a weak
visibility domain or a visibility domain.
\end{theorem}

\begin{proof}
If {\bf (V)} stands for ``the weak visibility property'', then appealing to Theorem~\ref{thm:loc_vis_imp_glob_vis_ntrnlzd_gen} with $\lambda=1$
shows that $\OM$ is a weak visibility domain. 
\smallskip

If, on the other hand, {\bf (V)} stands for ``the visibility property'', then, for every
$\lambda\geq 1$, we appeal to Theorem~\ref{thm:loc_vis_imp_glob_vis_ntrnlzd_gen} to conclude
that $\OM$ is a $\lambda$-visibility domain. Since $\lambda\geq 1$ is arbitrary, it follows
that $\OM$ is a visibility domain.
\end{proof}

\begin{remark} \label{rmk:BSP_assmp_suff}
In the theorem above, we could have dropped the assumption involving $r$ (in the fourth bullet point) and
instead assumed that $\OM$ satisfies ${\rm BSP}$, as we do in the theorem below.
\end{remark}

\begin{theorem} \label{thm:loc_vis_imp_glob_vis_xtrnl}
Let {\bf (V)} stand for either ``the weak visibility property'' or ``the visibility property''.
Suppose that $\OM\subset\C^d$ is a hyperbolic domain that satisfies ${\rm BSP}$ and let $S\subset\bdy\OM$
be a totally disconnected set. Suppose that, for
every $p\in\bdy\OM\setminus S$, there exists a neighbourhood $U$ of $p$ in $\C^d$ such that $U\cap\OM$ 
has only finitely many connected components, say $V_1,\dots,V_m$, and such that for each $j=1\dots m$, 
each pair of distinct points of $\bdy V_j \cap \bdy\OM$ satisfies {\bf (V)} with respect to $\koba_{V_j}$. 
Then $\OM$ satisfies {\bf (V)} with respect to $\koba_\OM$, i.e., $\OM$ is either a visibility domain or
a weak visibility domain.
\end{theorem}

\begin{proof}
We will invoke Theorem~\ref{thm:loc_vis_imp_glob_vis_ntrnlzd}. 
We verify that the hypotheses are satisfied. Consider a sequence $(x_n)_{n\geq 1}$ of points of
$\OM$ converging to an arbitrary point $p\in\bdy\OM\setminus S$. By hypothesis, there exists a 
neighbourhood $U$ of $p$ such that $U\cap\OM$ has finitely many components, say $V_1,\dots,V_m$, 
and such that, for every $j=1,\dots,m$, every pair of
distinct points of $\bdy V_j\cap\bdy\OM$ possesses {\bf (V)} with respect to $\koba_{V_j}$.
Since $x_n$ converges to $p$, since $U$ is a neighbourhood of $p$ and since
\[ U\cap\OM = V_1 \cup \dots \cup V_m, \]
there exists $j\in \{1,\dots,m\}$ such that $V_j$ contains $x_n$ for infinitely many $n$. We
may suppose, without loss of generality, that $V_1$ contains $x_n$ for infinitely many $n$;
this implies that there is a subsequence $x_{k_n}$ of $x_n$ lying in $V_1$. Note that $p\in
(\bdy V_1\cap\bdy\OM)\setminus \clos{\bdy V_1\cap\OM}$. 
Secondly, again by hypothesis, every pair of distinct points of
$\bdy V_1\cap\bdy\OM$ possesses {\bf (V)} with respect to $\koba_{V_1}$. Thirdly, if we choose
$r_0>0$ so small that $B(p;r_0)\cap\clos{\bdy V_1\cap\OM}=\emptyset$, it follows from 
Lemma~\ref{lmm:kob_sep_intrnl_dom} that $\koba_\OM(B(p;r_0)\cap V_1,\OM\setminus V_1)>0$. Thus all the 
hypotheses of Theorem~\ref{thm:loc_vis_imp_glob_vis_ntrnlzd} are
satisfied, with $V_1$ playing the role of the $U$ in the statement of the theorem. 
Consequently, by that theorem, $\OM$ possesses {\bf (V)} with respect to $\koba_\OM$.
\end{proof}

\subsection{Global visibility implies local visibility}\label{SS:glob_impl_loc}

In this subsection, we shall prove that, given a subdomain $U$ of $\OM$, for every 
$\lambda\geq 1$, $\lambda$-visibility with respect to $\koba_\OM$ implies 
$\lambda$-visibility with respect to $\koba_U$ (note that this is in contrast to the results 
in the previous subsection, which showed only that local $\lambda^4$-visibility implies 
global $\lambda$-visibility). This is achieved by showing that every 
$(\lambda, \kappa)$-almost geodesic with respect to $\koba_U$ is a 
$(\lambda, \widetilde{\kappa})$-almost geodesic with respect to $\koba_\OM$  for some 
$\widetilde{\kappa}>0$. At the heart of this latter result is a localization lemma
\Cref{lmm:loc_kob_dist_vis_glob_met_assmp} that says
\[
\koba_U\leq \koba_\OM+C
\]
on an appropriate open subset $W$ of $U$. This result is very similar to (and
inspired by) \cite[Theorem~1.3]{ADS2023} but it is in fact more general (it applies 
to a broader range of 
situations and eschews the technical but essential connectedness assumption made in 
\cite[Theorem~1.3]{ADS2023}, exactly as described at the beginning of 
subsection~\ref{SS:loc_implies_glob}). In the sense that this lemma deals with
arbitrary {\em sub-domains} $U$ of a given domain $\OM$ rather than with an 
``external'' intersection of the form $U\cap\OM$, it is {\em intrinsic}, as pointed
out in the abstract. We begin with stating two lemmas that 
are analogues of Lemma~\ref{lmm:vsb_sets_loc_met_upld} and
Lemma~\ref{lmm:dst_set_rdcd_pt_loc_upld} for $\koba_\OM$. We skip the proofs as 
they are closely analogous to those of the aforementioned lemmas. 
\begin{lemma} \label{lmm:visib_sets_glob_met}
Suppose that $\OM\subset\C^d$ is a hyperbolic domain and $U\subset \Omega$ is
a subdomain such that, given $\lambda\geq 1$, every pair of distinct points of 
$\bdy U \cap \bdy\OM$ satisfies the 
$\lambda$-visibility property with respect to $\koba_\OM$. 
Let $W_2\subset W_1\subset U$ be open sets such
that $W_2\Subset (W_1 \cup \bdy\OM) \setminus \clos{\bdy W_1 \cap \OM}$.
Then, for every $\kappa>0$, there exists a compact subset $K$ of $\OM$ such that every
$(\lambda,\kappa)$-almost-geodesic for $\koba_\OM$ that starts in $W_2$ and ends in
$\OM\setminus W_1$ intersects $K$. 
\end{lemma}

\begin{lemma} \label{lmm:dst_to_set_red_to_pt_glob}
        Suppose that $\OM\subset\C^d$ is a hyperbolic domain and $U\subset \Omega$ is
	a subdomain such that, given $\lambda\geq 1$, every pair of distinct points 
	of $\bdy U \cap \bdy\OM$ satisfies the 
	$\lambda$-visibility property with respect to $\koba_\OM$. 
	Then for any open set $W\subset U$\,---\,such
	that $W \Subset (U \cup \bdy\OM) \setminus \clos{\bdy U \cap \OM}$\,---\,and any 
        $o\in \OM$, there exists $L<\infty$ such that
	\begin{equation*}
         \koba_\OM(z,o) \leq \koba_\OM(z,\OM\setminus U)+L, \ \ \forall\,z\in W. 
	\end{equation*} 
\end{lemma}

Now we prove the intrinsic localization result for the Kobayashi distance mentioned 
above.
The proof of this result is more or less the same as that of \cite[Theorem~1.3]
{ADS2023},
with the (minor) necessary modifications to handle the more general situation to which the present result applies.

\begin{lemma} \label{lmm:loc_kob_dist_vis_glob_met_assmp}
	Suppose that $\OM\subset\C^d$ is a hyperbolic domain and $U\subset \Omega$ is
	a subdomain such that, given $\lambda\geq 1$, every pair of distinct points of $\bdy U \cap \bdy\OM$ satisfies the 
	$\lambda$-visibility property with respect to $\koba_\OM$. 
	Then for any open set $W\subset U$ such
	that $W \Subset (U \cup \bdy\OM) \setminus \clos{\bdy U \cap \OM}$,
	there exists a $C<\infty$ such that
	\begin{equation*}
         \koba_U(z,w) \leq \koba_\OM(z,w)+C \ \ \text{ for all $z,w\in W$}.
	\end{equation*}
\end{lemma}

\begin{proof}
	Suppose the conclusion of the lemma is not true. Then, for all $n\in\mathbb{Z}_{+}$,
	there exist points $z_n, w_n\in W$ such that
	\[ 
	\koba_U(z_n,w_n) > \koba_\OM(z_n,w_n) + n. 
	\]
	By the compactness of $\clos{W}$, we may assume, without loss of 
	generality, that $(z_n)_{n\geq 1}$ and $(w_n)_{n\geq 1}$ converge
	to points $\wh{z}$ and $\wh{w}$ of $\clos{W}$, respectively.
	For every $n$, we choose a $(1,1)$-almost-geodesic $\gamma_n : [a_n,b_n]
	\to \OM$ with respect to $\koba_\OM$ joining $z_n$ and $w_n$. Fix a point
	$o\in U$ and let $t^{(0)}_n\in [a_n,b_n]$ be such that
	\[
	\koba_\OM(\gamma_n(t^{(0)}_n),o)=\min_{t\in [a_n,b_n]}\koba_\OM(\gamma_n(t),o) \ \ \forall n\in\mathbb{Z}_{+}.
	\]
	
	Now fix an open set $W_1\subset U$ such that $W\Subset (W_1\cup\bdy\OM)\setminus
	\clos{\bdy W_1\cap\OM}$ and $W_1 \Subset (U \cup \bdy\OM)\setminus \clos{\bdy U \cap \OM}$.
	Depending on the position of $\gamma_n$ relative to $W_1$ two cases arise.
	\smallskip
	
	\noindent {\bf Case~1.} ${\rm ran}(\gamma_n)\subset W_1$ for all but finitely many $n$.
	\smallskip
	
	\noindent We assume, without loss of generality, that $\gamma_n
	\subset W_1$ for {\em all} $n\in\posint$. Note that
	\begin{equation} \label{eqn:leq_def_path_n}
	\koba_U(z_n,w_n) \leq \int_{a_n}^{b_n} \dkoba_U(\gamma_n(t);\gamma'_n(t)) dt \ \ \forall\,n\in\posint.
	\end{equation}
	Lemma~\ref{lmm:kob_sep_intrnl_dom} implies that 
	$\koba_\OM(W_1,\OM\setminus U)>0$, and
	therefore, by Lemma~\ref{lmm:refined_Roy_loc_lemm} there exists $L<\infty$ such that 
	\begin{equation*}
	\dkoba_U(z;v)\leq (1+Le^{-\koba_\OM(z,\OM\setminus U)}) \dkoba_{\OM}(z;v) \ \ \forall\,z\in W_1,\,\forall\,v\in\C^d.
	\end{equation*}
	Since $\gamma_n\subset W_1$ for all $n\in\posint$, the above inequality, together with the fact that each
	$\gamma_n$ is a $(1, 1)$-almost-geodesic implies that
		\begin{align*}
		\dkoba_U(\gamma_n(t);\gamma'_n(t)) &\leq 
		(1+Le^{-\koba_\OM(\gamma_n(t),\OM\setminus U)})\dkoba_{\OM}(\gamma_n(t);\gamma'_n(t))\\
		&\leq 1+Le^{-\koba_\OM(\gamma_n(t),\OM\setminus U)} \ \ \forall\,n\in\posint,\,\text{for a.e. } t\in [a_n,b_n].
		\end{align*}
	This last inequality and \eqref{eqn:leq_def_path_n} implies that for every $n\in\posint$
	\begin{align}
		\koba_U(z_n,w_n) &\leq \int_{a_n}^{b_n} 
		\dkoba_U(\gamma_n(t);\gamma'_n(t)) dt \notag\\
		&\leq (b_n-a_n) + L \int_{a_n}^{b_n} e^{-\koba_\OM(\gamma_n(t),\OM\setminus U)}
		dt \notag\\
		&\leq \koba_\OM(z_n,w_n) + 1 + L \int_{a_n}^{b_n}
		e^{-\koba_\OM(\gamma_n(t),\OM\setminus U)} dt. \label{eqn:k_U_OM_int_rel}
	\end{align}
	(One has $b_n-a_n\leq\koba_\OM(z_n,w_n)+1$ because $\gamma_n:[a_n,b_n]\to W_1$ is a
	$(1,1)$-almost-geodesic with respect to $\koba_\OM$.)
	By Lemma~\ref{lmm:dst_to_set_red_to_pt_glob}, we can assume that $L$ is so large that one
	also has
	\begin{equation} \label{eqn:dist_to_set_pt}
		\koba_\OM(\gamma_n(t),\OM\setminus U)
		\geq \koba_\OM(\gamma_n(t),o)-L, \ \ \forall\,n\in\posint,\,\forall\,t\in [a_n,b_n]. 
	\end{equation}
	Now, the choice of $t^{(0)}_n$, implies that 
	\begin{equation*} \label{eqn:dist_to_pt_eff_lb}
		\koba_\OM(\gamma_n(t),o) \geq \koba_\OM(\gamma_n(t),\gamma_n(t^{(0)}_n)) - 
		\koba_\OM(\gamma_n(t^{(0)}_n),o) \geq \koba_\OM(\gamma_n(t),\gamma_n(t^{(0)}_n)) -
		\koba_\OM(\gamma_n(t),o). 
	\end{equation*}
	This together with the fact that $\gamma_n$ is a $(1,1)$-almost-geodesic with respect to  
	$\koba_\OM$, we get 
	\begin{equation*} \label{eqn:dist_to_pt_fin_lb}
		\koba_\OM(\gamma_n(t),o) \geq (1/2)|t-t^{(0)}_n|-(1/2). 
	\end{equation*}
	The last inequality together with \eqref{eqn:dist_to_set_pt}, 
	implies that
	\begin{align*} \label{eqn:int_e_minus-dist_prel_bd}
		\int_{a_n}^{b_n} e^{-\koba_\OM(\gamma_n(t),\OM\setminus U)} dt &\leq 
		\int_{a_n}^{b_n} e^{-(1/2)|t-t^{(0)}_n|+(1/2)+L} dt = e^{(1/2)+L}
		\int_{a_n}^{b_n} e^{-(1/2)|t-t^{(0)}_n|} dt.\notag \\
		&\leq 4\,e^{(1/2)+L}. 
	\end{align*}
	The above inequality together with \eqref{eqn:k_U_OM_int_rel} implies that
	\begin{equation*}
         \koba_U(z_n,w_n) \leq \koba_\OM(z_n,w_n) + 1 + 4 L 
		e^{(1/2)+L}, \ \ \forall\,n\in\posint.
	\end{equation*}
	But this contradicts our starting assumption.
	\smallskip 
	
	Note that the argument above implies that given $z,w\in W$, if there is a $(1,1)$-almost-geodesic 
	for $\koba_\OM$ joining $z$ and $w$ and contained in $W_1$ 
	then there exists $L<\infty$ such that $\koba_U(z,w)\leq \koba_\OM(z,w) + 1 + 4Le^{(1/2)+L}$.
	\smallskip
	
	\noindent {\bf Case 2.} $\gamma_n\not\subset W_1$ for infinitely many $n$.
	\smallskip
	
	\noindent We may assume, without loss of generality, that $\gamma_n
	\not\subset W_1$ for all $n$. By Lemma~\ref{lmm:visib_sets_glob_met}, there exists a compact subset
	$K$ of $\OM$ such that every $(1,1)$-almost-geodesic for $\koba_\OM$ starting in $W$ and
	ending in $\OM\setminus W_1$ intersects $K$. We may suppose, without loss of generality, that 
	$K\subset U$. For every $n$, we define
	\begin{align*}
		u_n &\defeq \inf\{t\in [a_n,b_n]\mid \gamma_n(t)\notin W_1\}; \\
		v_n &\defeq \sup\{t\in [a_n,b_n]\mid \gamma_n(t)\notin W_1\}.
	\end{align*}
	Then $a_n<u_n\leq v_n<b_n$. Also, for every $n$, $\gamma_n([a_n,u_n))\subset W_1$,
	$\gamma_n((v_n,b_n])\subset W_1$, and $\gamma_n(u_n), \gamma_n(v_n)\in\bdy W_1\cap\OM$.
	The foregoing then implies that $\gamma_n([a_n,u_n])\cap K \neq\emptyset$
	and $\gamma_n([v_n,b_n])\cap K \neq\emptyset$ for every $n$. Choose $s_n\in [a_n,u_n]$
	and $s'_n\in [v_n,b_n]$ such that $\gamma_n(s_n)\in K$ and $\gamma_n(s'_n)\in K$.
	\smallskip
	
	Note by the remark made at the end of the discussion in Case~1, there exists $L<\infty$, such that
	\begin{align*}
		\koba_U(z_n,\gamma_n(s_n)) &\leq \koba_\OM(z_n,\gamma_n(s_n))+1+4Le^{(1/2)+L} \ \ 
		\text{and}\\
		\koba_U(\gamma_n(s'_n),w_n) &\leq \koba_\OM(\gamma_n(s'_n),w_n)+1+4Le^{(1/2)+L}.
	\end{align*}
	Since $K\subset U$ is a compact set, it follows that we may
	increase $L$ further so that we also have
	\[ 
	 \koba_U(z,w)-\koba_\OM(z,w) \leq L,  \ \ \forall\,z,w\in K.
	  \]
	Using the above three inequalities, we get
	\begin{align*} 
		\koba_U(z_n,w_n) &\leq \koba_U(z_n,\gamma_n(s_n)) + \koba_U(\gamma_n(s_n),\gamma_n(s'_n))
		+ \koba_U(\gamma_n(s'_n),w_n) \\
		&\leq \koba_\OM(z_n,\gamma_n(s_n)) + \koba_\OM(\gamma_n(s_n),\gamma_n(s'_n)) +
		\koba_\OM(\gamma_n(s'_n),w_n) + 2 + L + 8Le^{(1/2)+L}. 
	\end{align*} 
	Now, since $z_n,\gamma_n(s_n),\gamma_n(s'_n)$ and $w_n$ all
	lie on the $(1,1)$-almost-geodesic $\gamma_n$, we have
	\[ 
	 \koba_\OM(z_n,\gamma_n(s_n)) + \koba_\OM(\gamma_n(s_n),\gamma_n(s'_n)) +
	\koba_\OM(\gamma_n(s'_n),w_n) \leq \koba_\OM(z_n,w_n) + 6. 
	\]
	Let $M\defeq 8+L+8Le^{(1/2)+L}$ and note that $M$ is determined as soon as $U,W$ and
	$W_1$ are specified. By the foregoing inequality 
	\[ 
	 \koba_U(z_n,w_n) \leq \koba_\OM(z_n,w_n) + M, \ \ \forall\,n\in\posint. 
	 \]
	This shows once again that $\koba_U(z_n,w_n)-\koba_\OM(z_n,w_n)$ is bounded, and yields a
	contradiction.
	\smallskip
	
	So, our starting assumption must be wrong, and so the stated result holds.
\end{proof}

Now we prove that, under suitable hypotheses, $(\lambda,\kappa)$-almost-geodesics with 
respect to the local metric are $(\lambda,\wt{\kappa})$-almost-geodesics with respect to 
the global metric for some $\wt{\kappa}>0$.

\begin{corollary} \label{coro:ag_wrt_loc_ag_wrt_glob}
	Under the same assumptions as in Lemma~\ref{lmm:dst_to_set_red_to_pt_glob}, given $\lambda\geq 1$
	and $\kappa>0$, there exists $\wt{\kappa}>0$ such that every $(\lambda,\kappa)$-almost-geodesic 
	with respect to $\koba_U$ included in $W$ is a $(\lambda,\wt{\kappa})$-almost-geodesic with 
	respect to $\koba_\OM$.
\end{corollary}

\begin{proof}
	By Lemma~\ref{lmm:loc_kob_dist_vis_glob_met_assmp}, there exists $C<\infty$ such that
	\begin{equation} \label{eqn:kob_U_OM_cmpr}
	\koba_\OM(z,w)\leq\koba_U(z,w)\leq\koba_\OM(z,w)+C, \ \ \forall\,z,w\in W. 
	\end{equation}
	The above together with the fact that $\dkoba_{\OM}(z;v)\leq \dkoba_U(z;v)$ for all $z\in U$ and $v\in\C^d$
	implies the result with $\wt{\kappa}\defeq \kappa+C$. 	
\end{proof}

Now we prove a result that shows that global visibility implies local visibility. Roughly speaking, it says that if we start
out with a (weak) visibility domain then any sub-domain whose boundary intersects that of the original domain in a non-degenerate
way (see below) possesses the (weak) visibility property on the ``relative boundary'', where the latter term has the same meaning
as in the context of Theorem~\ref{thm:loc_vis_imp_glob_vis_ntrnlzd_gen}.

\begin{theorem}\label{thm:glob_vis_loc_vis_gen}
	Suppose $\OM\subset\C^d$ is a hyperbolic domain that satisfies
	 the $\lambda$-visibility property for a given  $\lambda\geq 1$. Let 
	$U\subset\OM$ be a subdomain such that
	$(\bdy U \cap \bdy\OM) \setminus \clos{\bdy U \cap \OM} \neq \emptyset$. Then every
	two distinct points of $(\bdy U \cap \bdy\OM) \setminus \clos{\bdy U \cap \OM}$ satisfy
	the $\lambda$-visibility property with respect to $\koba_U$.
\end{theorem}

\begin{proof}
Let $U$ be as above and choose $\zeta\neq\eta$
in $(\bdy U \cap \bdy\OM) \setminus \clos{\bdy U \cap \OM}$. Given $\kappa>0$, let 
$\gamma_n:[a_n, b_n]\lraw U$ be a sequence of $(\lambda,\kappa)$-almost-geodesics 
for $\koba_U$ such that $(\gamma_n(a_n))_{n\geq 1}$ converges to $\zeta$ and 
$(\gamma_n(b_n))_{n\geq 1}$ converges to $\eta$. Furthermore, assume that the sequence 
$(\gamma_n)_{n\geq 1}$ eventually avoids every compact set in $U$.
\smallskip

Since $\zeta\notin\clos{\bdy U\cap\OM}$, we may choose $r>0$ such 
that $\clos{B(\zeta;r)}\cap\clos{\bdy U\cap\OM}=\emptyset$ and such that 
$\eta\notin\clos{B(\zeta;r)}$. Therefore if we write
$W\defeq B(\zeta;r)\cap U$ then $\zeta\in (\bdy W\cap 
\bdy\OM) \setminus \clos{\bdy W\cap\OM}$ and $W\Subset (U\cup\bdy\OM)\setminus
\clos{\bdy U\cap\OM}$. 	
Therefore Lemma~\ref{lmm:kob_sep_intrnl_dom} implies
\begin{equation*} \label{eqn:koba_sep_cond} 
\koba_\OM(W,\OM\setminus U)>0.
\end{equation*}
For every $n$, put
\[ 
t_n \defeq \sup\{t\in [a_n,b_n]\mid \gamma_n([a_n,t])\subset B(\zeta;r)\}.
\]
Since $\eta\notin\clos{B}(\zeta;r)$, it follows that, for every $n$, 
$t_n<b_n$, $\gamma_n(t_n)\in \bdy B(\zeta;r)$ and 
\begin{equation*} \label{eqn:gamma_n_cont_cond} 
\gamma_n([a_n,t_n))\subset W.
\end{equation*}
This together with the choice of $W$\,---\,by appealing to 
Corollary~\ref{coro:ag_wrt_loc_ag_wrt_glob}\,---\,implies that there exists 
$\wt{\kappa}>0$ such that, for every $n$, $\gamma_n|_{[a_n,t_n]}$ is a
$(\lambda,\wt{\kappa})$-almost-geodesic with respect to $\koba_\OM$. We may assume, 
without loss of generality, that $(\gamma_n(t_n))_{n\geq 1}$ converges to some point $\xi\in 
\bdy B(\zeta;r)$. Observe that $\xi\in (\bdy U \cap \bdy\OM)\setminus\clos{\bdy U\cap\OM}$, 
and $\zeta\neq\xi$. Since $\OM$ satisfies the $\lambda$-visibility property and
$\gamma_n|_{[a_n,t_n]}$ is a $(\lambda,\wt{\kappa})$-almost-geodesic with respect to
$\koba_\OM$, there exists a compact subset $\wt{K}$ of $\OM$ such that, for every $n$, 
$\mathsf{ran}\big(\gamma_n|_{[a_n,t_n]}\big) \cap \wt{K} \neq \emptyset$. This implies
there exists a compact subset $K$ of $U$ such that, for every $n$, 
$\mathsf{ran}\big(\gamma_n|_{[a_n,t_n]}\big) \cap K \neq \emptyset$, which contradicts the
assumption that the sequence $(\gamma_n)_{n\geq 1}$ eventually avoids every compact 
subset of $U$. Hence the result. 
\end{proof}

We now state the following theorem, which is an immediate corollary of the one above. We omit the
proof.

\begin{theorem} \label{thm:glob_vis_loc_vis}
Suppose that $\OM\subset\C^d$ is a (weak) visibility domain. Let $U\subset\OM$ be a subdomain 
such that
$(\bdy U \cap \bdy\OM) \setminus \clos{\bdy U \cap \OM} \neq \emptyset$. Then every
two distinct points of $(\bdy U \cap \bdy\OM) \setminus \clos{\bdy U \cap \OM}$ satisfy the
(weak) visibility property with respect to $\koba_U$. 
\end{theorem}

The proof of Theorem~\ref{thm:glob_vis_iff_loc_vis} is now clear.

\begin{proof}[Proof of Theorem~{\ref{thm:glob_vis_iff_loc_vis}}]
First suppose that $\OM$ is a (weak) visibility domain. To show that it is a local (weak) 
visibility domain, take 
$p\in\bdy\OM$ and a sequence $(x_n)_{n\geq 1}$
in $\OM$ converging to $p$. Choose and fix a bounded neighbourhood $W$ of $p$. Using
\Cref{prp:vsb_nbd_ntrsc_fin}, there exist finitely many components of $W\cap\OM$, say
$V_1,\dots,V_m$, such that, for all $n$, $x_n$ is contained in one of the $V_j$. So we may
suppose, without loss of generality, that $V_1$ contains $x_n$ for infinitely many $n$.
This means there is a subsequence $(x_{k_n})_{n\geq 1}$ of $(x_n)_{n\geq 1}$ lying in $V_1$.
It is easy to see that $p\in (\bdy V_1\cap\bdy\OM)\setminus\clos{\bdy V_1\cap\OM}$. Therefore
we may directly invoke \Cref{thm:glob_vis_loc_vis} to conclude that every two distinct points
of $(\bdy V_1\cap\bdy\OM)\setminus\clos{\bdy V_1\cap\OM}$ possess the (weak) visibility
property. Since $p$ was arbitrary, this shows that $\OM$ is a local (weak) visibility domain.
Note that we did not need the ${\rm BSP}$ assumption for this implication.

Conversely, suppose that $\OM$ is a local (weak) visibility domain. Then, using the ${\rm BSP}$
assumption and \Cref{rmk:BSP_assmp_suff}, we see that we may invoke \Cref{thm:loc_vis_imp_glob_vis_ntrnlzd}
directly to conclude that $\OM$ is a (weak) visibility domain (we take $S\defeq\emptyset$ in
\Cref{thm:loc_vis_imp_glob_vis_ntrnlzd}). 
\end{proof} 

\section{Visibility of planar domains}\label{S:Visibility_Thm}

\subsection{ Conditions 1, 2 and 3}
Given two {\em distinct} points $x_1, x_2 \in \partial \mathbb{D}$,
\begin{enumerate}
	\item $Arc(x_1, x_2)$ and $Arc[x_1, x_2]$ will denote the open arc and 
	the closed arc, respectively, joining $x_1$ and $x_2$ and containing $1$;
	\smallskip
	
	\item $(x_1, x_2)_p$ will denote the bi-geodesic ray in $\unitdisk$ for 
	the hyperbolic distance induced by the
	Poincar{\'e} metric joining the boundary points $x_1$ and $x_2$, and 
	$[x_1, x_2]_p$ will denote $(x_1, x_2)_p\cup\{x_1, x_2\}$;
	\smallskip
	
	\item $Reg(x_1, x_2)$ will denote the open region bounded by
	the arc $Arc[x_1, x_2]$ and the bi-geodesic ray $[x_1, x_2]_p$.
\end{enumerate}

We now introduce three conditions that a hyperbolic planar domain $\OM\subset\C$
may satisfy.
\vspace{0.001mm}

\noindent
{\bf Condition~1.}
$\OM$ satisfies Condition~1 if for every point $p \in \partial 
\Omega $ there exists a topological embedding $\tau_p: \overline{\mathbb{D}} 
\longrightarrow \overline{\Omega}$ such that $\tau_p(\unitdisk)\subset\Omega$ and 
such that for any sequence $(z_n)_{n\geq 1}$ in $\Omega$ converging to 
$p$, the tail of the sequence must lie in $\tau_p(\unitdisk)$. 
\smallskip

It is not difficult to see that the above condition is equivalent to the following: 
for any point $p\in\bdy{\OM}$, there exist an $r>0$ and a topological embedding 
$\tau_p: \overline{\mathbb{D}}\longrightarrow \overline{\Omega}$ such that $\tau_p(\unitdisk)\subset\Omega$ and 
$B(p, r)\cap\OM\subset\tau_p(\unitdisk)$. 
\smallskip

Examples of domains that satisfy Condition~1 include: domains with continuous boundary;
simply connected domains with boundary a Jordan curve; multiply (possibly infinitely)
connected domains each of whose boundary components is a Jordan curve for which there exists
a positive quantity that is a lower bound for the Euclidean distance between any two of
these components.  
\smallskip

\noindent
{\bf Condition~2.}
$\OM$ satisfies Condition~2 if there
exists a closed, totally disconnected set $S \subset 
\bdy\OM$ such that for every $p \in \bdy\OM \setminus S $ and every sequence 
$(z_n)_{n \geq 1}$ in $\OM$ with $z_n \to p  $, there exist a subsequence 
$(z_{k_n})_{n \geq 1}$, an injective 
holomorphic map $\phi: \mathbb{D} \longrightarrow \Omega$ that extends continuously
to $\clos{\unitdisk}$ and 
an arc $Arc(x_1,\,x_2)$ containing $1$ such that 
\begin{enumerate} 
\item $\phi(1)=p$ and $\phi(Arc(x_1, x_2)) \subset \partial \OM$,
\smallskip

\item for all $n \in \posint$, $z_{k_n} \in \phi(Reg(x_1, x_2))$ and 
$\phi^{-1}(z_{k_n}) \to 1$ as $n \to \infty$. 
\end{enumerate}

\begin{remark}\label{Rem:distendpoints}
The existence of $\phi$ and $Arc(x_1,\,x_2)$ also implies that we may, without loss of 
generality, assume that the points $\phi(x_1), \phi(x_2)$ and $\phi(1)$ are pairwise 
distinct points of $\bdy\OM$. 
To see this, consider $\phi_1=\phi-\phi(1)$ which is holomorphic in $\unitdisk$ and extends 
continuously to $\overline{\unitdisk}$. Denoting the zero set of $\phi_1$ by $Z(\phi_1)$, we 
first observe that the arc-length 
measure of $Z(\phi_1)\cap\bdy\unitdisk$ is zero. Choose a point $x_1'\in Arc(x_1,\,x_2)\setminus 
Z(\phi_1)$ and consider $\phi_2=\phi-\phi(x_1')$. The set $Z(\phi_2)\cap\bdy\unitdisk$ has 
arc-length measure zero. 
Therefore, we can choose $x_2'\in Arc(x_1,\,x_2)\setminus\big(Z(\phi_1)\cup Z(\phi_2)\big)$. 
Clearly, $\phi(x_2')\neq\phi(x_1')$, $\phi(x_2')\neq\phi(1)$. It is evident that the points 
$x_1', x_2'$ can be chosen arbitrarily close to $1$. 
\end{remark}

\noindent
{\bf Condition~3.}
$\OM$ satisfies Condition~3 if there
exists a closed, totally disconnected set $S \subset \bdy\OM$ such that for every 
$p \in \bdy\OM 
\setminus S$, for every neighbourhood $U$ of $p$, and for every sequence 
$(z_n)_{n \geq 1}$ in $\OM$ with $z_n \to p$, there exist a subsequence 
$(z_{k_n})_{n \geq 1}$, an injective holomorphic map $\phi: \mathbb{D} 
\longrightarrow \OM$ that extends continuously to 
$\clos{\unitdisk}$, and arcs $Arc(x_1,\,x_2)$, $Arc(y_1,\,y_2)$ containing 
$1$ such that $Arc(x_1,\,x_2)\Subset Arc(y_1,\,y_2)$ and such that:
\begin{enumerate}
\item $\phi(1)=p$, $\phi(Arc(y_1, y_2)) \subset \partial \Omega$, $\phi(Reg(y_1, y_2)) 
\subset U$,
\smallskip

\item for all $n \in \posint$, $z_{k_n} \in \phi(Reg(x_1, x_2))$ and 
$\phi^{-1}(z_{k_n}) \to 1$ as $n \to \infty$.
\end{enumerate}
\smallskip

Observe that if a domain $\OM$ satisfies Condition~3 then it clearly satisfies Condition~2. 
It is not very difficult to see that Condition~2 also implies Condition~3. To see that, 
assume $\OM$ satisfies Condition~2. Choose a point $p\in\bdy\OM\setminus S$, where $S$ 
is as in Condition~2, let $U$ be a neighbourhood of $p$ and $(z_n)_{n\geq 1}$ be a 
sequence in $\OM$ converging to $p$. Let $\phi:\unitdisk\longrightarrow\OM$ be the injective 
holomorphic map provided by Condition~2. Then $\phi^{-1}(U)$ is an open subset of 
$\clos{\unitdisk}$ containing $1$. Therefore there exists a small $r>0$ such that 
$B(1,\,r)\cap\clos{\unitdisk}\subset\phi^{-1}(U)$ and
$B(1,\,r)\cap\bdy\unitdisk\subset Arc(x_1,\,x_2)$, where $Arc(x_1,\,x_2)$ is 
as provided by Condition~2. It is now clear that we can choose points $x'_j, y'_j, j=1,2$, 
in $B(1,\,r)\cap\bdy\unitdisk$ that are close to $1$ such that all the requirements in 
Condition~3 are satisfied. 
\smallskip

Observe that, arguing as described in Remark~\ref{Rem:distendpoints}, we can also assume, 
without loss of generality, that the points $\phi(x_j), \phi(y_j), \phi(1)$ in Condition~3 are all 
distinct for $j=1,2$. 

\begin{remark}
It is clear that if a hyperbolic domain $\OM\subset\C$ satisfies Condition~1, then it
satisfies Condition~2, and hence Condition~3. 
\end{remark}

\begin{theorem}\label{T:cond2ImplVisibility}
Let $\Omega \subset \C$ be a hyperbolic domain
satisfying Condition~2. 
Then $\Omega$ is a visibility domain.
\end{theorem}

\begin{proof}
To get a contradiction, assume that $\OM$ does not
possess the visibility property. Then, by definition, there exist $\lambda\geq 
1$, $\kappa>0$, points $p,q\in \bdy{\overline{\OM}}^{\text{End}}$, $p\neq q$, 
sequences $(p_n)_{n\geq 1}$ and $(q_n)_{n\geq 1}$ in $\OM$ converging to $p$ and 
$q$, respectively (in the topology of ${\overline{\OM}}^{\text{End}}$), and a 
sequence $(\gamma_n)_{n\geq 1}$ of $(\lambda,\kappa)$-almost-geodesics,
$\gamma_n : [a_n,b_n] \to \OM$, joining $p_n$ to $q_n$, such that 
$\gamma_n$ eventually avoids every compact set in $\OM$.
An application of \Cref{rmk:cnvg_outside_tot_disc} gives the following:
there exist a sequence $(t_n)_{n\geq 1}$, $a_n\leq t_n\leq b_n$, a $\xi\in\bdy\OM
\setminus S$, and an $r>0$ such that $(\gamma_n(t_n))_{n\geq 1}$ converges to $\xi$
and such that each $\gamma_n$ intersects $\OM\setminus D(\xi;r)$. We may suppose,
without loss of generality, that, for every $n$, there exists $t'_n$, $t_n<t'_n\leq
b_n$, such that $\gamma_n(t'_n) \in \OM \setminus D(\xi;r)$. By Condition~3, there
exist a subsequence $\big(\gamma_{k_n}(t_{k_n})\big)_{n\geq 1}$, an injective
holomorphic map $\phi:\unitdisk\to\OM$ that extends
continuously to $\clos{\unitdisk}$, and 
arcs $Arc(x_1,\,x_2)$, $Arc(y_1,\,y_2)$ containing $1$ such that
$Arc(x_1,\,x_2)\Subset Arc(y_1,\,y_2)$ and such that:
\begin{enumerate}
\item $\phi(1)=\xi$, $\phi(Arc(y_1, y_2)) \subset \partial \Omega$, $\phi(Reg(y_1, y_2)) 
\subset D(\xi, r/2)$,
\smallskip
\item for all $n \in \posint$, $\gamma_{k_n}(t_{k_n}) \in \phi(Reg(x_1, x_2))$ and 
$\phi^{-1}(\gamma_{k_n}(t_{k_n}) ) \to 1$ as $n \to \infty$.
\end{enumerate}

Moreover, we can also assume that $\phi(x_1)$, $\phi(x_2)$, $\phi(y_1)$, $\phi(y_2)$
and $\xi$ are all distinct and that, for all $x\in Arc[x_1,\,x_2]$ and $j=1,2$, 
$\phi(y_j)\neq\phi(x)$. 
\smallskip

We now suppress the subsequential notation by relabelling, and suppose that (2) above
holds for the original sequences $(\gamma_n)_{n\geq 1}$ and 
$(t_n)_{n\geq 1}$. For every $n$, we let
\begin{equation*}
	\wt{t}_n \defeq \inf\{t \in [t_n,b_n] \mid \gamma_n(t) \notin \phi(Reg(x_1,x_2))\}.
\end{equation*}
\smallskip
\noindent {\bf Claim 1.} For every $n$, $\wt{t}_n>t_n$,
$\gamma_n\big([t_n,\wt{t}_n)\big)\subset \phi(Reg(x_1,x_2))$ and 
$\gamma_n(\wt{t}_n)\in \phi((x_1,x_2)_p)$.

\noindent The inequality in the claim above follows immediately because we know,
by (2) above, that $\gamma_n(t_n)\in \phi(Reg(x_1,x_2))$, that $\phi(Reg(x_1,x_2))$
is open, and that $\gamma_n$ is continuous. The inclusion follows because of the
definition of $\wt{t}_n$ as an infimum. As for the membership relation,  $\gamma_n(\wt{t}_n) \in \big(\overline{\phi(Reg(x_1,x_2))} \setminus
\phi(Reg(x_1,x_2))\big)\cap \OM = \phi\big((x_1,x_2)_p\big)$.
\hfill $\blacktriangleleft$
\smallskip

Clearly, we may assume, without loss of generality, that 
$(\gamma_n(\wt{t}_n))_{n\geq 1}$ converges to a point
of $\phi([x_1,x_2]_p)$. However, since $(\gamma_n)_{n\geq 1}$ eventually 
avoids every compact subset of $\OM$ and since $\phi(x_1)$ and $\phi(x_2)$ 
are the only two points of $\phi([x_1,x_2]_p)\cap\bdy\OM$, it follows that 
$(\gamma_n(\wt{t}_n))_{n\geq 1}$ converges either to $\phi(x_1)$ or to 
$\phi(x_2)$.
Suppose, without loss of generality, that it converges to $\phi(x_1)$.
\smallskip

Now consider the sequence $\big( \gamma_n|_{[t_n,\wt{t}_n]} \big)_{n\geq 1}$ of 
$(\lambda,\kappa)$-almost-geodesics with respect to $\koba_{\OM}$. We shall show that
the sequence is also a sequence of $(\lambda',\kappa)$-almost-geodesics with respect to 
$\koba_{\phi(\unitdisk)}$ for some $\lambda'> 1$. To do this, we shall use the localization
lemma, but first we make the following:
\smallskip

\noindent {\bf Claim 2.} $\koba_{\OM}\big( \phi(Reg(x_1,x_2)),\OM\setminus
\phi(\unitdisk) \big) > 0$.

\smallskip

\noindent First we make the following observation:
\begin{equation}\label{E:auxclaim2}
\koba_{\OM}\big(\phi(Reg(x_1,x_2)),\OM\setminus
\phi(\unitdisk)\big)\geq \koba_{\OM}\big(\phi(Reg(x_1,x_2)),\phi(\unitdisk) 
\setminus \phi(Reg(y_1,y_2))\big).
\end{equation}

\noindent To see the above, let $x$ and $y$ be arbitrary 
points of $\phi(Reg(x_1,x_2))$ and $\OM\setminus\phi(\unitdisk)$, respectively. 
Using the completeness of $(\OM,\koba_{\OM})$, choose a $\koba_{\OM}$-geodesic
$\gamma  : [0,L]\to\OM$ such that $\gamma(0) = x \in \phi(Reg(x_1,x_2))$ and $\gamma(L) = y \in \OM \setminus \phi(\unitdisk)$.
Let
\begin{equation*}
	S \defeq \{t\in [0,L] \mid \gamma([0,t]) \subset \phi(Reg(y_1,y_2))\}.
\end{equation*}  
Further, let $t_0\defeq \sup S$. Clearly 
$t_0>0$ and it follows that $\gamma(t_0)\in \bdy\phi(Reg(y_1,y_2))\cap\OM =
\phi((y_1,y_2)_p) \subset \phi(\unitdisk) \setminus \phi(Reg(y_1,y_2))$. Now
\begin{align*}
	\koba_{\OM}(x,y) = L \geq t_0 = \koba_{\OM}(x,\gamma(t_0)) \geq
	\koba_{\OM}\big(\phi(Reg(x_1,x_2)),\phi(\unitdisk)\setminus\phi(Reg(y_1,y_2))\big). 
\end{align*}
Since $x\in Reg(x_1,x_2)$ and $y\in \OM\setminus\phi(\unitdisk)$ were
arbitrary, the inequality \eqref{E:auxclaim2} follows. 
\smallskip

Therefore, to prove Claim~2, it suffices to prove that
\begin{equation} \label{eqn:koba_sep_pos}
	\koba_{\OM}\big(\phi(Reg(x_1,x_2)),\phi(\unitdisk)\setminus\phi(Reg(y_1,y_2))\big)
	>0.	
\end{equation}
Assume, to get a contradiction, that \eqref{eqn:koba_sep_pos} does not hold. Then
there exist sequences $(u_n)_{n\geq 1}$ in $\phi(Reg(x_1,x_2))$ and 
$(v_n)_{n\geq 1}$ in $\phi(\unitdisk)\setminus\phi(Reg(y_1,y_2))$ such that 
$\koba_{\OM}(u_n,v_n)\to 0$. For each $n$, choose a $\koba_{\OM}$-geodesic 
$\sigma_n:[0,\,L_n]\lraw\OM$ joining $u_n$ and $v_n$. Also, for each
$n$, let 
\begin{align*}
	s_n &\defeq \sup\{t\in [0,L_n] \mid \sigma_n([0,t]) \subset \phi(Reg(x_1,x_2))\},
	\text{ and} \\
	t_n &\defeq \inf\{t \geq s_n \mid \sigma_n(t) \notin \phi(Reg(y_1,y_2))\}.
\end{align*}
Now note that, clearly, $\sigma_n(s_n)\in \bdy\phi(Reg(x_1,x_2))\cap\OM = 
\phi((x_1,x_2)_p)$ and that $\sigma_n(t_n)\in \bdy\phi(Reg(y_1,y_2))\cap\OM = 
\phi((y_1,y_2)_p)$. So we can find sequences $(\xi_n)_{n\geq 1}$ and 
$(\zeta_n)_{n\geq 1}$ in $(x_1,x_2)_p$ and $(y_1,y_2)_p$, respectively, such that,
for every $n$, $\phi(\xi_n)=\sigma_n(s_n)$ and $\phi(\zeta_n)=\sigma_n(t_n)$.
By compactness we may suppose that $(\xi_n)_{n\geq 1}$ and 
$(\zeta_n)_{n\geq 1}$
converge to points $\check{\xi}$ and $\check{\zeta}$ of $[x_1,x_2]_p$ and 
$[y_1,y_2]_p$, respectively. By the continuity of $\phi$ on 
$\clos{\unitdisk}$, $(\phi(\xi_n))_{n\geq 1}$ and 
$(\phi(\zeta_n))_{n\geq 1}$ converge to $\phi(\check{\xi})$ and
$\phi(\check{\zeta})$, respectively, i.e., $(\sigma_n(s_n))_{n\geq 1}$ and
$(\sigma_n(t_n))_{n\geq 1}$ converge to $\phi(\check{\xi})$ and
$\phi(\check{\zeta})$, respectively. It follows easily from the hypotheses 
on $\phi$ that $\phi(\check{\xi}) \neq \phi(\check{\zeta})$.
Note that $\sigma_n(s_n),\sigma_n(t_n)\in \mathsf{ran}(\sigma_n)$ and so, since
$\koba_{\OM}(u_n,v_n)\to 0$, it follows that $\koba_{\OM}(\sigma_n(s_n),\sigma_n(t_n)) \to 0.$
On the other hand, since $\OM$ is a hyperbolic planar domain and 
$\phi(\check{\xi})$, $ \phi(\check{\zeta})$ are distinct points of $\clos{\OM}$, we have
\[
\liminf_{(z,\,w)\to(\phi(\check{\xi}),\,\phi(\check{\zeta}))}\koba_{\OM}(z, w)>0.
\]
This is a contradiction, and hence Claim~2 follows. 
\hfill{$\blacktriangleleft$}
\smallskip

Let $k\defeq \coth{\koba_{\OM}\big(\phi(Reg(x_1,x_2),\OM\setminus\phi(\unitdisk))\big)}$ 
and let $\lambda' \defeq {k}\lambda$. Now we make the following claim.
\smallskip

\noindent {\bf Claim 3.} For every $n$, 
$\big(\gamma_n|_{[t_n,\wt{t}_n]}\big)_{n\geq 1}$ is a 
$(\lambda',\kappa)$-almost-geodesic with respect to $\koba_{\phi(\unitdisk)}$.
\smallskip

\noindent Since each $\gamma_n$ is a
$(\lambda,\kappa)$-almost-geodesic with respect to $\koba_{\OM}$, the condition on the Kobayashi--Royden metric 
together with the localization lemma gives
\begin{equation} \label{eqn:ub_kob_met_trnc}
	\forall\, n\in\posint,\, \text{for a.e. } t\in [t_n,\wt{t}_n],\; 
	\dkoba_{\phi(\unitdisk)}(\gamma_n(t))|\gamma_n'(t)| \leq k \dkoba_{\OM}(\gamma_n(t))
	|\gamma_n'(t)| \leq k\lambda = \lambda'.
\end{equation}
On the other hand, by the distance-decreasing property of the inclusion map,
\begin{equation} \label{eqn:lb_kob_dist_trnc}
	\forall\,n\in\posint,\,\forall\,s,t\in [t_n,\wt{t}_n],\; (1/\lambda)|s-t|-\kappa\leq
	\koba_{\OM}(\gamma_n(s),\gamma_n(t))\leq 
	\koba_{\phi(\unitdisk)}(\gamma_n(s),\gamma_n(t)).
\end{equation}
Finally
\begin{align}
	\forall\,n\in\posint,\,\forall\,s,t\in [t_n,\wt{t}_n],\,s\leq t,\;
	\koba_{\phi(\unitdisk)}(\gamma_n(s),\gamma_n(t)) &\leq 
	\int_{s}^{t}\dkoba_{\phi(\unitdisk)}(\gamma_n(\tau))|\gamma_n'(\tau)|d\tau \notag\\
	&\leq \lambda'|s-t| \leq \lambda'|s-t|+\kappa. \label{eqn:ub_kob_dist_trnc}
\end{align}
By \eqref{eqn:ub_kob_met_trnc}, \eqref{eqn:lb_kob_dist_trnc} and 
\eqref{eqn:ub_kob_dist_trnc}, the claim follows. 
\hfill{$\blacktriangleleft$}
\smallskip

Now consider the sequence of curves given by
\begin{equation*}
	\rho_n \defeq \phi^{-1}\circ\gamma_n|_{[t_n,\wt{t}_n]}.
\end{equation*}
Note that $(\rho_n)$ is a sequence of curves in $\unitdisk$ such that
$\rho_n \text{ joins } \phi^{-1}(\gamma_n(t_n)) 
\text{ and } \phi^{-1}(\gamma_n(\wt{t}_n))$ for each $n$.
Observe that $\big(\phi^{-1}(\gamma_n(t_n))\big)_{n\geq 1} \text{ converges to } 1$
and we may assume, without loss of generality, that
$\big(\phi^{-1}(\gamma_n(\wt{t}_n))\big)_{n\geq 1} \text{ converges to } x_1$.
Now, by the fact that $\phi$ is a biholomorphism and from Claim~3, it follows that
$\rho_n$ is a $(\lambda',\kappa)\text{-almost-geodesic}$ with respect to $\koba_{\unitdisk}$
whose end points converge to $1$ and $x_1$ respectively.
This together with the well-known fact that $\unitdisk$ is a 
visibility domain, implies that there exists a compact subset $K$ of $\unitdisk$ such that
\begin{equation*}
	\forall\,n\in\posint,\;\mathsf{ran}(\rho_n)\cap K\neq\emptyset.
\end{equation*}
Obviously then
\begin{equation*}
	\forall\,n\in\posint,\;\mathsf{ran}(\gamma_n)\cap\phi(K)\neq\emptyset.
\end{equation*}
But then $\phi(K)$ is a compact subset of $\OM$ that {\em all} the $\gamma_n$
intersect. This contradicts the initial assumption and completes the proof.
\end{proof} 

We now prove a result that shows that the Euclidean boundary of a planar
hyperbolic domain satisfying Condition~1 possesses a property that is an
analogue of a corresponding property possessed by the ideal boundary
of a (proper, geodesic) Gromov hyperbolic space (see, for example, 
\cite[Proposition~2.10]{ZimCharDomLimAut}). It is important to note that the
analogy is not perfect. Specifically, the lemma below does not deal with the
boundary of the end-compactification of $\clos{\OM}$ (which one expects to be
the true analogue of the ideal boundary), but rather with the Euclidean boundary:
in effect this means that the ends of $\clos{\OM}$ are excluded from consideration
below.

\begin{lemma}\label{L:Unbounded_Gromov_prod}
	Suppose $\OM \subset \C$ is a hyperbolic planar domain satisfying 
	Condition~1. Then, for any fixed $o \in \OM$ and every $\xi_1, \xi_2 \in 
	\bdy \OM$,
	\[ \liminf_{(z,w) \to (\xi_1,\xi_2),\, z,w\in\OM}(z|w)_o = +\infty \]
	if and only if $\xi_1 = \xi_2$.
\end{lemma}

\begin{proof}
	First note that, by Theorem~\ref{T:Visibility_Thm_in_plane}, $\Omega$ is a 
	visibility domain. It follows from Lemma~\ref{lmm:weak_visib_Grom_lim_fin}
	that if \[
	\liminf_{(z,w) \to (\xi_1,\xi_2),\,z,w\in\OM}(z|w)_o = +\infty,
	\]
	then $\xi_1 = \xi_2$.
	\smallskip
	
	Conversely, suppose that $\xi_1 = \xi_2 \defines \xi$. To get a 
	contradiction 
	suppose that there exist a point $o$ and sequences $(z_n)_{n\geq 1}$ and 
	$(w_n)_{n\geq 1}$ in $\OM$ 
	such that $z_n \to \xi$ and $w_n \to \xi$ as $n \to \infty$ and, for all $n \in 
	\posint$,
	\[
	(z_n|w_n)_o \leq  \kappa/2 < +\infty,
	\]
	for some positive constant $\kappa$. 
	Choose, for every $n\in\posint$, a $\koba_{\OM}$-geodesic $\sigma_n^1$ joining
	$o$ and $z_n$ and a $\koba_{\OM}$-geodesic $\sigma_n^2$ joining $o$ and $w_n$.
	For all $s, t \in [0, +\infty)$ and $n \in \posint$,		
	by the triangle inequality and using the fact that $\sigma_n^1$ and $\sigma_n^2$ 
	are geodesics, we have
	\begin{equation}\label{E:Bounded_Gromov_product}
		\begin{split}
			&\koba_{\Omega}(\sigma_n^1(t), o ) + \koba_{\Omega}( o, \sigma_n^2(s)) - \koba_{\Omega}(\sigma_n^1(t), \sigma_n^2(s) ) \\
			&= \koba_{\Omega}(z_n, o ) - \koba_{\Omega}(\sigma_n^1(t), z_n )  + 
			\koba_{\Omega}( o, w_n) - \koba_{\Omega}( \sigma_n^2(s), w_n) - \koba_{\Omega}(\sigma_n^1(t),    \sigma_n^2(s) )\\
			&\leq \koba_{\Omega}(z_n, o ) + \koba_{\Omega}( o, w_n) -  \koba_{\Omega}(z_n, w_n)\\
			&= 2(z_n|w_n)_o \\
			&\leq  \kappa.
		\end{split}
	\end{equation}

	By Lemma~\ref{lmm:seq_geods_conv_geod_ray},
	passing to subsequences, we may assume that $(\sigma_n^1)_{n\geq 1}$ 
	and $(\sigma_n^2)_{n\geq 1}$ converge to $\koba_{\OM}$-geodesic 
	rays $\sigma_1: [0, +\infty) 
	\longrightarrow \OM$ and $\sigma_2: [0, +\infty) \longrightarrow \OM$, 
	respectively, emanating from $o$ and landing at the boundary point $\xi$. 
	Now, define
	\[
	\sigma(t) \defeq \begin{cases}
		\sigma_1(t),  &\forall\, t \geq 0 ,\\
		\sigma_2(-t), &\forall\, t \leq 0.
	\end{cases}
	\]
	
	We claim that $\sigma$ is a $(1, \kappa)$-almost-geodesic of $(\Omega, 
	\koba_{\Omega} )$. It is easy to see that 
	\[ \forall\, t \in \R \setminus \{0\},\; \kappa_{\OM}(\sigma(t);\sigma'(t)) = 1 \]
	since the Kobayashi metric (Poincar\'e metric) of the hyperbolic
	planar domain $\Omega$ is a 
	Riemannian metric and geodesics are the unit-speed curves. For $s, t$ in $[0, 
	+\infty)$ or $( -\infty, 0]$, it can be seen easily, by definition, that 
	$\koba_{\Omega}(\sigma(t), \sigma(s) ) = |t - s|$. For $s \leq 0$ and $ t \geq 0$, 
	we have
	\begin{equation*}
		\begin{split} \koba_{\Omega}(\sigma(t), \sigma(s) ) &\leq  
		\koba_{\Omega}(\sigma(t), \sigma(0) ) + \koba_{\Omega}(\sigma(0), \sigma(s) )\\
			&= t - s = |t - s|,
		\end{split}
	\end{equation*}
	and 
	\begin{equation*}
		\begin{split} \koba_{\Omega}(\sigma(t), \sigma(s) )
			&=\koba_{\Omega}(\sigma_1(t), \sigma_2(-s) )\\
			&= \lim_{n \to \infty} \koba_{\Omega}(\sigma_n^1(t), \sigma_n^2(-s) )\\
			&= \lim_{n \to \infty} \big\{ \koba_{\Omega}(\sigma_n^1(t), o ) + \koba_{\Omega}( o, \sigma_n^2(-s)) \\
			&- \left( \koba_{\Omega}(\sigma_n^1(t),o) + \koba_{\Omega}(o, \sigma_n^2(-s)) - \koba_{\Omega}(\sigma_n^1(t),\sigma_n^2(-s))\right) \big\} \\
			&\geq  \lim_{n \to \infty} ( \koba_{\Omega}(\sigma_n^1(t), o ) + \koba_{\Omega}( o, \sigma_n^2(-s)) - \kappa ) \quad (\text{by (\ref{E:Bounded_Gromov_product})})\\
			&= \koba_{\Omega}(\sigma_1(t), o ) + \koba_{\Omega}( o, \sigma_2(-s)) - \kappa \\
			&= t - s - \kappa = |t - s| - \kappa.
		\end{split}
	\end{equation*}
	Hence $\sigma$ is a $(1, \kappa)$-almost-geodesic of $(\Omega,\koba_{\Omega})$.
	Note that $\lim_{t\to\infty} \sigma(t) = \lim_{t\to\infty} \sigma(-t) = \xi$.
	\smallskip
	
	By Condition~1, there exists a topological embedding 
	$\tau:\clos{\unitdisk}\to\clos{\OM}$ such that $\tau(\unitdisk)\subset\OM$ and such
	that every sequence $(z_n)_{n\geq 1}$ in $\OM$ converging to $\xi$ lies eventually 
	in $\tau(\unitdisk)$. By the remark made right after the statement of Condition~1, 
	there exists a neighbourhood
	$U$ of $\xi$ in $\C$ such that $U\cap\OM \subset \tau(\unitdisk)$. We choose a
	neighbourhood $W$ of $\xi$ in $\C$ such that $W\Subset U$. After this,
	we choose $T_0<\infty$ such that, for all $t\in\R$ with $|t| \geq T_0$, 
	$\sigma(t)\in W$. It is immediate that 
	$d_{Euc}(W\cap\OM,\OM\setminus\tau(\unitdisk))>0$, where $d_{Euc}$ is 
	the Euclidean distance. From this, \Cref{res:hyp_plan_dom_BSP} and \Cref{lmm:kob_sep_intrnl_dom}, 
	it now follows easily that $\koba_{\OM}(W \cap \OM, \OM \setminus \tau(\unitdisk)) > 0$ 
	and hence that $\coth\koba_{\OM}(W \cap \OM, \OM \setminus \tau(\unitdisk)) \defines \lambda<\infty$.
	Hence, by \Cref{rmk:weak_loc_almost-geod}, 
	$\sigma|_{[T_0, \infty)}$ and $\sigma|_{(-\infty, -T_0]}$ are 
	$(\lambda, \kappa)$-almost-geodesics of $(\tau(\mathbb{D}), 
	\koba_{\tau(\mathbb{D})})$.
	\smallskip
	
	Choose a biholomorphism $\phi:\unitdisk\to\tau(\unitdisk)$. By 
	the Carath{\'e}odory Extension Theorem $\phi$ extends to a homeomorphism from 
	$\overline{\mathbb{D}}$ onto $\tau(\overline{\mathbb{D}})$.  Choose a continuous
	curve $\sigma_0 : [-T_0,T_0] \lraw \tau(\unitdisk)$ such that
	$\sigma_0(-T_0) = \sigma(-T_0)$ and $\sigma_0(T_0) = \sigma(T_0)$.
	
	Now, define $\gamma : \R \lraw \tau(\unitdisk)$ by
	\[
	\gamma(t) \defeq 
	\begin{cases}
		\sigma(t),    &\forall\, t \,\text{such that}\, |t|\geq T_0, \\
		\sigma_0(t),  &\forall\, t \,\text{such that} \,|t|\leq T_0.
	\end{cases}
	\]
	Now, we show that $\gamma$ is a $(\lambda_1, \kappa_1)$-quasi-geodesic loop of 
	$(\tau(\mathbb{D}),\koba_{\tau(\mathbb{D})})$ at the point 
	$\xi \in \bdy\tau(\mathbb{D})$ for some $\lambda_1 \geq 1$ and $\kappa_1 > 0$. To 
	see this we proceed as follows.\\
	
	\noindent{\bf Case 1.}
	For $s \leq -T_0$ and $t \geq T_0$, by the monotonicity of the Kobayashi distance 
	with respect to inclusions, and using the fact that $\sigma$ is a 
	$(1,\kappa)$-almost-geodesic of $(\OM, \koba_{\OM})$, we have
	\begin{equation*}
		\koba_{\tau(\mathbb{D})}(\gamma(t),\gamma(s)) \geq 
		\koba_{\Omega}(\sigma(t),\sigma(s)) \geq |t - s| - \kappa,
	\end{equation*}
	and by the triangle inequality, definition of $\gamma$ and using the fact that 
	$\sigma|_{[T_0,\infty)}$ and $\sigma|_{(-\infty, -T_0]}$ are $(\lambda, 
	\kappa)$-almost-geodesics of $(\tau(\mathbb{D},\koba_{\tau(\mathbb{D})})$, 
	we have
	\begin{equation*}
		\begin{split} \koba_{\tau(\mathbb{D})}(\gamma(t), \gamma(s) )
			&\leq  \koba_{\tau(\mathbb{D})}(\sigma(s), \sigma(-T_0)) + 
			\koba_{\tau(\mathbb{D})}(\sigma_0(-T_0),\sigma_0(T_0)) +  
			\koba_{\tau(\mathbb{D})}(\sigma(T_0),\sigma(t))\\
			&\leq  \lambda(-T_0-s) + \lambda(t-T_0)  + 
			\text{diam}_{\koba_{\tau(\mathbb{D})}}(\sigma_0)+2\kappa\\
			&\leq \lambda(t-s)  + \text{diam}_{\koba_{\tau(\mathbb{D})}}(\sigma_0)+ 2\kappa.
		\end{split}
	\end{equation*}
	
	\noindent{\bf Case 2.}
	For $s, t \geq T_0$ or $s, t \leq -T_0$, and $s\leq t$, by the monotonicity of the 
	Kobayashi distance with respect to inclusions, and using the fact that $\sigma$ is 
	a $(1,\kappa)$-almost-geodesic of $(\Omega,\koba_{\Omega})$, we have
	\begin{equation*}
		\koba_{\tau(\mathbb{D})}(\gamma(t),\gamma(s)) = 
		\koba_{\tau(\mathbb{D})}(\sigma(s),\sigma(t))  
		\geq \koba_{\Omega}(\sigma(t), \sigma(s) ) \geq |t - s| - \kappa.
	\end{equation*}
	Further, by the definition of $\gamma$ and using the fact that 
	$\sigma|_{[T_0,\infty)}$ and $\sigma|_{(-\infty, -T_0]}$ are $(\lambda, 
	\kappa)$-almost-geodesics of $(\tau(\mathbb{D},\koba_{\tau(\mathbb{D})})$, 
	we have
	\begin{equation*}
		\koba_{\tau(\mathbb{D})}(\gamma(s),\gamma(t)) \leq  
		\koba_{\tau(\mathbb{D})}(\sigma(s),\sigma(t)) \leq \lambda(t-s)+\kappa.
	\end{equation*}
	
	\noindent{\bf Case 3.}
	For $-T_0\leq s\leq t\leq  T_0$, the following inequalities follow easily:
	\[
	(t-s)-2T_0 \leq 0 \leq \koba_{\tau(\mathbb{D})}(\gamma(s),\gamma(t)) \leq  
	(t-s)+\text{diam}_{\koba_{\tau(\mathbb{D})}}(\sigma_0).
	\]
	
	\noindent{\bf Case 4.}
	For $s\leq -T_0\leq t\leq T_0$ or $-T_0\leq s\leq T_0\leq t$, the following
	inequalities follow easily:
	\begin{align*}
		\frac{1}{\lambda}(t-s)-\kappa-\text{diam}_{\koba_{\tau(\mathbb{D})}}(\sigma_0)-
		2T_0 &\leq \koba_{\tau(\mathbb{D})}(\gamma(s),\gamma(t)) \\
		&\leq \lambda(t-s)+\kappa+\text{diam}_{\koba_{\tau(\mathbb{D})}}(\sigma_0). 
	\end{align*}
	
	\noindent From the consideration of the four cases above, it is now clear that there exists
	$\kappa_1>0$ such that
	$\gamma$ is a $(\lambda,\kappa_1)$-quasi-geodesic loop at the boundary point 
	$\xi$ of $\tau(\mathbb{D})$.
	\smallskip
	
	Since the map $\phi^{-1} : \tau(\mathbb{D}) \lraw \mathbb{D}$ is a 
	Kobayashi isometry that extends up to the boundary as a homeomorphism, $\phi^{-1} 
	\circ \gamma :\mathbb{R}\lraw \mathbb{D}$ is a 
	$(\lambda,\kappa_1)$-quasi-geodesic loop at the boundary point $\phi^{-1}(\xi)$
	of $\unitdisk$. This is a 
	contradiction because $(\mathbb{D},\koba_{\mathbb{D}})$ is a Gromov hyperbolic 
	distance 
	space for which the Euclidean boundary can be identified with the Gromov 
	boundary, and hence there does not exist any quasi-geodesic loop at any boundary 
	point of $\mathbb{D}$.
	\end{proof}	


\subsection{Simply Connected Visibility Domains}\label{ss:simp_conn_visib}
The main result of this subsection gives a characterization of simply connected visibility domains. To prove this 
result we need two lemmas below. These results will also be an important tool in the next subsection in which 
we provide a characterization of domains that satisfy Condition~2 and Condition~1 respectively.
\smallskip

In what follows, $\C_\infty$ will denote the Riemann sphere, i.e., the one-point
compactification of $\C$, i.e. $\C\cup\{\infty\}$. Also, for a domain $\OM$ in $\C$, $\cl_{\C_\infty}(\OM)$ and 
$\bdy_{\C_\infty}\OM$ denote the closure and the boundary of $\OM$ in $\C_\infty$ respectively.
\begin{lemma} \label{lmm:simp_conn_loc_conn}
		Let $\OM\subset\C_\infty$ be a hyperbolic simply connected domain. Suppose that there is a totally disconnected subset
		$S\subset\bdy_{\C_\infty}\OM\cap\C$ such that $\bdy_{\C_\infty}\OM$
		is locally connected at every point of $(\bdy_{\C_\infty}\OM\cap\C)\setminus S$.
		Then $\bdy_{\C_\infty}\OM$
		is locally connected everywhere.
	\end{lemma}
	
\begin{proof}
	Since $\OM$ is hyperbolic and simply connected, choose a biholomorphism $\phi:\unitdisk\to\OM$. We wish to show that $\phi$ extends to
	a continuous map from $\clos{\unitdisk}$ to $\cl_{\C_\infty}(\OM)$. To do this, it is sufficient (see \cite[Chapter~2]{Pomme})
	to show that for every
	prime end $p$ of $\OM$, the impression $I(p)$ of $p$ is trivial (i.e., is a singleton). To get a contradiction, we assume that
	there exists a prime end $p$ of $\OM$ such that $I(p)$ is not a singleton. We know that $I(p)$ is a compact, connected set that is
	contained in $\bdy_{\C_\infty}\OM$ (see, for example, \cite[Section~2.5]{Pomme}); and by assumption, it contains more than one point. By 
	\cite[Theorem~1.1]{Rempe}, there are at most two points of $I(p)$ where $\C_\infty\setminus\OM$ is locally connected. Note that 
	$I(p)\setminus\{\infty\}\not\subset S$. The reason is that, since $I(p)$ is compact, connected, and contains more than one point,
	$I(p)\setminus\{\infty\}$ is not totally disconnected, whereas $S$ is. Therefore, choose $x_0\in I(p)\setminus(S\cup\{\infty\})$. Since 
	$x_0\in \bdy_{\C_\infty}\OM\setminus S$, $\bdy_{\C_\infty}\OM$ is locally connected at $x_0$ and, hence, so is $\C_\infty\setminus\OM$. 
	Now, note that $I(p)\setminus \{\infty,x_0\}$ is non-empty and not totally disconnected. Choose $y_0\in I(p)\setminus\{\infty,x_0\}$
	such that the connected component $C$ of $I(p)\setminus\{\infty,x_0\}$ containing $y_0$ is not a singleton. Once again, $C\not\subset
	S$, so there exists $z_0\in C\setminus S$. Since $z_0\in \bdy_{\C_\infty}\OM\setminus S$, $\bdy_{\C_\infty}\OM$ is locally connected 
	at $z_0$ and, hence, so is $\C_\infty\setminus\OM$.
	Repeating this procedure, we find $v_0\in I(p)\setminus\{\infty,x_0,z_0\}$ where $\C_{\infty}\setminus\OM$ is also locally connected. 
	Therefore, we have three distinct points of $I(p)$,
	namely $x_0,z_0$ and $v_0$, where $\C_\infty\setminus\OM$ is locally connected. This contradicts \cite[Theorem~1.1]{Rempe}.
	Thus $I(p)$ must be a singleton; since the prime end $p$ of $\OM$ was arbitrary, this shows that $\phi$ extends to a continuous map
	from $\clos{\unitdisk}$ to $\cl_{\C_\infty}(\OM)$, which we continue to denote by $\phi$, and which maps $\bdy\unitdisk$ continuously
	to $\bdy_{\C_\infty}\OM$. The latter set is, therefore, locally connected (everywhere), 
	as required.
\end{proof}

\begin{lemma} \label{lmm:pln_ext_Kob_isom_gen}
	Let $\OM_1,\OM_2\subset\C$ be hyperbolic domains. Suppose that $\OM_1$ is bounded and satisfies Condition~1
	and that $\OM_2$ is a weak visibility domain.
	Then any biholomorphism (indeed, any Kobayashi-isometric embedding) from $\OM_1$ to $\OM_2$ extends to a continuous map from $\clos{\OM}_1$ to
	$\cl_{\C_\infty}(\OM_2)$.	
\end{lemma}

\begin{proof}
	Choose and fix a point $o\in\OM_1$.
	Now suppose that $f:\OM_1\to\OM_2$ is a Kobayashi-isometric embedding. In order to show
	that $f$ extends to a continuous map from $\clos{\OM}_1$ to $\cl_{\C_\infty}(\OM_2)$, it suffices to show that, for every
	$\xi\in\bdy\OM_1$, $\lim_{z\to\xi,\,z\in\OM_1}f(z)$ exists as an element of $\bdy_{\C_\infty}\OM_2$. Assume, to get
	a contradiction, that the stated limit does not exist. Then, by the compactness of $\cl_{\C_\infty}(\OM_2)$, there exist
	sequences $(z_n)_{n\geq 1}$ and $(w_n)_{n\geq 1}$ in $\OM_1$ converging to $\xi$ such that $(f(z_n))_{n\geq 1}$ and
	$(f(w_n))_{n\geq 1}$ converge to distinct points $\zeta_1,\zeta_2$ of $\bdy_{\C_\infty}\OM_2$.
	Since one of the points $\zeta_1, \zeta_2$ must belong to $\bdy\OM_2$,
	we can use arguments as in the proof of Theorem~\ref{thm-totally disconnected} to conclude 
	that the aforementioned points satisfy visibility property. Therefore, by Lemma~\ref{lmm:weak_visib_Grom_lim_fin} we have
	\[ \limsup_{n\to\infty} (f(z_n)|f(w_n))_{f(o)} < \infty. \]
	Since $f$ is assumed to be a Kobayashi isometric embedding,
	\[ (f(z_n)|f(w_n))_{f(o)} = (z_n|w_n)_o. \]
	By Lemma~\ref{L:Unbounded_Gromov_prod}, 
	\[ \lim_{n\to\infty} (z_n|w_n)_o = +\infty. \]
	Thus, we have a contradiction, which completes the proof.
\end{proof}

We now present the main result of this subsection.

\begin{theorem} \label{thm:vis_simp_conn_loc_conn}
	Let $\OM\subset\C$ be a hyperbolic simply connected domain.
	If, for some totally disconnected subset $S$ of $\bdy\OM$ (not necessarily closed), every pair of distinct 
	points of $\bdy\OM\setminus S$ possesses the visibility property with respect to $\koba_\OM$, then $\bdy\OM$ 
	is locally connected and $\bdy_{\C_\infty}\OM$ is a continuous surjective image of $S^1$. Conversely, if 
	$\bdy\OM$ is locally connected, then $\OM$ is a visibility domain.
\end{theorem}

\begin{proof}
	First suppose that for some totally disconnected (not necessarily closed) subset $S$ of $\bdy\OM$, every pair of distinct
	points of $\bdy\OM\setminus S$ possesses the visibility property with respect to $\koba_\OM$. Then, by 
	Theorem~\ref{thm-totally disconnected}, $\OM$ is a visibility domain. Since $\OM$ is hyperbolic and simply connected,
	we may choose a biholomorphism $f$ from $\unitdisk$ to $\OM$. By Lemma~\ref{lmm:pln_ext_Kob_isom_gen}, $f$ extends
	to a continuous map $\wt{f}$ from $\clos{\unitdisk}$ to $\cl_{\C_\infty}(\OM)$. By compactness and continuity, this map is surjective
	(and maps $\bdy\unitdisk$ onto $\bdy_{\C_\infty}\OM$). 
	Therefore, $\bdy_{\C_\infty}\OM$,
	being the continuous image of the compact, locally connected space $S^1$, is locally connected. Note that $\bdy_{\C_\infty}\OM$
	is either equal to $\bdy\OM$ (if $\bdy\OM$ is bounded) or equal to $\bdy\OM\cup\{\infty\}$ (if $\bdy\OM$ is unbounded). 
	In either case, $\bdy\OM$ is an open subset of $\bdy_{\C_\infty}\OM$, and is therefore itself locally connected. 
        \smallskip
	
	Conversely, suppose that $\bdy\OM$ is locally connected. Then it follows by
 Lemma~\ref{lmm:simp_conn_loc_conn} that
	$\bdy_{\C_\infty}\OM$ is also locally connected. We may now regard $\OM$ as a simply connected hyperbolic domain in $\C_\infty$
	with locally connected boundary, and we may choose a biholomorphism $f:\unitdisk\to\OM$. By (a version of) Carath{\'e}odory's 
	theorem on the extension of biholomorphisms (see \cite[Chapter~2, Section~1]{Pomme}), it follows that $f$ extends to a 
	continuous map from $\clos{\unitdisk}$ to $\cl_{\C_\infty}(\OM)$. Note that the latter set may not be homeomorphic to
	$\clos{\OM}^{End}$. Nevertheless, the proof method of Theorem~\ref{thm:cont_surj_im_vis_dom_vis} readily implies 
	that $\OM$ is a visibility domain.
\end{proof}

\subsection{Characterization of domains that satisfy Condition~2 and applications}

\begin{definition}
Let $\OM\subset\C$ be a domain and let $\gamma$ be a connected component of $\bdy\OM$. Let
$K_\gamma$ be the connected component of $\C\setminus\OM$ that contains $\gamma$. Define
\begin{equation*}
  \OM_\gamma\defeq
  \begin{cases}
    \C\setminus K_\gamma, &\text{if $K_\gamma$ is unbounded as a subset of $\C$},\\
    \C_\infty\setminus K_\gamma, &\text{if $K_\gamma$ is bounded as a subset of $\C$}.
  \end{cases}
\end{equation*}
In either case, we regard $\OM_\gamma$ as a subset of $\C_\infty$.
\end{definition}
Given a domain $\OM$ in $\C$ and a connected component $K$ of $\C\setminus\OM$, it is not difficult to show that 
$\bdy{K}\subset\bdy{\OM}$ and $\C\setminus K$ is connected. Moreover, 
it is a fact \cite[Theorem~14.5, page~124]{Newman} that $K$ contains 
{\bf only} one component of $\bdy{\OM}$. Therefore the components of $\C\setminus\OM$ and $\bdy{\OM}$ are in 
one-to-one correspondence with each other in such a way that given a component $K$ of $\C\setminus\OM$, there 
is a unique component $\gamma$ of $\bdy{\OM}$ such that
\[
\bdy{K}=\gamma.
\]
This fact helps us to prove the following 
important result. 

\begin{proposition}
	Let $\OM\subset\C$ be a domain and let $\gamma\subset\bdy\OM$ be a connected component. Then
	$\OM_\gamma$ is a simply connected domain in $\C_\infty$ such that 
	\[
	 \bdy_{\C_\infty}{\OM_\gamma}\cap\C=\gamma=\bdy{K_\gamma}.
	 \]
\end{proposition}

\begin{proof}
	Note that from the definition of $\OM_\gamma$ and the fact that $\C\setminus K_\gamma$ is connected, 
	it follows that $\OM_\gamma$ is a connected open set in $\C_\infty$. Note that
	$\C_\infty\setminus\OM_\gamma$ is ${K_\gamma}$ when ${K_\gamma}$ is bounded and it is 
	${K_\gamma}\cup\{\infty\}$ when $K_\gamma$ is unbounded; we see that $\C_\infty\setminus\OM_\gamma$
	is connected in either case. Therefore $\OM_\gamma$ is simply connected.
	\smallskip
	
	Note that $\bdy_{\C_\infty}{\OM_\gamma}\cap\C=\bdy{K_\gamma}$. Since $\gamma\subset\bdy{K_\gamma}$,
	it follows from the discussion above that $\bdy{K_\gamma}=\gamma$. This establishes the result. 
	\end{proof}

Now we have an important technical lemma.

\begin{lemma} \label{lmm:loc_conn_outside_a_pt_loc_conn}
Let $\gamma\subset\C_\infty$ be a closed, connected set, let $a\in\gamma$, and suppose that
$\gamma$ is locally connected at every point other than $a$. Then $\gamma$ is locally 
connected at $a$ as well.
\end{lemma}

\begin{proof}
By using the inversion self-homeomorphism of $\C_\infty$, it is clear that it is enough to
prove the result assuming that $a\in\C$. So we assume that $a\in\C$ and that $\gamma\subset
\C$ is a closed, connected set in $\C$ that is locally connected at all points other than
$a$. 

To prove that $\gamma$ is locally connected at $a$ as well, we suppose that $U$ is a
neighbourhood of $a$ in $\C$. We choose $r>0$ such that $B(a;r)\subset U$. We first make the
following

\noindent {\bf Claim~1.} There does not exist any connected component $\sigma$ of $B(a;r)
\cap \gamma$ such that $\clos{\sigma}\subset B(a;r)$ and such that $a\notin\sigma$.

\noindent {\itshape Proof of Claim~1.} Assume, to get a contradiction, that there does exist
such a component $\sigma$. Note that since $\clos{\sigma}\subset B(a;r)$ (by assumption) and
$\sigma$ is closed in $B(a;r)$ (since $\gamma$ is closed in $\C$ and since $\sigma$ is a 
connected component of $B(a;r)\cap\gamma$), it follows that $\sigma$ is closed in $\C$, i.e.,
$\clos{\sigma}=\sigma$.
Now, there cannot exist any neighbourhood $V$ of $\sigma$
in $B(a;r)$ (or, equivalently, in $\C$) such that $V\cap\gamma=\sigma$. For, suppose that
there exists such a neighbourhood $V$. Since, by the above, $\sigma$ is a compact set,
it follows that we may choose a neighbourhood $W$ of $\sigma$ such that $W\Subset V$.
Clearly, one also has $W\cap\gamma=\sigma$. Now, $W\cap\gamma$ and $(\C\setminus\clos{W})\cap
\gamma$ are disjoint, non-empty open subsets of $\gamma$ and, clearly, their union is $\gamma$.
This shows that $\gamma$ is disconnected, which is a contradiction. So, there indeed doesn't
exist any neighbourhood $V$ of $\sigma$ such that $V\cap\gamma=\sigma$. Therefore, for every
neighbourhood $V$ of $\sigma$, $(V\cap\gamma)\setminus\sigma\neq\emptyset$. In particular, once
again by the compactness of $\sigma$, there exists a sequence $(y_n)_{n\geq 1}$ in 
$\gamma\setminus\sigma$ that converges to some point $p\in\sigma$. Now, each $y_n$ belongs to
{\em some} connected component $C_n$ of $B(a;r)\cap\gamma$. First assume that the total number
of distinct $C_n$'s is finite, that is,
\begin{equation} \label{eqn:Cns_set} 
\{C_n\mid n\in\posint\} 
\end{equation} 
is a finite set. This means that there exist some connected component $\wh{C}$ of 
$B(a;r)\cap\gamma$ and some sequence $(z_n)_{n\geq 1}$ in it (a subsequence of $(y_n)_{n\geq 1}$)
that converges to $p$. Since, as above, $\wh{C}$ is closed in $B(a;r)$, it follows that $p\in
\wh{C}$. But that is a contradiction because $p\in\sigma$, and $\sigma$ and $\wh{C}$ are distinct
connected components of $B(a;r)\cap\gamma$. So \eqref{eqn:Cns_set} must be infinite, i.e., there
exist infinitely many distinct $C_n$'s. Therefore we may assume, without loss of generality,
that $(C_n)_{n\geq 1}$ is a sequence of pairwise distinct components of $B(a;r)\cap\gamma$. Since
$p\in\sigma$ and hence $p\neq a$, $\gamma$ is, by assumption, locally connected at $p$. 
Therefore, there exists a neighbourhood $W$ of $p$ that is included in $B(a;r)$ and such that
$W\cap\gamma$ is connected. Since $(y_n)_{n\geq 1}$ converges to $p$, there exists 
$n_0\in\posint$ such that, for all $n\geq n_0$, $y_n\in W$. In particular, $y_{n_0},y_{n_0+1}\in
W$. So, $y_{n_0},y_{n_0+1}\in W\cap\gamma$. Since $W\cap\gamma$ is connected, this implies that
$y_{n_0},y_{n_0+1}$ belong to the same connected component of $B(a;r)\cap\gamma$. But this is a
contradiction because they belong to the {\em distinct} components $C_{n_0}$ and $C_{n_0+1}$ of
$B(a;r)\cap\gamma$. This contradiction establishes Claim~1.\hfill$\blacktriangleleft$

Our next observation is as follows.

\noindent {\bf Claim~2.} For every $\delta\in (0,r)$, only finitely many components of
$B(a;r)\cap\gamma$ intersect $\bdy B(a;\delta)$.

\noindent {\itshape Proof of Claim~2.} Assume, to get a contradiction, that for some $\delta\in
(0,r)$, infinitely many components of $B(a;r)\cap\gamma$ intersect $\bdy B(a;\delta)$. In 
particular, there exists a sequence $(C_n)_{n\geq 1}$ of distinct components of $B(a;r)\cap
\gamma$ such that, for every $n$, $C_n\cap\bdy B(a;\delta)\neq\emptyset$. Now exactly the same
argument as that at the end of the proof of Claim~1 shows that we have a contradiction to the
local connectedness of $\gamma$ at points other than $a$. This contradiction completes the proof.
\hfill$\blacktriangleleft$

Write
\[ \sigma_0 \defeq \text{the connected component of $B(a;r)\cap\gamma$ that contains $a$}, \]
and write $\mathscr{C}$ for the set of connected components of $B(a;r)\cap\gamma$. Our third
claim is:

\noindent {\bf Claim~3.} $\sigma_0\cap\clos{\left(\bigcup_{\sigma\in\mathscr{C}\setminus
\{\sigma_0\}}\sigma\right)}=\emptyset$.

\noindent {\itshape Proof of Claim~3.} Assume, to get a contradiction, that
\[ \sigma_0 \cap \clos{\left(\bigcup_{\sigma\in\mathscr{C}\setminus\{\sigma_0\}}\sigma\right)} 
\neq \emptyset. \]
Then there exist a sequence $(x_n)_{n\geq 1}$ in 
$\cup_{\sigma\in\mathscr{C}\setminus\{\sigma_0\}}\sigma$ and a point $y_0\in\sigma_0$ such that
$x_n\to y_0$. Note that, since $\sigma_0\subset B(a;r)$, $y_0\in B(a;r)$. Choose $\epsilon_0>0$
such that $\clos{B}(y_0;\epsilon)\subset B(a;r)$ and write $\delta\defeq |y_0-a|+\epsilon$; then
$\clos{B}(a;\delta)\subset B(a;r)$ and $0<\delta<r$. Since $(x_n)_{n\geq 1}$ is a sequence in
$\cup_{\sigma\in\mathscr{C}\setminus\{\sigma_0\}}\sigma$, for every $n$ there exists a connected
component $C_n$ of $B(a;r)\cap\gamma$ other than $\sigma_0$ such that $x_n\in C_n$. We may assume 
that, for every $n$, $x_n\in B(y_0;\epsilon)\subset B(a;\delta)$. Note that $a\notin C_n$.
Consequently, by Claim~1, $\clos{C}_n\cap\bdy B(a;r)\neq\emptyset$. Therefore, $C_n$ is a
connected set that intersects both $B(a;\delta)$ and $\bdy B(a;r)$; hence it must also intersect
$\bdy B(a;\delta)$. So, by Claim~2, there cannot be infinitely many distinct $C_n$, i.e., 
$\{C_n\mid n\in\posint\}$ must be finite. Thus, there must exist some connected component $C$ of
$B(a;r)\cap\gamma$ {\em other than} $\sigma_0$ such that $x_n\in C$ for infinitely many $n$. 
Suppose, without loss of generality, that $x_n\in C$ for {\em all} $n$. Then, as $C$ is closed in
$B(a;r)$ and as $x_n\to y_0\in B(a;r)$, it follows that $y_0\in C$. But that is an immediate
contradiction because $y_0\in\sigma_0$ and $C\cap\sigma_0=\emptyset$. This contradiction 
completes the proof. \hfill $\blacktriangleleft$

Write 
\[ V\defeq B(a;r)\setminus\clos{\bigcup_{\sigma\in \mathscr{C}\setminus\{\sigma_0\}}\sigma}; \]
then, by Claim~3, $V$ is a neighbourhood of $\sigma_0$ in $B(a;r)$. Furthermore, clearly,
$V\cap\gamma=\sigma_0$. Thus, we have a neighbourhood of $a$, namely $V$, contained in $B(a;r)$,
whose intersection with $\gamma$ is connected (and is equal to $\sigma_0$). This proves the local
connectedness of $\gamma$ at $a$. The proof is complete.
 
\end{proof}

\begin{lemma} \label{lmm:cond-2-nbd-bhvr}
Suppose that $\OM\subset\C$ is a hyperbolic domain satisfying Condition~2 and let $S$ be the
associated totally disconnected subset of $\bdy\OM$. Given
$p\in\bdy\OM\setminus S$ denote by $\gamma$ and $K_\gamma$ the connected component
of $\bdy\OM$ and the connected component of $\C\setminus\OM$ containing $p$, respectively.
Then there exists a neighbourhood $V$ of $p$ in $\C$ such that
\[ 
 V = (V\cap\OM)\cup (V\cap K_\gamma). 
  \]
\end{lemma}
\begin{proof}
Assume, to get a contradiction, that there does not exist any such $V$. Then, for every
	$n\in\posint$, there exists $x_n\in D(p;1/n)\setminus(\OM\cup K_\gamma)$.
	\smallskip
	
	\noindent{\bf Case~1.} $x_n\in\bdy\OM$ for infinitely many $n$.
	
	\noindent Without loss of any generality, we assume that $x_n\in\bdy\OM$ for all 
	$n$. Since $x_n\notin K_\gamma$, there exists $r_n>0$ such that $D(x_n;r_n)\cap K_\gamma=
	\emptyset$ for all $n$. As $x_n\to p\in K_\gamma$, it follows that $r_n\to 0$ as $n\to
	\infty$. For every $n$,
	choose $z_n\in D(x_n;r_n)\cap\OM$; clearly, $z_n\to p$. Since $p\in\bdy\OM\setminus S$ and
	$(z_n)_{n\geq 1}$ is a sequence in $\OM$ converging to $p$, by Condition~2, there exist an
	injective holomorphic map $\phi:\unitdisk\to\OM$ that extends to a continuous map from
	$\clos{\unitdisk}$ to $\clos{\OM}$, points $\xi_1,\xi_2\in S^1$, $\xi_1\neq\xi_2$, and a 
	subsequence $z_{k_n}$ of $z_n$ such that $\phi(1)=p$, $\phi\big(Arc[\xi_1,\xi_2]\big)\subset\bdy\OM$, 
	and such that $z_{k_n} \in \phi(Reg(\xi_1,\xi_2))$ for every $n$. This 
	clearly implies that $\phi\big(Arc[\xi_1,\xi_2]\big)\subset\gamma\subset K_\gamma$ and hence 
	$x_{n}\notin \phi(\clos{Reg(\xi_1,\xi_2)})=\clos{\phi(Reg(\xi_1,\xi_2))}$. 
	\smallskip
	
	Note that, since $p\notin
	\phi([\xi_1,\xi_2])$, there exists $r>0$ such that $D(p;2r)\cap\phi([\xi_1,\xi_2])=\emptyset$. 
	Therefore, for all $n$ large enough, $z_{k_n}\in D(p;r)\cap\phi(Reg(\xi_1,\xi_2))$. We may
	also suppose that $D(x_{k_n};r_{k_n})\subset D(p;r)$ for all such $n$.
	Now, for
	all $n$, join $x_{k_n}$ and $z_{k_n}$ by the radial line $L_n$
	in $D(x_{k_n};r_{k_n})$.
	Note $L_n$ is a connected set that intersects both
	$\phi(Reg(\xi_1,\xi_2))$ and $\C\setminus\clos{\phi(Reg(\xi_1,\xi_2))}$. Consequently, it must
	intersect $\bdy\phi(Reg(\xi_1,\xi_2))\subset\phi(\bdy Reg(\xi_1,\xi_2))$, i.e., for every $n$,
	there exists $u_n\in\bdy Reg(\xi_1,\xi_2)$ such that $\phi(u_n)\in L_n$. Now, $u_n$ cannot be
	in $(\xi_1,\xi_2)$ because then $\phi(u_n)$ would be in $\phi([\xi_1,\xi_2])$ whereas one has
	$\phi(u_n)\in L_n\subset D(x_{k_n};r_{k_n})\subset D(p;r)$ and $D(p;r)\cap\phi([\xi_1,\xi_2])
	=\emptyset$. Therefore $u_n\in Arc[\xi_1,\xi_2]$ and hence $\phi(u_n)\in\phi(Arc[\xi_1,\xi_2])
	\subset \gamma$, which is again a contradiction because $\phi(u_n)\in L_n\subset 
	D(x_{k_n};r_{k_n})$ and $D(x_{k_n};r_{k_n})\cap\gamma=\emptyset$. Thus, Case~1 cannot arise.
	\smallskip
	
	\noindent{\bf Case~2.} There are only finitely many $n$ such that $x_n\in\bdy\OM$.
	
	\noindent We may assume, without loss of any generality that 
	$x_n\in\C\setminus (\clos{\OM}\cup K_\gamma)$, for all $n$.
	For every $n$, choose $z_n\in D(p;1/n)\cap\OM$ and join $z_n$ to $x_n$ by the line segment 
	$L_n$. Then $L_n\subset D(p;1/n)$ and $L_n$ can be parametrized in this way: $\theta_n\defeq
	t\mapsto z_n+t(x_n-z_n) : [0,1]\to D(p;1/n)$. For every $n$, let
	\[ t_n\defeq \sup\{t\in[0,1]\mid \theta_n(t)\in\OM\}. \]
	Then, for every $n$, $t_n<1$ (because $\theta_n(1)=x_n\notin\clos{\OM}$) and $\theta_n(t_n)
	\in\bdy\OM$. We claim that $\theta_n(t_n)\notin K_\gamma$ for all $n$. This is because 
	$\theta_n([t_n,1])\subset\C\setminus\OM$ is a connected set. Hence, if $\theta_n(t_n)\in
	K_\gamma$, then $\theta_n([t_n,1])\subset K_\gamma$, so also $\theta_n(1)=x_n\in K_\gamma$, which is
	a contradiction. Thus, for all $n$, $\theta_n(t_n)\in\bdy\OM\setminus K_\gamma$ and $\theta_n(t_n)$ 
	converges to $p$. But by Case~1, this is impossible.
	\smallskip
	
	Therefore, we have a contradiction in each case whence the result. 
\end{proof}

\begin{lemma} \label{lmm:dom_cond2_hfidDom}
	Let $\OM\subset\C$ be a hyperbolic domain that satisfies Condition~2 and let $S\subset\bdy\OM$ be the 
	associated totally disconnected subset. Let $\gamma$ be a non-degenerate
	component of $\bdy\OM$. Then $\OM_\gamma\cap\C$ satisfies Condition~2 with the associated 
	totally disconnected set $S\cap\gamma$.
\end{lemma}

\begin{proof}
	Note that $\bdy(\OM_\gamma\cap\C)=\bdy_\C\OM_\gamma=\gamma$.
	Choose $q\in\gamma\setminus S$ and let $(z_n)_{n\geq 1}\subset\OM_\gamma$ be 
	convergent to $q$. Using Lemma~\ref{lmm:cond-2-nbd-bhvr}, choose a neighbourhood
	$V$ of $q$ such that
	\[ V = (V\cap\OM) \cup (V\cap K_\gamma). \]
	Since $z_n\notin K_\gamma$, therefore, by the above, for all $n$ large enough, $z_n\in V\cap\OM\subset\OM$.
	Therefore $(z_n)_{n\geq 1}$ can be considered to be a sequence in $\OM$ converging to the point
	$q\in \gamma\setminus S \subset \bdy\OM\setminus S$. Since $\Omega$ satisfies Condition~2, 
	there exist an injective holomorphic map $\phi:\unitdisk\to\OM$ that extends continuously to $\clos{\unitdisk}$, points $\xi_1,
	\xi_2\in S^1$ close to $1$ and on opposite sides of $1$, and a subsequence $(z_{k_n})_{n\geq 1}$ such that
	$\phi(1)=q$, $\phi(Arc(\xi_1,\xi_2))\subset\bdy\OM$ and such that $z_{k_n}
	\in \phi(Reg(\xi_1,\xi_2))$. Clearly $\phi(Arc(\xi_1,\xi_2))\subset\gamma=\bdy(\OM_\gamma\cap\C)$, this 
	together with the fact that $\OM\subset\OM_\gamma$ implies that $\OM_\gamma$ satisfies Condition~2. 
\end{proof}

\begin{theorem} \label{thm:loc_conn_apt_tot_disconn_cond2}
	Let $\OM\subset\C$ be a hyperbolic domain. 
        Then $\OM$ satisfies Condition~2 if and only if 
	there exists a totally
	disconnected subset $S\subset\bdy\OM$ (not necessarily closed) such that, for every $p\in\bdy\OM\setminus S$, $\bdy\OM$ is 
	locally connected at $p$ and the connected component of $\bdy\OM$ containing $p$ is not a singleton (i.e., is non-degenerate). 	
\end{theorem}

\begin{proof}
        First suppose $\OM$ satisfies Condition~2 and let $S\subset\bdy\OM$ be the 
        associated totally disconnected subset. 
        Let $p\in\bdy\OM\setminus S$ and let $\gamma$ be the boundary component of $\OM$ such 
	that $p\in\gamma$. Clearly $\gamma$ is not singleton. 
	By Lemma~\ref{lmm:cond-2-nbd-bhvr}, there exists a neighbourhood $V$ of $p$ such that
	$V = (V\cap\OM) \cup (V\cap K_\gamma)$. From this it follows that $V\cap\bdy\OM=V\cap\gamma$. 
	From this latter fact
	it follows that, in order to prove the local connectedness of $\bdy\OM$ at $p$, it suffices to prove that of $\gamma$ at $p$.
	By Lemma~\ref{lmm:dom_cond2_hfidDom}, $\OM_\gamma\cap\C$ also satisfies Condition~2.
	Therefore, by Theorem~\ref{T:Visibility_Thm_in_plane}, it is a visibility domain. 
	Hence, by Lemma~\ref{lmm:pln_ext_Kob_isom_gen}, the Riemann map 
	$\phi$ extends to a continuous map from $\clos{\unitdisk}$ to 
	$\cl_{\C_\infty}(\OM_\gamma\cap\C)=\cl_{\C_\infty}(\OM_\gamma)$, 
	which we will continue to denote by $\phi$. Clearly $\phi$ maps $\bdy\unitdisk$ onto
	$\bdy_{\C_\infty}\OM_\gamma=\cl_{\C_\infty}(\gamma)$,
	which is therefore locally connected. Consequently, $\gamma$ itself is locally connected and in particular
	$\bdy\OM$ is locally connected at $p$ whence the conclusion. 
	\smallskip
	
	For the converse,
	given $p\in \bdy\OM\setminus S$, let $\gamma$ denote the connected component of $\bdy\OM$ containing $p$ and let
	$K_\gamma$ be the connected component of $\C\setminus\OM$ that contains $p$.
	\smallskip
	
	\noindent {\bf Claim~1.} For every $p\in\bdy\OM\setminus S$, there exists $r>0$ such that $D(p;r)\cap\bdy\OM=D(p;r)\cap\gamma$.
	\smallskip
	
	\noindent Assume, to get a contradiction, that for every $r>0$, $D(p;r)\cap\bdy\OM\supsetneq D(p;r)\cap
	\gamma$. Then, in particular, there exists a sequence $(z_n)_{n\geq 1}$ such that, for every $n$, $z_n\in D(p;1/n)\cap\bdy\OM$
	and $z_n\notin\gamma$. Since $z_n\notin\gamma$, $z_n$
	belongs to some connected component $C_n$ of $\bdy\OM$ other than $\gamma$.
	Now, considering the set $\{C_n\mid n\in\posint\}$ and arguing as in the proof of Lemma~\ref{lmm:loc_conn_outside_a_pt_loc_conn},
	Claim~1, it is easy to contradict the local connectedness of $\bdy\OM$ at $p$. Therefore, Claim~1 holds.\hfill$\blacktriangleleft$
	\smallskip
	
	\noindent {\bf Claim~2.} For every $p\in\bdy\OM\setminus S$, there exists a neighbourhood $V$ of $p$ such that $V=(V\cap\OM)
	\cup (V\cap K_\gamma)$.
	\smallskip
	
	\noindent By Claim~1, there exists $r_0>0$ such that $D(p;r_0)\cap\bdy\OM=D(p;r_0)\cap\gamma$. Put
	$r\defeq r_0/2$ and $V\defeq D(p;r)$; then we claim that $V$ satisfies the requirement of Claim~2. Assume, to get a contradiction,
	that it doesn't. Then there exists $q\in V\setminus (\OM\cup K_\gamma)$. Now, since $q\in V\subset D(p;r_0)$ and $q\notin\gamma$
	(because $q\notin K_\gamma$), $q\notin\bdy\OM$. So $q\notin\clos{\OM}$. Consider the straight-line segment joining $q$ to $p$;
	more precisely, consider the mapping
	\[ f\defeq t\mapsto q+t(p-q) : [0,1]\to V. \]
	Obviously it is continuous, and $f(0)=q\notin\clos{\OM}$. Let
	\[ t_0\defeq \sup\{t\in [0,1] \mid f(t)\notin\clos{\OM} \}. \]
	Then $f(t_0)\in\bdy\OM$ and $f([0,t_0])\subset\C\setminus\OM$. Now $f(t_0)\notin\gamma$ because, if it were, then $f([0,t_0])$
	would be a connected subset of $\C\setminus\OM$ that intersects $\gamma$, and hence one would have $f([0,t_0])\subset K_\gamma$,
	so in particular $f(0)=q\in K_\gamma$, which is not the case. Thus $f(t_0)\in\bdy\OM\setminus\gamma$. But $f(t_0)\in V\subset
	D(p;r_0)$ and so this is a contradiction. This proves Claim~2.\hfill$\blacktriangleleft$
	\smallskip
	
	To see that $\OM$ satisfies Condition~2, choose $p\in\bdy\OM\setminus S$ and let $(z_n)_{n\geq 1}$
	be a sequence in $\OM$ that converges to $p$. Let $\gamma$, $K_\gamma$ and $\OM_\gamma$ have the same meanings as above.
	By hypothesis $\gamma$ is not a singleton, so $\OM_\gamma$ is a hyperbolic, simply connected domain in $\C_\infty$.
	\smallskip
	
	Using the fact that $\bdy_{\C_\infty}\OM_\gamma\cap\C=\gamma$ and Claim~1 above, it follows that 
	$\bdy_{\C_\infty}\OM_\gamma$ is locally connected at every point of $(\bdy_{\C_\infty}\OM_\gamma\cap\C)\setminus S$. 
	Therefore we may invoke Lemma~\ref{lmm:simp_conn_loc_conn} to conclude 
	that $\bdy_{\C_\infty}\OM_\gamma$ is locally connected everywhere.
	Choose a biholomorphism $\phi:\unitdisk\to
	\OM_\gamma$. By Carth{\'e}odory's extension theorem \cite[Theorem~4.3.1]{BCM}, $\phi$ extends to a continuous map
	from $\clos{\unitdisk}$ onto 
	$\cl_{\C_\infty}(\OM_\gamma)$ (which we will continue to denote by $\phi$); note that $\phi$ maps $\bdy\unitdisk$ continuously to 
	$\cl_{\C_\infty}(\gamma)$.
	Using Claim~2, choose a neighbourhood $V$ of $p$ such that $V=(V\cap\OM)\cup(V\cap K_\gamma)$.
	Also, choose a neighbourhood $W$ of $p$ such that $W\Subset V$. (Recall that $p\in\C$ and that $V$
	and $W$ are neighbourhoods in $\C$.)
	We may suppose, without loss of generality, that $(z_n)_{n\geq 1}\subset W$.
	Consider $\phi^{-1}(W)$, which is an open subset of $\clos{\unitdisk}$ that intersects $\bdy\unitdisk$ (because it includes 
	$\phi^{-1}\{p\}$, which is non-empty). Because $(z_n)_{n\geq 1}$ converges to $p$, $(\phi^{-1}(z_n))_{n\geq 1}$, which is a sequence
	in $\unitdisk$, has a subsequence, say $(\phi^{-1}(z_{k_n}))_{n\geq 1}$, that converges to a point, say $x_0$, of $\phi^{-1}\{p\}$ 
	(indeed, the set of limit points of $(\phi^{-1}(z_n))_{n\geq 1}$ is included in $\phi^{-1}\{p\}$). Using a rotation, we may assume
	that $x_0=1$. Choose $r>0$ very small such that $D(1;r)\cap\clos{\unitdisk}\subset\phi^{-1}(W)$. 
	Then note that $U\defeq D(1;r)\cap\unitdisk$ is a
	convex domain that is mapped biholomorphically to $\phi(U)\subset W\cap\OM=W\cap\OM_\gamma$ by $\phi$.
	Further, $(z_{k_n})_{n\geq 1}$
	eventually lies in $\phi(U)$. Finally, writing $C\defeq D(1;r)\cap\bdy\unitdisk$, $C$ is a piece of $\bdy\unitdisk$ homeomorphic to an open
	interval; $C$ is mapped continuously into $\gamma\subset\bdy\OM$ by $\phi$, and $p\in\phi(C)$ (because $p=\phi(1)$). Now it is easy
	to see, using the Riemann mapping theorem, that one may choose an injective holomorphic map $\psi:\unitdisk\to\OM$ such that all
	the requirements of Condition~2 are satisfied (simply map $\unitdisk$ biholomorphically to $U$ by a Riemann map).
	Since $p\in\bdy\OM\setminus S$ and $(z_n)_{n\geq 1}$ were arbitrary, it follows that $\OM$ satisfies Condition~2.
	\end{proof}
	
\subsection{Proof of {\Cref{T:Visibility_Thm_in_plane}}} 
\label{ss:proof_of_Visibility_Thm_in_plane}
Our proof of \Cref{T:Visibility_Thm_in_plane} follows from 
\Cref{T:cond2ImplVisibility} and \Cref{thm:loc_conn_apt_tot_disconn_cond2}.

\subsection{Characterization of Condition 1}	
\begin{proposition}\label{prp:cond1_jordanbdy}
	Let $\OM\subset\C$ be a hyperbolic domain that satisfies Condition~1. Then
	(the closure in $\C_\infty$ of) every boundary component of $\OM$ is
	a Jordan curve. Conversely, if (the closure in $\C_\infty$ of) every boundary
	component of $\OM$ is a Jordan curve and $\bdy\OM$ is locally connected,
	then $\OM$ satisfies Condition~1.
\end{proposition}

\begin{proof}
	First suppose that $\OM$ satisfies Condition~1. 
	Choose $p\in\bdy\OM$ and let $\gamma$, $\OM_\gamma$ and $K_\gamma$ have the same meanings as above.
	Note that, by Condition~1,
	$\gamma$ is not a singleton (indeed, it includes a non-degenerate arc). 
	Consider the hyperbolic simply connected domain $\OM_\gamma\subset\C_\infty$. Since $\OM$ satisfies Condition~1,
	using Lemma~~\ref{lmm:cond-2-nbd-bhvr} and arguing as in Lemma~\ref{lmm:dom_cond2_hfidDom},
	 it is easy to see that $\OM_\gamma$ satisfies Condition~1.
	 \smallskip
	
	Choose a biholomorphism $\phi:\unitdisk\to\OM_\gamma$.
	We now consider two cases: $\gamma$ bounded and $\gamma$ unbounded. In the former
	case, $\OM_\gamma$ contains $\infty$ and so, strictly speaking, is not a planar domain; nevertheless,
	it is easy to see
	that Theorem~\ref{thm:ext_biholo} implies that $\phi$ extends to a homeomorphism from $\clos{\unitdisk}$
	to $\cl_{\C_\infty}(\OM_\gamma)$, which we will continue to denote by $\phi$, and which maps $\bdy\unitdisk$ homeomorphically to
	$\bdy_{\C_\infty}\OM_\gamma=\gamma$. Hence, in this case, $\gamma$
	(which coincides with $\cl_{\C_\infty}(\gamma)$) is a homeomorphic
	image of $S^1$, and is therefore a Jordan curve. In the latter case, $\OM_\gamma$ is a planar domain that satisfies Condition~1. By a
	direct application of Theorem~\ref{thm:ext_biholo}, $\phi$ extends to a homeomorphism from
	$\clos{\unitdisk}$ to $\clos{\OM}^{End}_\gamma$,
	which we will continue to denote by $\phi$, and which maps
	$\bdy\unitdisk$ homeomorphically to $\bdy\clos{\OM}^{End}_\gamma$. From this
	fact it follows easily that $\clos{\OM}_\gamma$ has only one end,
	whence $\clos{\OM}^{End}_\gamma \simeq \cl_{\C_\infty}(\OM_\gamma)$
	and $\bdy\clos{\OM}^{End}_\gamma \simeq \bdy_{\C_\infty}\OM_\gamma$.
	But $\bdy_{\C_\infty}\OM_\gamma=\cl_{\C_\infty}(\gamma)$, whence
	$\cl_{\C_\infty}(\gamma)$ is a homeomorphic image of $S^1$ and hence a Jordan curve. 
	Since $p\in\bdy\OM$ was arbitrary, it follows that (the closure in $\C_\infty$ of) 
	every component of $\bdy\OM$ is a Jordan curve.
	\smallskip
	
         To see the converse, choose $p\in\bdy\OM$ 
	and let $\gamma$, $\OM_\gamma$ and $K_\gamma$ have the same meanings as above.
	Since $\cl_{\C_\infty}(\gamma)$ is, by assumption, a Jordan curve, $\gamma$ is not a singleton.
	Therefore $\OM_\gamma\subset\C_\infty$ is a hyperbolic simply connected domain.
	Choose a biholomorphism $\phi:\unitdisk\to\OM_\gamma$.
	Since $\bdy_{\C_\infty}\OM_\gamma=\cl_{\C_\infty}(\gamma)$, which is a Jordan curve, it follows from Carath{\'e}odory's extension
	theorem for biholomorphisms that $\phi$ extends to a homeomorphism from $\clos{\unitdisk}$ to $\cl_{\C_\infty}(\OM_\gamma)$, which
	we will continue to denote by $\phi$. We may suppose, without loss of generality, that $\phi(1)=p$.
	Since $\bdy\OM$ is locally connected by assumption,
	the same proof as that of Claim~2 in Theorem~\ref{thm:loc_conn_apt_tot_disconn_cond2}
	works to show that there exists a neighbourhood $V$ of $p$ such that $V=(V\cap\OM)\cup(V\cap K_\gamma)$. This implies in particular
	that $V\cap\OM=V\cap\OM_\gamma$. Now choose $r>0$ such that $\clos{D}(p;r)\subset V$ and write $U\defeq 
	\phi^{-1}\big(D(p;r)\cap\clos{\OM}_\gamma\big)$; then $U$ is a neighbourhood of $1$ in $\clos{\unitdisk}$. Choose $s>0$ such that
	$\clos{D}(1;s)\cap\clos{\unitdisk}\subset U$. Once again using a Riemann map to map 
	$\unitdisk$ biholomorphically to $D(1;s)\cap\unitdisk$
	(and $\clos{\unitdisk}$ homeomorphically to $\clos{D(1;s)\cap\unitdisk}$), it is now easy to see that the requirement of Condition~1 
	is satisfied at $p$ (note that we need to restrict to $D(1;s)\cap\unitdisk$; otherwise $\phi$ might not map into $\OM$). Since $p$ was
	arbitrary, this shows that $\OM$ satisfies Condition~1.
\end{proof}

Next, we characterize visibility on a large class of planar domains.

\begin{proposition}\label{prp:visib_loc_conn}
	Let $\OM\subset\C$ be a hyperbolic domain and suppose that there exists a (closed) totally disconnected subset $S$ of $\bdy\OM$ such that,
	for every $p\in\bdy\OM\setminus S$, the connected component $\gamma$ of $\bdy\OM$ containing $p$ is not a singleton and there exists a
	neighbourhood $V$ of $p$ such that $V\cap\bdy\OM=V\cap\gamma$. Then $\OM$ is a visibility domain if and only if $\bdy\OM$ is locally
	connected at every point of $\bdy\OM\setminus S$.
\end{proposition}

\begin{proof}
	First suppose $\bdy\OM$ is locally connected at every point of $\bdy\OM\setminus S$. Then
	Theorem~\ref{thm:loc_conn_apt_tot_disconn_cond2} implies that 
	$\OM$ satisfies Condition~2 whence by Theorem~\ref{T:Visibility_Thm_in_plane} is
	a visibility domain.
	\smallskip
	
	Conversely, suppose that $\OM$ is a visibility domain. Let $p\in\bdy\OM\setminus S$ and suppose
	$\gamma$, $\OM_\gamma$ and $K_\gamma$ have their usual meanings. 
	\smallskip
	
	\noindent {\bf Claim~1.} $\OM_\gamma$ is a visibility domain.
	\smallskip
	
	\noindent We shall use Theorem~\ref{thm:loc_vis_imp_glob_vis_ntrnlzd} to prove the claim. It is clear that, for the
	purpose of this claim, it is enough to deal with $\OM_\gamma\cap\C$, i.e, to regard $\OM_\gamma$ as being contained in $\C$, which we shall do.
	First note that $\gamma\cap S$ is a totally disconnected subset of $\gamma=\bdy_\C\OM_\gamma$.
	Suppose that $q\in\gamma\setminus S$ and that 
	$(z_n)_{n\geq 1}$ is a sequence in $\OM_\gamma$ converging to $q$. By hypothesis, there exists a neighbourhood $V$ of $q$ such that
	$V\cap\bdy\OM=V\cap\gamma$. From this it follows, as in the proof of Claim~2 in 
	Theorem~\ref{thm:loc_conn_apt_tot_disconn_cond2}, that
	there also exists $r>0$ such that $D(q;r)=(D(q;r)\cap\OM)\cup(D(q;r)\cap K_\gamma)$. We may assume, without loss of generality, that
	$(z_n)_{n\geq 1}\subset D(q;r)$. Since $\OM$ is a visibility domain, we may invoke Lemma~\ref{prp:vsb_nbd_ntrsc_fin} to conclude that
	there are only finitely many connected components of $D(q;r)\cap\OM=D(q;r)\cap\OM_\gamma$ that contain $z_n$ for some $n$. From this it follows
	that there is some connected component of $D(q;r)\cap\OM_\gamma$, say $W$, that contains $z_n$ for infinitely many $n$. So there is a subsequence
	$(z_{k_n})_{n\geq 1}$ of $(z_n)_{n\geq 1}$ that lies in $W$. Note that, since $\OM$ is a visibility domain, it follows by 
	Theorem~\ref{thm:glob_vis_loc_vis} that every two points of $(\bdy W\cap\bdy\OM)\setminus\clos{\bdy W\cap\OM}$ possess the visibility property
	with respect to $\koba_W$. It is clear that $q\in(\bdy W\cap\bdy\OM)\setminus\clos{\bdy W\cap\OM}$. Furthermore, since $\OM_\gamma$ is a planar domain,
	it follows that
	\[ \koba_{\OM_\gamma}\big( D(q;r/2)\cap W,\OM_\gamma\setminus W \big)>0. \]
	Now recall that $q\in\gamma\setminus S$ and the sequence $(z_n)_{n\geq 1}$ converging to $q$ were arbitrary. Therefore, it follows by
	Theorem~\ref{thm:loc_vis_imp_glob_vis_ntrnlzd} that $\OM_\gamma$ is a visibility domain. \hfill $\blacktriangleleft$
	\smallskip
	
	Since $\OM_\gamma$ is a simply connected planar visibility domain and hence, by Theorem~\ref{thm:vis_simp_conn_loc_conn}, 
	$\bdy\OM_\gamma=\gamma$ is locally connected. In particular, it is locally connected at $p$. Once again recall, by hypothesis, that there exists a 
	neighbourhood $U$ of $p$ such that $U\cap\bdy\OM=U\cap\gamma$. From this and from the local connectedness of $\gamma$ at $p$ the local connectedness 
	of $\bdy\OM$ at $p$ follows. This completes the proof.
\end{proof}

Finally, we state a corollary whose proof we omit, because it follows immediately from the foregoing.

\begin{corollary}
	Let $\OM\subset\C$ be a hyperbolic planar domain such that there exists $\delta>0$ such that for every two distinct connected components
	$\gamma_1$ and $\gamma_2$ of $\bdy\OM$, $\distance_{Euc}(\gamma_1,\gamma_2)\geq\delta$. Then the following are equivalent:
	\begin{enumerate}
		\item $\OM$ is a visibility domain;
		\item $\bdy\OM$ is locally connected;
		\item Every connected component $\gamma\subset\bdy\OM$ is locally connected;
		\item For every connected component $\gamma\subset\bdy\OM$, $\cl_{\C_\infty}(\gamma)$ is a continuous surjective image of $S^1$.
	\end{enumerate}
\end{corollary}

\begin{proposition} \label{prp:cond1-sat-xtrnl}
	Suppose that $\OM\subset\C$ is a hyperbolic domain such that, for every $p\in\bdy\OM$ there exists a neighbourhood $U$ of $p$ such that
	$U\cap\OM$ is connected, simply connected, and has a boundary that is a Jordan curve. Then $\OM$ satisfies Condition~1.
\end{proposition}

We omit the proof of this because it follows the same pattern as, and is much simpler than, the proof of the following Proposition.

\begin{proposition} \label{prp:cond2-sat-xtrnl}
	Suppose that $\OM\subset\C$ is a hyperbolic domain such that there exists a (closed) totally disconnected subset $S$ of $\bdy\OM$ such
	that, for every $p\in\bdy\OM\setminus S$ there exists a bounded neighbourhood $U$ of $p$ such that $U\cap\OM$ has finitely many components,
	each of which is simply connected and has a boundary that is locally connected. Then $\OM$ satisfies Condition~2.
\end{proposition}

\begin{proof}
	Suppose that $p\in\bdy\OM\setminus S$ and that $(x_n)_{n\geq 1}$ is a sequence in $\OM$ converging to $p$. Using the hypothesis,
	choose a bounded neighbourhood $U$ of $p$ such that $U\cap\OM$ has finitely many components, say $V_1,\dots, V_m$, each of which is simply
	connected and has a locally connected boundary. Since $(x_n)_{n\geq 1}$ is a sequence in $\OM$ converging to $p$, since $U$ is a 
	neighbourhood of $p$, and since $U\cap\OM=V_1\cup\dots\cup V_m$, we may suppose that $V_1$ contains $x_n$ for infinitely many $n$.
	This implies that there is a subsequence $(x_{k_n})_{n\geq 1}$ of $(x_n)_{n\geq 1}$ that belongs to $V_1$. Since $V_1$ is, by assumption,
	simply connected, choose a biholomorphism $\phi:\unitdisk\to V_1$. Since $V_1$ is by assumption locally connected, by \cite[Theorem~4.3.1]{BCM}, 
	$\phi$ extends to a continuous surjective map from $\clos{\unitdisk}$ to $\clos{V}_1$. It is easy to see that 
	$p\in (\bdy V_1 \cap \bdy\OM)\setminus\clos{\bdy V_1\cap\OM}$. Choose a neighbourhood $W$ of $p$ such that
	$\clos{W}\cap\clos{\bdy V_1\cap\OM}=\emptyset$. 
	Consider the sequence
	$(\phi^{-1}(x_{k_n}))_{n\geq 1}$. It is clear that all its limit points must belong to $\bdy\unitdisk$ and must map to $p$ under $\phi$; furthermore,
	there is at least one such limit point. Thus, without loss of generality, we may suppose that $1$ is a limit point. This means that, passing to a
	subsequence, we may suppose that $(\phi^{-1}(x_{k_n}))_{n\geq 1}$ converges to $1$. Now note that $\phi^{-1}(W)$ is a neighbourhood of 
	$\phi^{-1}\{p\}$ in $\clos{\unitdisk}$. 
	In particular, $\phi^{-1}(W)$ is a neighbourhood of $1$ in $\clos{\unitdisk}$. 
	We may choose $x_1,x_2\in\bdy\unitdisk$ close to $1$ and on opposite sides of $1$ such that 
	$\clos{Reg}(x_1,x_2)\subset \phi^{-1}(W)$. We may suppose, without loss of generality, that $(\phi^{-1}(x_{k_n}))_{n\geq 1}$ lies in $Reg(x_1,x_2)$.
	We also claim that $\phi(Arc[x_1,x_2])\subset\bdy\OM$. To see why this is true, note that $\phi(Arc[x_1,x_2])\subset\phi(\bdy\unitdisk)=\bdy V_1$.
	Further, $\bdy V_1 = (\bdy V_1 \cap \bdy\OM) \cup (\bdy V_1\cap\OM)$. Now, $\phi(Arc[x_1,x_2])\subset W$,
	and so it is clear that $\phi(Arc[x_1,x_2])\cap\clos{\bdy V_1\cap\OM}=\emptyset$. From this it follows that $\phi(Arc[x_1,x_2])\subset 
	\bdy V_1\cap\bdy\OM$, as required. This shows, since $p\in\bdy\OM\setminus S$ and $(x_n)_{n\geq 1}$ were arbitrary, that $\OM$ satisfies Condition~2,
	as claimed.
\end{proof}

\section{Extension of biholomorphisms between planar domains}\label{sec:ext_biholo}
In this section, we shall present the proof of Theorem~\ref{thm:ext_biholo}.

\begin{proof}[Proof of \Cref{thm:ext_biholo}]
Since Condition~1 implies Condition~2, by 
Theorem~\ref{T:Visibility_Thm_in_plane} both  $\Omega_1$ and $\Omega_2$ 
are visibility domains. 		
	
First, we shall show that for all $\xi \in \partial \Omega_1$, $\lim_{z 
\to \xi} f(z)$ exists and lies in $\partial \overline \Omega_2^{End}$. 
Assume, to get a contradiction, that this does not hold. By the 
compactness of $\clos{\OM}_2^{End}$ there exist $\xi\in\bdy\OM_1$,
$p, q \in \partial \overline \Omega_2^{End}$, $p \neq q$, and sequences 
$(z_n)_{n \geq 1}$, $(w_n)_{n \geq 1}$ in $\Omega_1$ such that 
$z_n, w_n \to \xi$ as $n \to \infty$ and such that 
$\lim_{n \to \infty}f(z_n) = p$, $\lim_{n \to \infty}f(w_n) = q$. Since
 $\Omega_2$ is a visibility domain, Lemma~\ref{lmm:weak_visib_Grom_lim_fin} implies that
\[
\limsup_{n \to \infty}(f(z_n)| f(w_n))_{f(o)} < +\infty,
\]
where $o\in\OM_1$ is an arbitrary but fixed point.
Since $f$ is a Kobayashi isometry,
\[
\limsup_{n\to\infty} (z_n | w_n)_o = \limsup_{n \to \infty}(f(z_n)| f(w_n))_{f(o)} 
< +\infty.
\]
This implies that $\liminf_{z,w\to\xi}(z|w)_o<+\infty$, which is a contradiction to 
Lemma~\ref{L:Unbounded_Gromov_prod}.
Hence for every $\xi \in \partial \Omega_1$ there exists $p \in \partial 
\overline \Omega_2^{End}$ such that $\lim_{z \to \xi,\,z\in\OM_1}f(z) = p$.

\smallskip

Next, we shall show that for every end $\xi \in \overline{\Omega}_1^{End} 
\setminus \overline{\Omega}_1$, there exists $p \in \partial 
\overline{\Omega}_2^{End}$  such that 
$\lim_{z \to \xi,\,z\in\OM_1} f(z) = p$. 
Assume, to get a contradiction, that this does not hold. Once again by the
compactness of $\clos{\OM}_2^{End}$, there exist a point $\xi \in 
\clos{\OM}_1^{End}\setminus\clos{\OM}_1$, points $p\neq q \in \partial 
\overline{\Omega}_2^{End}$ and sequences $(z_n)_{n\geq 1}$ and 
$(w_n)_{n\geq 1}$ in $\Omega_1$ converging to $\xi$ such that 
$\lim_{n\to\infty} f(z_n) = p$ and $\lim_{n\to\infty} f(w_n) = q$. Since
$(z_n)_{n\geq 1}$ and $(w_n)_{n\geq 1}$ are sequences in $\OM_1$ converging
to the same end $\xi$, we may choose, for each $n$, a path 
$\sigma_n : [0, 1] \longrightarrow \Omega_1$ joining $z_n$ and 
$w_n$ such that $(\sigma_n)_{n\geq 1}$  eventually avoids every compact set.
Define $\alpha_n(t) \defeq 
f(\sigma_n(t))$ for all $n\in\posint$ and all $t \in [0,1]$.
Since
$(\sigma_n)_{n\geq 1}$ eventually avoids every compact subset of $\OM_1$
and $f$ is a biholomorphism, $(\alpha_n)_{n\geq 1}$ eventually 
avoids every compact subset of $\OM_2$.		
By \Cref{rmk:cnvg_outside_tot_disc}, there exist a sequence of points $z'_n \in 
\mathsf{ran}(\alpha_n)$ and $p_0 \in \partial \Omega_2 \setminus S$ such that $z'_n \to p_0$ as 
$n \to \infty$ and $p_0\neq p, q$. Say $z'_n = \alpha_n(t_n)$, where $(t_n)_{n\geq 1}$ is a sequence 
in $(0,1)$.
Now, take a neighbourhood $U$ of 
$p_0$ such that $p,q \in \C \setminus U$. By 
Condition~2 there exist a subsequence $(z'_{k_n})_{n\geq 1}$ of $(z'_n)_{n\geq 1}$,
a one-one holomorphic map $\phi: \mathbb{D} \longrightarrow \OM$ that extends continuously up to 
$\clos{\unitdisk}$, and arcs $Arc(x_1,\,x_2)$, $Arc(y_1,\,y_2)$ containing $1$, such that
$Arc(x_1,\,x_2)\Subset Arc(y_1,\,y_2)$, and satisfying the following properties:
\begin{enumerate}
\item $\phi(1)=p_0$, $\phi(Arc(y_1, y_2)) \subset \partial \Omega$, $\phi(Reg(y_1, y_2)) \subset U$,
\smallskip

\item for all $n \in \posint$, $z'_{k_n} \in \phi(Reg(x_1, x_2))$ and 
$\phi^{-1}(z'_{k_n}) \to 1$ as $n \to \infty$.
\end{enumerate}
As remarked just after Condition~3 in Section~\ref{S:Visibility_Thm}, we can choose $x_j, y_j$, $j=1,2$, 
such that 
the points $\phi(x_1), \phi(x_2), \phi(y_1), \phi(y_2), \phi(1)$ are 
distinct. In what follows, we suppress the subsequential notation and write $z'_{k_n}$ as $z'_n$. 
\smallskip

For every $n$, let
\begin{align*}
	s'_n &\defeq \inf\{ t \in [0,t_n] \mid \alpha_{n}((t,t_n])\subset
	\phi(\unitdisk) \}, \text{ and}\\
	s''_n &\defeq \sup\{ t \in [t_n,1] \mid \alpha_{n}([t_n,t))\subset
	\phi(\unitdisk) \}.
\end{align*}
It is obvious that the infimum above is actually a minimum and the supremum above
is actually a maximum. Consequently, one has $\alpha_{n}((s'_n,t_n])\subset
\phi(\unitdisk)$ and $\alpha_{n}([t_n,s''_n))\subset \phi(\unitdisk)$. It is also
clear that $\alpha_{n}(s'_n) \in \phi(\bdy\unitdisk) \cap \OM_2$ and
$\alpha_{n}(s''_n) \in \phi(\bdy\unitdisk) \cap \OM_2$. Now put
\begin{align*}
	u'_n &\defeq \inf\{ t\in [s'_n,t_n] \mid \alpha_{n}((t,t_n]) \subset
	\phi(Reg(y_1,y_2)) \}, \\
	v'_n &\defeq \inf\{ t\in [u'_n,t_n] \mid \alpha_{n}((t,t_n]) \subset
	\phi(Reg(x_1,x_2))\}.
\end{align*}
Then it follows easily that $s'_n<u'_n<v'_n<t_n<s_n''$.
It also follows
that $\alpha_{n}(u'_n)\in \phi((y_1,y_2)_p)$, that
$\alpha_{n}((u'_n,t_n])\subset\phi(Reg(y_1,y_2))$, that $\alpha_{n}(v'_n)\in 
\phi((x_1,x_2)_p)$, and that \linebreak
$\alpha_{n}((v'_n,t_n])\subset
\phi(Reg(x_1,x_2))$.
\smallskip

Observe that $\alpha_{n}([s'_n,s''_n])\subset\phi(\clos{\unitdisk})$ and
$\alpha_{n}((s'_n,s''_n))\subset\phi(\unitdisk)$. Put $X_n\defeq 
\overline{\phi^{-1}\big(\alpha_{n}((s'_n,s''_n))\big)}$. Note that this is a
closed, connected subset of $\clos{\unitdisk}$. Therefore we may suppose, without
loss of generality, that $(X_n)_{n\geq 1}$ converges, in the Hausdorff distance
associated to the Euclidean distance in $\overline{\unitdisk}$, to
some closed subset $X$ of $\clos{\unitdisk}$. Since each $X_n$ is connected,
it follows easily that $X$ is also connected. Since $(\alpha_n)_{n\geq 1}$
eventually avoids every compact subset of $\OM_2$, it follows immediately that
$(X_n)_{n\geq 1}$ eventually avoids every compact subset of $\unitdisk$. 
Consequently, $X \subset \bdy\unitdisk$. Note that, for every $n$,
$\phi^{-1}\big(\alpha_{n}(u'_n)\big)\in [y_1,y_2]_p$ and
$\phi^{-1}\big(\alpha_{n}(v'_n)\big)\in [x_1,x_2]_p$. Therefore, we may suppose,
without loss of generality, that
$\big(\phi^{-1}\big(\alpha_{n}(v'_n)\big)\big)_{n\geq 1}$ converges to one of
$x_1$ or $x_2$, say to $x_1$. Further, for every $n$, $\phi^{-1}(z'_{n})\in X_n$ 
and $\phi^{-1}(z'_{n})\to 1$. Thus $X$ is a closed, connected subset of 
$\bdy\unitdisk$ that contains both $1$ and $x_1$. It is easy to see that, in this
case, $Arc[x_1,1]\subset X$.
Using the continuity of $\phi$ on $\overline{\unitdisk}$, we see that there 
exists $x_1'\in Arc(x_1,1)$ such that 
\begin{equation} \label{eqn:small_arc_dstnct}
\forall\, x \in Arc(x_1',1), \; \phi(x) \notin \{\phi(x_1), \phi(x_2)\}	
\end{equation}
(it suffices to ensure that $\phi(x)$ is very close to $\phi(1)$, which, by
assumption, is distinct from $\phi(x_1)$ and $\phi(x_2)$). 
Now we make the following
\smallskip

\noindent {\bf Claim 0.}
For every $x \in Arc(x_1', 1)$,
\[
\lim_{r \to 1-}|f^{-1}(\phi(rx ) )| = \infty.
\]
Let $x$ be an arbitrary but fixed point of $Arc(x_1', 1)$.
Suppose that $\gamma_x : [0, +\infty) \longrightarrow \mathbb{D}$ is the 
parametrization of the radial geodesic ray of $(\mathbb{D}, \koba_{\mathbb{D}})$ 
joining $0$ and $x \in \partial \mathbb{D}$. Then $\phi \circ \gamma_x$ is a 
geodesic ray of $(\phi(\mathbb{D}), \koba_{\phi(\mathbb{D})} )$ emanating from the 
point $\phi(0)$ and landing at the boundary point $\phi(x) \in \bdy\phi(\unitdisk)\cap \bdy 
\OM_2$. Now we make the following subsidiary
\smallskip
			
\noindent
{\bf Claim 1.} There 
exists a point $\xi_0 \in \partial \overline \Omega_1^{End}$ such that   
\[
\lim_{t \to \infty} f^{-1}(\phi \circ \gamma_x (t) ) = \xi_0,\, \, \,\text{which is equivalent to} \, \, \,
\lim_{r \to 1^-} f^{-1}(\phi(rx) ) = \xi_0.
\]
The proof of Claim~1 crucially depends on the Claim~2 stated below. 
However, 	before we state and prove Claim~2, we need a few notations and a certain preparation
that will help in proving Claim~2.		
Given $A_j, B_j \subset \Omega_j$, $j=1,2$, the Hausdorff distance with respect to 
$\koba_{\OM_j}$ between $A_j, B_j$ will be denoted by $\mathcal{H}_{\koba}^j(A_j, B_j)$.
Recall that $x$ denotes an arbitrary but fixed element of $Arc(x_1', 1)$. Suppose 
that $(\beta_t)_{t\in [0,\infty)}$ is {\em any} family of curves in $\OM_2$ such
that, for every $t\in [0,\infty)$, $\beta_t : [0, T_t] \longrightarrow 
\Omega_2$ is a geodesic of $(\Omega_2, \koba_{\Omega_2})$ joining $\phi(0)$ and 
$\phi \circ \gamma_x(t)$, i.e., $\beta_t(0) = \phi(0)$ and $\beta_t(T_t) = \phi 
\circ \gamma_x(t)$. As $\OM_2$ is complete hyperbolic and as $\phi(x)\in 
\bdy\OM_2$, it follows that $T_t\to\infty$ as $t\to\infty$.
\smallskip

Choose $T_0<\infty$ 
sufficiently large so that, for all $t\geq T_0$, $\gamma_x(t)\in Reg(x_1',1)$.
Note that, for every $t\geq T_0$, $\beta_t(T_t) = \phi\circ\gamma_x(t) \in 
\phi(Reg(x'_1,1)) \subset \phi(Reg(x_1,x_2))$. So, for every $t\geq T_0$, 
$\beta_t(s) \in \phi(Reg(x_1,x_2))$ for every $s$ sufficiently large. Define,
for every $t\geq T_0$,
\begin{equation*}
S_t \defeq
\begin{cases}
	0, &\text{if } \mathsf{ran}\,\beta_t \subset \phi(Reg(x_1,x_2)),\\
	\sup\big\{ s\in [0,T_t] \mid \beta_t(s)\notin\phi(Reg(x_1,x_2)) \big\}
	&\text{otherwise.} 
\end{cases}
\end{equation*}
We claim that $\sup_{t\in [T_0,\infty)} S_t<\infty$. To get a contradiction, suppose there exists a 
sequence $(t_n)_{n\geq 1}$ such that $S_{t_n}\to\infty$ as $n\to\infty$. It is 
easy to see that if this is so, then one must also have $t_n\to\infty$. Note that, 
for every $n$, $\beta_{t_n}(S_{t_n})\in \phi((x_1,x_2)_p)$. From this, 
compactness, and the fact that $S_{t_n}\to\infty$ as $n\to\infty$, we may suppose, 
without loss of generality, that $\big(\beta_{t_n}(S_{t_n})\big)_{n\geq 1}$ 
converges to either $\phi(x_1)$ or $\phi(x_2)$.
Now consider the sequence of $\koba_{\OM_2}$-geodesics
$\big(\beta_{t_n}|_{[0,S_{t_n}]}\big)_{n\geq 1}$. Since 
$\beta_{t_n}(S_{t_n})\to \phi(x_1)$ or $\phi(x_2)$, therefore, by
Lemma~\ref{lmm:seq_geods_conv_geod_ray}, we may assume, without loss of 
generality, that $\big(\beta_{t_n}|_{[0,S_{t_n}]}\big)_{n\geq 1}$ converges 
locally uniformly on $[0,\infty)$ to a $\koba_{\OM_2}$-geodesic ray $\wt{\beta}:
[0,\infty)\to\OM_2$ that lands at $\phi(x_1)$ or $\phi(x_2)$, as the case may be.
On the other hand, the sequence of $\koba_{\OM_2}$-geodesics
$\big(\beta_{t_n}\big)_{n\geq 1}$ satisfy
$\beta_{t_n}(T_{t_n})=\phi(\gamma_x(t_n))$. As $\gamma_x(t_n)\to x$ as $n\to
\infty$ and $\phi(\gamma_x(t_n))\to\phi(x)$ as $n\to\infty$, we
can, once again by Lemma~\ref{lmm:seq_geods_conv_geod_ray}, suppose without loss
of generality that $\big(\beta_{t_n}\big)_{n\geq 1}$ converges, locally uniformly
on $[0,\infty)$, to a geodesic ray $\wh{\beta}:[0,\infty)\to\OM_2$ that lands at
$\phi(x)$.
But since $\beta_{t_n}|_{[0,S_{t_n}]}$ is merely a restriction of $\beta_{t_n}$,
$\wt{\beta}=\wh{\beta}$. But this implies that either $\phi(x)=\phi(x_1)$ or
$\phi(x)=\phi(x_2)$, which is an immediate contradiction to 
\eqref{eqn:small_arc_dstnct}. Therefore, it must be that $S^{*} \defeq 
\sup_{t\in [T_0,\infty)} S_t<\infty$. Thus, clearly, by the definition of $S_t$, 
for every $t\in [T_0,\infty)$ such that $T_t>S^{*}$, and for all $s>S^{*}$, 
$\beta_t(s) \in \phi(Reg(x_1,x_2))$. Now we make the following claim.
\medskip 
			
\noindent
{\bf Claim~2.} There exists a constant $M < \infty$ such that the Hausdorff 
distance with respect to $\koba_{\OM_2}$ between $\gamma_t \defeq \phi \circ \gamma_x|_{[0, t]}$
and  $\beta_t$\,---\,recall that $\beta_t : [0, T_t] \longrightarrow \Omega_2$ is a geodesic of 
$(\Omega_2, \koba_{\Omega_2})$ joining $\phi(0)$ and $\phi \circ \gamma_x(t)$\,---\,is 
bounded by $M$ for all $t \in [0, \infty)$, i.e.,  
\[
\forall\, t \in [0, \infty),\; \mathcal{H}_{\koba}^2(\beta_t, \gamma_t ) \leq M.
\]
Note that, to prove Claim~2, it suffices to 
show that there exists $M<\infty$ such that
\begin{equation} \label{eqn:Haus_dist_NTS}  
\forall\,t \text{ sufficiently large, } \mathcal{H}_{\koba}^2( \beta_t|_{[S^{*},T_t]},\gamma_t ) 
\leq M.
\end{equation} 
Write $\mathcal{H}_{\koba}^{\phi(\unitdisk)}(A,B)$ for the Hausdorff distance
with respect to the Kobayashi distance between subsets $A,B \subset
\phi(\unitdisk)$. It is obvious that, for all such $A$, $B$,
$\mathcal{H}_{\koba}^2(A,B) \leq \mathcal{H}_{\koba}^{\phi(\unitdisk)}(A,B)$. Since, for all $t$ 
sufficiently large, $\beta_t([S^*+1,T_t]) \subset \phi(Reg(x_1,x_2)) 
\subset \phi(\unitdisk)$, it follows that,
for all such $t$, $\beta_t([S^*+1,T_t])$ and
$\mathsf{ran}(\gamma_t)$ are both subsets of $\phi(\unitdisk)$. Therefore, by
our comments above, in order to establish \eqref{eqn:Haus_dist_NTS}, it suffices 
to establish that there exists $M<\infty$ such that
\begin{equation} \label{eqn:Haus_dist_final_suff_show}
\forall\,t \text{ sufficiently large, } 
\mathcal{H}_{\koba}^{\phi(\unitdisk)}(\beta_t|_{[S^{*}+1,T_t]},\gamma_t) \leq M.
\end{equation}
To show this we first note that there exists $\lambda<\infty$ such that
for sufficiently large $t$, $\beta_t|_{[S^*,T_t]}$ is a
$(\lambda,0)$-almost-geodesic for $(\phi(\unitdisk),\koba_{\phi(\unitdisk)})$.
To see this, we follow exactly the same arguments as given 
in the proof of Theorem~\ref{T:Visibility_Thm_in_plane}, more specifically, the arguments 
in Claim~2 and Claim~3 therein. 
In particular, each $\beta_t|_{[S^{*}+1,T_t]}$ is a (continuous) 
$(\lambda,0)$-quasi-geodesic for $(\phi(\unitdisk),\koba_{\phi(\unitdisk)})$ 
joining $\beta_t(S^{*}+1)$ and $\beta_t(T_t)=\phi\circ\gamma_x(t)$. Now, it is 
easy to see that 
\begin{equation*}
\{ \beta_t(S^{*}+1)\mid t\in [0,\infty) \text{ is large enough that } 
T_t>S^{*}+1 \}	
\end{equation*}
is a relatively compact subset of $\phi(\unitdisk)$. Therefore, it follows immediately
from the Geodesic Stability Theorem on the Gromov hyperbolic space 
$(\phi(\unitdisk), \koba_{\phi(\unitdisk)})$ (see \cite[1.7~Theorem]{Bridson_Haefliger})
that there exists $M<\infty$ such that
\[ \forall\,t \text{ sufficiently large,}\; 
\mathcal{H}_{\koba}^{\phi(\unitdisk)}(\beta_t|_{[S^{*}+1,T_t]},\gamma_t) \leq M. \]
This shows that \eqref{eqn:Haus_dist_final_suff_show} holds, and completes the 
proof of Claim~2.
\hfill$\blacktriangleleft$
\smallskip

We shall now prove Claim~1 but first we note two straightforward consequences that follow 
from Claim~2.
\begin{itemize}
\item[$(a)$] Consider the collection of paths $(\beta_t)$, $t\in[0, \infty)$. Then
\[
\mathcal{H}^2_{\koba}\left(\cup_{t \geq 0} \beta_t, \phi \circ \gamma_x \right) \leq M.
\]
\item[$(b)$] Suppose that $\beta, \wh{\beta}$ are geodesic rays which are obtained as  
subsequential limits (uniformly on compact subsets of $[0, +\infty)$) of the 
collection of geodesic segments $(\beta_t)_{t \geq 0}$. Then
\[
\mathcal{H}^2_\koba(\beta, \wh{\beta}) \leq 2M.
\]
\end{itemize}
			
Suppose $p_n, q_n \in \mathsf{ran}(\phi \circ\gamma_x)$ are such that $p_n, q_n \to 
\phi(x) \in \bdy \OM_2$. Then there exist $s_n, t_n \in [0, +\infty)$ such that 
$s_n, t_n \to +\infty$ as $n \to \infty$ and $\phi \circ \gamma_x(s_n) = p_n$ 
and $\phi \circ \gamma_x(t_n) = q_n$ for all $n \in \posint$.
Suppose $\beta$ and $\wh{\beta}$ are geodesic rays that are limits
of subsequences of the sequences of 
geodesics $(\beta_{s_n})_{n\geq 1}$ and $(\beta_{t_n})_{n\geq 1}$, respectively. 
Then, by part~$(b)$ above, $\mathcal{H}^2_\koba(\beta, \wh{\beta}) \leq 2M$. Consider
the sequences $(f^{-1}\circ\beta_{s_n})_{n\geq 1}$ and 
$(f^{-1}\circ\beta_{t_n})_{n\geq 1}$ of geodesics in $\OM_1$.
Since the subsequences of $(\beta_{s_n})_{n\geq 1}$ and $(\beta_{t_n})_{n\geq 1}$ in
question converge to $\beta$ and $\wh{\beta}$, respectively, it follows, since $f$ is a
biholomorphism, that the corresponding subsequences of 
$(f^{-1}\circ\beta_{s_n})_{n\geq 1}$ and $(f^{-1}\circ\beta_{t_n})_{n\geq 1}$ converge,
locally uniformly on $[0,\infty)$, to the geodesic rays $f^{-1}\circ\beta$ and
$f^{-1}\circ\wh{\beta}$ in $\OM_1$. By \Cref{lmm:seq_geods_conv_geod_ray}, $f^{-1}\circ
\beta$ and $f^{-1}\circ\wh{\beta}$ land at the corresponding subsequential limits of
$f^{-1}(\beta_{s_n}(T_{s_n}))=f^{-1}(p_n)$ and $f^{-1}(\beta_{t_n}(T_{t_n}))=
f^{-1}(q_n)$, say $\xi_1$ and $\xi_2$, respectively. 
Using the fact that $f$ preserves the Kobayashi distance,
we get $\mathcal{H}^1_\koba(f^{-1}\circ\beta,f^{-1}\circ\wh{\beta}) = 
\mathcal{H}^2_\koba(\beta, \wh{\beta} )\leq 2M 
< +\infty$. Lemma~\ref{lmm:diff_ends_geod_rays_Haus_infty} implies that $\xi_1 
= \xi_2$. Since $p_n$ and $q_n$ were completely arbitrary, Claim~1 follows; i.e., there 
exists a point $\xi_0 \in \bdy 
\clos{\OM}_1^{End}$ such that    
\[
\lim_{r \to 1^-} f^{-1}(\phi(rx) ) = \xi_0. 
\]
\hfill$\blacktriangleleft$

To establish Claim~0, we shall show that $\xi_0 = \xi$. Recall that $X_n$ is a closed, connected 
subset of $\unitdisk$ which contains a path joining two points, one of 
which, $\phi^{-1}(\alpha_{n}(v'_n))$, comes from $(x_1,x_2)_p$, and converges to 
$x_1$, and the other is $\phi^{-1}(z'_{n})$. Hence, for sufficiently 
large $n$, $X_n$ must intersect $\mathsf{ran}(\gamma_x)$. 
For every $n$ sufficiently large, choose $\zeta_n \in \mathsf{ran}(\gamma_x)\cap X_n$. 
Since $(X_n)_{n\geq 1}$ eventually avoids every compact subset of $\unitdisk$, it
follows that $(\zeta_n)_{n\geq 1}$ converges to $x$. Therefore $(\zeta_n)_{n\geq1}$
is a sequence of points of $\mathsf{ran}(\gamma_x)$ tending to $x$.
By Claim~1,
\[
\lim_{r \to 1-}f^{-1}(\phi(rx) ) = \xi_0 = \lim_{n \to \infty} f^{-1} 
(\phi(\zeta_n) ).
\]
But $f^{-1}(\phi(\zeta_n)) \in \mathsf{ran}(\sigma_n)$. Since, by definition, for every sequence 
$(z_n)_{n\geq 1}$ such that $z_n \in \sigma_n$, $\lim_{n \to \infty}z_n = \xi$, we 
have
\[
\lim_{r \to 1-}f^{-1}(\phi(rx) ) = \xi_0 = \lim_{n \to \infty} f^{-1} 
(\phi(\zeta_n) ) = \xi.
\]
Since $\xi$ is an end of $\clos{\OM}_1$, we have,
\[
\lim_{r\to 1-} |f^{-1}(\phi(rx))| = \infty.
\]
This proves Claim~0. 
\hfill$\blacktriangleleft$
\smallskip
			
We shall now complete the proof of the theorem by contradicting our assumption.		
Since $\Omega_1$ satisfies Condition~1, there exists $z_0 
\in \C \setminus \Omega_1$ and $r_0 > 0$ such that  
$\overline{D(z_0;r_0)} \subset \C \setminus \OM_1$. Hence the following 
function is well-defined, holomorphic, and bounded on the unit disk $\unitdisk$:				
\[
g(\zeta) \defeq \frac{r_0}{z_0 - f^{-1}\circ\phi(\zeta)} \quad \forall\, \zeta \in 
\unitdisk.
\]
It follows immediately from Claim~0 that, for every $x \in Arc(x_1',1)$,
\[
\lim_{r \to 1-}g(rx) = 0.
\]
By a standard result in the theory of univalent functions (see, for example,
\cite[Proposition~3.3.2]{BCM}),
this is an immediate contradiction. 

Thus our starting assumption must be wrong and so, for every $\xi\in 
\clos{\OM}^{End}_1 \setminus \clos{\OM}_1$, $\lim_{z \to \xi,\,z\in\OM_1}
f(z)$ must exist as an element of $\bdy\clos{\OM}^{End}_2$.

Combining this with what we obtained before, we can say: for every $\xi 
\in \clos{\OM}_1^{End} \setminus \OM_1$, there exists 
$p\in\bdy\clos{\OM}_2^{End}$ 
such that
\[
\lim_{z \to \xi,\,z\in\OM_1} f(z) = p.
\]
Now we define $\wt{f}:\clos{\OM}_1^{End} \lraw \clos{\OM}_2^{End}$ as
follows:
\begin{equation*}
\wt{f}(\xi) \defeq \begin{cases}
	f(\xi), &\text{if } \xi\in\OM_1,\\
	\lim_{z \to \xi,\,z\in\OM_1} f(z), &\text{if } \xi\in
	\clos{\OM}_1^{End} \setminus \OM_1.
\end{cases}
\end{equation*}
Then it is straightforward to check that $\wt{f}$ is continuous (we keep
in mind that $\clos{\OM}_1^{End}$ is first-countable). Thus we have our
required continuous extension. It is also straightforward to check that
$\wt{f}$ is surjective (simply use the fact that $\clos{\OM}_1^{End}$ 
and $\clos{\OM}_2^{End}$ are compact).
\smallskip

For the final assertion of the theorem, note that if $\OM_2$ also 
satisfies Condition~1, we can apply the first assertion of the theorem to
$f^{-1}$, which is a biholomorphism from $\OM_2$ to $\OM_1$. Consequently,
$f^{-1}$ extends to a continuous mapping $\wt{f^{-1}}$ from 
$\clos{\OM}_2^{End}$ to $\clos{\OM}_1^{End}$. It is then easy to see that 
$\wt{f}$ is a homeomorphism, with inverse $\wt{f^{-1}}$. 
\end{proof}

\subsection{Comparison of \Cref{thm:ext_biholo} with earlier extension results}\label{sss-compare} 

In this subsection we compare our extension result \Cref{thm:ext_biholo} with the following three recent results in this direction. The first one is due to Luo and Yao \cite[Theorem 3]{Luo-Yao2022}.

\begin{result}[Luo-Yao]\label{res-LuoYao2022}
    Let $\OM_1$ and $\OM_2$ be two generalized Jordan domains. Let $\{P_n\}$ and $\{Q_n\}$ be nondegenerate components of $\bdy\OM_1$ and $\bdy\OM_2$ respectively. Assume that: 
    \begin{itemize}
        \item [(i)] $\sum_n\diam(Q_n)<\infty$.
        \item[(ii)] The point components of $\bdy\OM_1$ or $\bdy\OM_2$ forms a set of $\sigma$-finite linear measure.
    \end{itemize}
    Then any biholomorphism $f:\OM_1 \to\OM_2$ extends to a continuous map from $\overline{\OM_1}$ to $\overline{\OM_2}$. If, in addition, $\sum\diam(P_n)<\infty$, then there is a homeomorphic extension of $f$ from $\overline{\OM_1}$ to $\overline{\OM_2}$.
\end{result}

The next one is due to Ntalampekos \cite[Theorem~1.4]{Ntalampekos2023}.
\begin{result}[Ntalampekos]\label{res-Ntalampekos2023}
    Let $\OM_1$ and $\OM_2$ be two generalized Jordan domains in $\C\cup
    \{\infty\}$ such that $\OM_1$ is cofat and every compact subset of the 
    point components of $\bdy\OM_1$ is CNED. Assume that $f:\OM_1\to\OM_2$ is
    a biholomorphism. 
    \begin{itemize}
        \item[(i)] If the diameters of $\bdy\OM_2$ lie in $l^2$ then $f$ has a continuous extension from $\overline{\OM_1}$ to $\overline{\OM_2}$.
        \item[(i)] If, in addition to $(i)$, $\OM_2$ is also cofat, then there is a homeomorphic extension of $f$ from $\overline{\OM_1}$ to $\overline{\OM_2}$.
    \end{itemize}
\end{result}
The next result is due to Bharali and Zimmer \cite[Theorem 1.10]{BZ2023}.
\begin{result}[Bharali-Zimmer]\label{res-BZ2023}
   Let $\OM_1$ and $\OM_2$ be two proper Lipschitz subdomains of $\C$. Then 
   any biholomorphism $f:\OM_1\to\OM_2$ extends to a homeomorphism from 
   $\overline{\OM_1}^{End}$ to $\overline{\OM_2}^{End}$. 
\end{result}

\noindent There are visibility domains which do not satisfy the conditions 
of \Cref{res-LuoYao2022} or \Cref{res-BZ2023} but satisfy the conditions of 
\Cref{thm:ext_biholo} (see \Cref{ex-extenion1} for details). 
\section{Examples}\label{Sec:Examples}

In this section, we will construct some planar domains $\OM$ such that $\bdy \OM$ is
{\em very} irregular (in particular, such that $\bdy \OM \setminus \bdy_{lg} \OM$ is 
not totally disconnected) but such that we can still conclude that $\OM$ is a 
visibility domain. We start with a certain continuous nowhere-differentiable 
function on $\R$, namely the {\em Takagi} or {\em blancmange} function. It is 
defined as follows.
\begin{equation} \label{eqn:Takagi_fn_def}
	\forall\,t\in\R,\; T(t) \defeq \sum_{j=0}^{\infty} 2^{-j} \distance(2^j t,\Z).
\end{equation}
In the above equation, $\distance(\bcdot\,,\Z)$ denotes the Euclidean distance 
between an arbitrary point of $\R$ and the set of integers $\Z$. 
Note that $T$ is clearly continuous and
$1$-periodic, hence uniformly continuous. It also follows immediately
from the definition that $0\leq T\leq 1$.

We shall only use two main facts about the Takagi function, which we now state.

\begin{result}[see, for example, \cite{ShiSab}]  \label{res:Taka_s-Holder}
	For every $s\in (0,1)$, $T$ is $s$-H{\"o}lder-continuous, i.e., for every $s\in 
	(0,1)$, there exists $M_s<\infty$ such that
	\begin{equation*}
		\forall\,x,y\in\R,\; |T(x)-T(y)|\leq M_s|x-y|^s.
	\end{equation*}
\end{result}

\begin{result}[see, for example {\cite[Theorem~9.3.1]{JarPflContNDiff}}]
	\label{res:Taka_derivs}  
	For every dyadic rational $x\in\R$, the right-hand derivative of $T$ at $x$ is
	$+\infty$ and the left-hand derivative of $T$ at $x$ is $-\infty$.
\end{result}

We first use the Takagi function to construct the following simple domain.
\begin{equation*}
	D_T \defeq \{x+iy\in\C\mid y>T(x)\}.
\end{equation*}

We now state and prove two lemmas.

\begin{lemma} \label{prp:cusps_fit_into_TFD}
	There exists $r_0>0$ such that, for every $x_0\in\R$,
	\begin{equation} \label{eq:incl_cusp_D_propStmt}
		V_{x_0,r_0} \defeq \{x+iy \in \R^2 \mid |x-x_0|<r_0, \, y>T(x_0)+|x-x_0|^{1/2}\} 
		\subset D_T.
	\end{equation} 
	In other words, quadratic cuspidal domains having the $y$-axis as their axis of 
	symmetry and having their vertex at an arbitrary boundary point can be fitted into 
	$D_T$.
\end{lemma}

\begin{proof}
	Using Result~\ref{res:Taka_s-Holder}, there exists $M<\infty$ such that
	\begin{equation*}
		\forall\,x,y\in\R,\; |T(x)-T(y)|\leq M|x-y|^{3/4}.
	\end{equation*}
	Choose $r_0>0$ such that $Mr^{1/4}_0<1$. Suppose that $x_0\in\R$ is arbitrary, 
	that $x\in\R$ is such that $|x-x_0|<r_0$, and that $y\in\R$ satisfies $y>T(x_0)+
	|x-x_0|^{1/2}$. What we have to do is show that $x+iy\in D_T$, or, equivalently,
	that $y>T(x)$. But we know that
	\begin{align*}
		|T(x)-T(x_0)|\leq M|x-x_0|^{3/4} = M |x-x_0|^{1/2}|x-x_0|^{1/4} < Mr^{1/4}_0
		|x-x_0|^{1/2} \leq |x-x_0|^{1/2},
	\end{align*}
	by the choice of $r_0$. From the above we obtain
	\[ T(x_0)+|x-x_0|^{1/2} \geq T(x). \]
	Since $y > T(x_0)+|x-x_0|^{1/2}$, it follows from the above that $y>T(x)$, as 
	required.
\end{proof}

\begin{remark} \label{rmk:fit_cusp_trunc_Taka_dom}
	It follows immediately from the above lemma that given $x_0\in\R$ and given a
	neighbourhood $V$ of $\wh{z}_0\defeq x_0+iT(x_0)$, there exists $s>0$ such that 
	$V_{x_0,r_0}\cap D(\wh{z}_0;s) \subset D_T\cap V$.
\end{remark}

The next lemma shows that, given any boundary point of $D_T$ that is indexed
(in the obvious sense) by a dyadic rational, there exist arbitrarily narrow angular
segments with vertex at that point that contain suitably small truncations of $D_T$
at that point.

\begin{lemma} \label{lmm:nrrw_cntng_trngls}
	Suppose that $x_0\in\R$ is a dyadic rational. Then for every $M$, $0<M<\infty$,
	there exists $t_0>0$ such that
	\begin{equation} \label{eqn:Taka_rstrc_included_in_trngl}
		D_T \cap \{x+iy\in\C\mid |x-x_0|<t_0\} \subset U_{x_0,M,t_0},  
	\end{equation}
	where
	\begin{equation} \label{eqn:def_big_trngl} 
		U_{x_0,M,t_0}\defeq \{x+iy\in\C\mid |x-x_0| < t_0,\, y> T(x_0)+M|x-x_0|\}.
	\end{equation}
\end{lemma}

\begin{proof}
	By Result~\ref{res:Taka_derivs}, there exists $t_0>0$ such that
	\[ \forall\,x\in (x_0-t_0,x_0+t_0)\setminus\{x_0\},\; 
	\frac{T(x)-T(x_0)}{|x-x_0|}>M. \]
	From the above it follows that
	\[ \forall\,x\in (x_0-t_0,x_0+t_0),\; T(x)>T(x_0)+M|x-x_0|. \]
	Now, if $x+iy\in D_T$ is such that $|x-x_0|<t_0$, then we know that $y>T(x)$ and 
	so it follows from the above that $y>T(x_0)+M|x-x_0|$, i.e., $x+iy\in 
	U_{x_0,M,t_0}$, as required. 
\end{proof}

Now we turn to the construction of the domain that we are interested in.

\begin{example}
	For every $n\in\Z$, let 
	\begin{align*}
		H_{n,T}\defeq &\{ z\in\C \mid 4n \leq \rprt(z) \leq 4n+1,\, \iprt(z)\geq 
		6-T(\rprt(z)) \} \\
		&\cup \{ z\in\C \mid \iprt(z)\geq 6,\, 4n-T(\iprt(z))\leq \rprt(z) 
		\leq 4n+1+T(\iprt(z)) \}.
	\end{align*}
	(The choices for the various parameters will become clear as we proceed.) 
	Note that each $H_{n,T}$ is 
	the closed, infinite rectangle 
	\[ \{z\in\C\mid \iprt(z)\geq 6,\, 4n\leq\rprt(z)\leq 4n+1\} \]
	with each of the three lines constituting its boundary replaced by (a 
	suitably oriented copy of) the graph of the Takagi function. Note also that 
	$\bdy H_{n,T}$ is homeomorphic to $\R$. 
	It is clear that, if we write $H_T\defeq \cup_{n\in\Z}H_{n,T}$ and 
	\begin{equation} \label{eqn:def_U_T} 
		U_T\defeq \C\setminus H_T, 
	\end{equation} 
	then $U_T$ is a simply connected (hyperbolic) domain in $\C$ with infinitely many 
	ends (it has an end corresponding to each $n\in\Z$ and one other ``at $\infty$'').
	
	Now we proceed to obtain an infinitely-connected domain from $U_T$. For 
	every $n\in \Z$, consider
	\begin{align*} 
		S_{n,T} \defeq\; &\{z\in\C \mid 4n\leq\rprt(z)\leq 4n+1, \, 2-T(\rprt(z)) 
		\leq\iprt(z) \leq 3+T(\rprt(z)) \} \\ 
		&\cup \{z\in\C \mid 2\leq\iprt(z)\leq 3, \, 4n-T(\iprt(z))
		\leq \rprt(z) \leq 4n+1+T(\iprt(z)) \}. 
	\end{align*}
	Note that $S_{n,T}$ is basically the closed rectangle
	\[ \{z\in\C\mid 4n\leq\rprt(z)\leq 4n+1,\,2\leq\iprt(z)\leq 3\} \]
	with each side replaced by (a suitably oriented copy of) the graph of 
	$T|_{[0,1]}$. Note also that $\bdy S_{n,T}$ is homeomorphic to $S^1$. 
	Thus, if we write $V_T\defeq U_T\setminus S_T$, then it is clear that 
	$V_T$ is an infinitely-connected domain in $\C$ with 
	infinitely many ends. Every boundary component of $V_T$ is either homeomorphic
	to $\R$ or is a Jordan curve (as can be seen, the components of $\bdy V_T$ 
	that are homeomorphic to $\R$ are $\bdy H_{n,T},\, n\in\Z$,
	whereas those homeomorphic to $S^1$ are $\bdy S_{n,T},\, n\in\Z)$.
	By construction, $\bdy V_T$ is locally connected. Therefore, by 
	\Cref{prp:cond1_jordanbdy}, $V_T$ satisfies 
	Condition~1 and, consequently, is a 
	visibility domain. 
	However, it will follow from the proposition below 
	that $V_T$ does {\em not} satisfy the hypotheses of
	\cite[Theorem~1.4]{BZ2023}. In fact, it will follow that {\em no point} 
	in $\bdy V_T$ is a local Goldilocks point, to use the terminology of 
	\cite[Definition~1.3]{BZ2023}.
\end{example}  

Before we begin the task of proving that no point in $\bdy V_T$ is a local
Goldilocks point, we need to make some preparations. We first take note of the 
following elementary estimate of the Kobayashi distance on an infinite angular 
sector.

\begin{proposition} 
	Let $M>0$ be arbitrary. Then for every $h_0>0$, $p_0>0$, there exists $C<\infty$ 
	such that, writing $U_M\defeq \{z\in\C\mid \rprt(z)>M|\iprt(z)|\}$, and regarding 
	$(0,\infty)\subset U_M$,
	\begin{equation*}
		\forall\,h\in (0,h_0),\; \koba_{U_M}(p_0,h)\geq \frac{\pi}{8\cot^{-1}(M)}
		\log(1/h)-C.
	\end{equation*}
\end{proposition}

Using the facts that $U_{x_0,M,t_0}$ is simply the rotation of $U_M$ by $\pi/2$ in
the anticlockwise sense followed by its translation by 
$\wh{z}_0\defeq x_0+iT(x_0)$ and that the Kobayashi distance of a subdomain 
dominates that of the original domain, we can also say: 

\begin{proposition} \label{prp:trngl_kob_dist_lb}
	Let $M>0$ and $r>0$ be arbitrary. Then for every $h_0>0$, $p_0>0$, there exists 
	$C<\infty$ such that, 
	\begin{equation*}
		\forall\,h\in (0,h_0),\; 
		\koba_{U_{x_0,M,t_0}\cap D(z_0;r)}(z_0+ip_0,z_0+ih)\geq 
		\frac{\pi}{8\cot^{-1}(M)} \log(1/h) - C.
	\end{equation*}
\end{proposition}

We shall also need another result, which allows us to compare the Kobayashi 
distance of a truncation of $V_T$ near a boundary point with that of the whole
domain. We use a very recent result providing a localization of the Kobayashi 
distance under an assumption of visibility, namely \cite[Theorem~1.3]{ADS2023}. We 
have already remarked that $V_T$ is a visibility domain. It is also clear from 
the construction of $V_T$ that for every $p\in \bdy V_T$, there exists $\wt{r}>0$
such that, for every $r$, $0<r\leq\wt{r}$, $D(p;r)\cap V_T$ is connected (indeed,
contractible). Since {\em every} two distinct points of $\bdy V_T$ satisfy the
visibility property with respect to $\koba_{V_T}$, so do every two distinct points
of $D(p;\wt{r})\cap\bdy V_T$. Therefore, by invoking \cite[Theorem~1.3]{ADS2023} 
directly, we can say that given $p\in\bdy V_T$ and $r>0$ small, there exists 
$C<\infty$ such that
\begin{equation} \label{eqn:loc_Kob_dist_V_T}
	\forall\,z,w\in D(p;r)\cap V_T,\; \koba_{D(p;2r)\cap V_T}(z,w)\leq 
	\koba_{V_T}(z,w)+C.
\end{equation}



Now we turn to the task of proving that no point in $\bdy V_T$ is a local 
Goldilocks point. Since the set of all local Goldilocks points is an open subset 
of the boundary, in order to show that no point in $\bdy V_T$ is a local
Goldilocks point, it suffices to show that there exists a {\em dense} subset 
$A$ of $\bdy V_T$ such that no point of $A$ is a local Goldilocks point of $V_T$.
Note that, by construction, there exists a closed, discrete subset $S$ of 
$\bdy V_T$ such that, near every point of $\bdy V_T \setminus S$, $\bdy V_T$ is 
the graph of $T|_I$, where $I$ is a suitably small interval of $\R$ depending on 
the boundary point (we can take $S$ to be the set of all corners of all the
``rectangles'' thrown out of $\C$ in the process of obtaining $V_T$).

It is clear from the construction of $V_T$ that for every 
$p\in \bdy V_T \setminus S$, there exist a neighbourhood $U$ of $p$, an $r_1>0$
and an $x_0\in\R$ such that after just a translation and rotation one can write
\begin{equation} \label{eqn:xpr_bdy_nbd}
	U = \{x+iy\in\C\mid |x-x_0|<r_1,\, -1/2<y<5/4\} 
\end{equation}
and 
\begin{equation} \label{eqn:xpr_bdy_nbd_intrsc_dom}
	U \cap V_T = \{x+iy\in U\mid y>T(x)\} = \{x+iy\in\C\mid |x-x_0|<r_1,\, T(x) <
	y < 5/4\}.
\end{equation}
In particular, since the dyadic rationals are dense in $\R$, it follows that there
exists a dense subset $R$ of $\bdy V_T \setminus S$ (hence also a dense subset of
$\bdy V_T$) such that for all $p\in R$ one can write \eqref{eqn:xpr_bdy_nbd} and
\eqref{eqn:xpr_bdy_nbd_intrsc_dom} with $x_0$ a dyadic rational.

With this is mind, we will first prove that if $x_0\in\R$ is any dyadic rational 
and if $r_1>0$ is arbitrary then, if we write
\begin{align} \label{eqn:def_trunc_Taka_dom}
	D_{T,x_0,r_1} &\defeq \{x+iy\in\C\mid |x-x_0|<r_1,\,T(x)<y<5/4\} \notag \\ 
	&= D_T \cap \{x+iy\in\C\mid |x-x_0|<r_1,\,-1/4<y<5/4\},
\end{align}
$D_{T,x_0,r_1}$ is not locally Goldilocks at the boundary point $\wh{z}_0 
\defeq x_0+iT(x_0)$.
(We point out that the upper bound $5/4$ above is entirely arbitrary; we could 
replace $5/4$ with any number $\rho>1$.) We shall see later why this is enough to 
conclude that $V_T$ is not locally Goldilocks at any boundary point.

\begin{proposition} \label{prp:trunc_Taka_dom_not_alph_grwth}
	$D_{T,x_0,r_1}$ (defined in {\eqref{eqn:def_trunc_Taka_dom}}) is not locally 
	Goldilocks at the boundary point $\wh{z}_0\defeq x_0+iT(x_0)$.
\end{proposition}

\begin{proof}
	We will show that the condition on the growth of the Kobayashi distance that is
	part of the definition of a local Goldilocks point is violated near the boundary
	point $\wh{z}_0$. We assume, to get a contradiction, that the condition holds
	near that point. This means that there exist a neighbourhood $V$ of $\wh{z}_0$,
	a point $z_0$ of $V\cap D_{T,x_0,r_1}$ and $C,\alpha<\infty$ such that
	\begin{equation} \label{eqn:alph_log_grwth_assmp_trunc_Taka}
		\forall\,z\in V\cap D_{T,x_0,r_1},\; \koba_{D_{T,x_0,r_1}}(z_0,z)\leq (\alpha/2)
		\log\big( 1/\dtb{D_{T,x_0,r_1}}(z) \big)+C.
	\end{equation}
	Using Remark~\ref{rmk:fit_cusp_trunc_Taka_dom}, we choose $s>0$ such that
	$V_{x_0,r_0}\cap D(\wh{z}_0;s) \subset D_T\cap V = D_{T,x_0,r_1}\cap V$ (we
	suppose, without loss of generality, that $V$ is so small that this happens). Then,
	for every $h\in (0,s)$, since 
	\[ u_h\defeq \wh{z}_0+ih\in V_{x_0,r_0}\cap D(\wh{z}_0;s), \]
	it follows that there exists a sufficiently small $h_0>0$ such that, for every 
	$h\in (0,h_0)$,
	\begin{equation} \label{eqn:dist_cmpr_trunc_Taka_cusp}
		\dtb{D_{T,x_0,r_1}}(u_h) \geq \dtb{V_{x_0,r_0}\cap D(\wh{z}_0;s)}(u_h) 
		= \dtb{V_{x_0,r_0}}(u_h).	
	\end{equation}
	It follows from plane analytic geometry and calculus that by possibly shrinking 
	$h_0$ we may assume that there exists $c>0$ such that
	\begin{equation*}
		\forall\, h\in (0,h_0),\; ch^2 \leq \dtb{V_{x_0,r_0}}(u_h)\leq (1/c)h^2.
	\end{equation*}
	By \eqref{eqn:alph_log_grwth_assmp_trunc_Taka}, 
	\eqref{eqn:dist_cmpr_trunc_Taka_cusp} and the above, we can say that
	\begin{equation} \label{eqn:alph_log_trunc_Taka_one_dir}
		\forall\,h\in (0,h_0),\; \koba_{D_{T,x_0,r_1}}(z_0,u_h)\leq \alpha\log(1/h)+C,
	\end{equation}
	with a possibly increased $C$.
	
	Now choose $M$, $0<M<\infty$, so large that
	\begin{equation} \label{eqn:choice_of_M} 
		\frac{\pi}{8\cot^{-1}(M)}-\alpha>0.
	\end{equation}
	Now, using Lemma~\ref{lmm:nrrw_cntng_trngls}, we choose $t_0>0$ such that
	\begin{equation*}
		D_T \cap \{x+iy\in\C\mid |x-x_0|<t_0\} \subset U_{x_0,M,t_0}.
	\end{equation*}
	We can also suppose, without loss of generality, that $t_0<s$. In particular, by
	the above inclusion, we also have
	\begin{equation} \label{eqn:inclsn_trunc_Taka_dom_trngl}
		D_T \cap D(\wh{z}_0;t_0) = D_{T,x_0,r_1} \cap D(\wh{z}_0;t_0) \subset 
		U_{x_0,M,t_0} \cap D(\wh{z}_0;t_0).
	\end{equation}
	Write
	\begin{equation*}
		\wt{D}_{x_0,r}\defeq D_{T,x_0,r_1} \cap D(\wh{z}_0;r)
	\end{equation*}
	for $r>0$ small; $\wt{D}_{x_0,r}$ is a truncation of $D_{T,x_0,r_1}$. Now note 
	that $D_{T,x_0,r_1}$ satisfies Condition~1 and is, therefore, itself a visibility 
	domain. Consequently, every pair of distinct boundary points of $D_{T,x_0,r_1}$ 
	satisfies the visibility property with respect to $\koba_{D_{T,x_0,r_1}}$. In 
	particular, every pair of distinct points of $\bdy D_{T,x_0,r_1}\cap 
	D(\wh{z}_0;t_0)$ satisfies the visibility property with respect to 
	$\koba_{D_{T,x_0,r_1}}$. Consequently, once again by \cite[Theorem~1.3]{ADS2023}, we 
	can say that there exists $C<\infty$ such that
	\begin{equation*}
		\forall\,z,w\in \wt{D}_{x_0,t_0/2},\; \koba_{\wt{D}_{x_0,t_0}}(z,w)\leq
		\koba_{D_{T,x_0,r_1}}(z,w)+C.
	\end{equation*}
	We may assume, without loss of generality, that $z_0\in \wt{D}_{x_0,t_0/2}$. We may
	also shrink $h_0$ so that $h_0<t_0/2$. Then, using the inequality above, we can
	write
	\begin{equation*}
		\forall\,h\in (0,h_0),\; \koba_{\wt{D}_{x_0,t_0}}(z_0,u_h)\leq
		\koba_{D_{T,x_0,r_1}}(z_0,u_h)+C.
	\end{equation*}
	Combining this inequality with \eqref{eqn:alph_log_trunc_Taka_one_dir} we can
	write
	\begin{equation} \label{eqn:alph_log_trunc_Taka_fin}
		\forall\,h\in (0,h_0),\; \koba_{\wt{D}_{x_0,t_0}}(z_0,u_h)\leq \alpha\log(1/h)+C,
	\end{equation}
	with a possibly increased $C$. But we also know from 
	\eqref{eqn:inclsn_trunc_Taka_dom_trngl} that 
	\begin{equation*}
		\wt{D}_{x_0,t_0}\subset U_{x_0,M,t_0}.
	\end{equation*}
	Therefore, by Proposition~\ref{prp:trngl_kob_dist_lb},
	\begin{equation*}
		\forall\,h\in (0,h_0),\; \koba_{\wt{D}_{x_0,t_0}}(z_0,u_h) \geq 
		\koba_{U_{x_0,M,t_0}}(z_0,u_h) \geq \frac{\pi}{8\cot^{-1}(M)}\log(1/h)-C.
	\end{equation*}
	Combining this with \eqref{eqn:alph_log_trunc_Taka_fin} we obtain, by once again 
	possibly increasing $C$,
	\begin{equation*}
		\forall\,h\in (0,h_0),\; \frac{\pi}{8\cot^{-1}(M)}\log(1/h)\leq \alpha\log(1/h)+C.
	\end{equation*}
	But since $\pi/(8\cot^{-1}(M))>\alpha$, the inequality above gives an immediate 
	contradiction when $h\to 0+$. This contradiction shows that, indeed, 
	$D_{T,x_0,r_1}$ is not locally Goldilocks at $\wh{z}_0$, as required.  
\end{proof}
Now we will see why this suffices to show that no point of $\bdy V_T$ is
a local Goldilocks point. Suppose, to get a contradiction, that there exists some
point $q$ of $\bdy V_T$ that is a local Goldilocks point. Then there exist a small
positive real number $r$, a $z_0\in V_T$, an $\alpha\geq 1$ and a $C<\infty$ such 
that
\begin{equation} \label{eqn:assmp_loc_Gold_V_T}
	\forall\,z\in D(q;r)\cap V_T,\; \koba_{V_T}(z_0,z)\leq \frac{\alpha}{2}
	\log\left(\frac{1}{\dtb{V_T}(z)}\right)+C.
\end{equation}
By the density of the dyadic rationals in $\R$ and the explicit form of $V_T$, we
see that there exists $p\in D(q;r)\cap \bdy V_T$ such that one can write 
\eqref{eqn:xpr_bdy_nbd} and \eqref{eqn:xpr_bdy_nbd_intrsc_dom}, with $x_0$ a
dyadic rational. We choose a small $s>0$ such that
\[ \clos{D}(p;2s)\cap V_T \subset D_{T,x_0,r_1} \cap D(q;r), \]
where $D_{T,x_0,r_1}$ now stands for the appropriately translated and rotated
version of the domain that we were previously calling by the same name 
(obviously, this makes no difference). We may suppose, without loss of
generality, that $z_0\in D(p;s)\cap V_T$ and (by shrinking $s$ further, if 
necessary) that
\begin{equation} \label{eqn:Taka_dom_trunc_dyad_rat_dtb_eq}
	\forall\,z\in D(p;2s)\cap V_T,\; \dtb{D_{T,x_0,r_1}}(z)=\dtb{V_T}(z).
\end{equation}
Then, by \eqref{eqn:assmp_loc_Gold_V_T} and \eqref{eqn:loc_Kob_dist_V_T} we can 
write
\begin{equation} \label{eqn:Taka_dom_trunc_dyad_rat_loc_Gold_assmp_grwth}
	\forall\,z\in D(p;s)\cap V_T,\; \koba_{D(p;2s)\cap V_T}(z_0,z)\leq 
	\frac{\alpha}{2} \log\left(\frac{1}{\dtb{V_T}(z)}\right)+C.
\end{equation}
Now note that
\[ D(p;2s)\cap V_T = D(p;2s)\cap D_{T,x_0,r_1}\subset D_{T,x_0,r_1} \]
and, consequently, \eqref{eqn:Taka_dom_trunc_dyad_rat_loc_Gold_assmp_grwth} 
coupled with \eqref{eqn:Taka_dom_trunc_dyad_rat_dtb_eq} implies:
\begin{equation*}
	\forall\,z\in D(p;s)\cap D_{T,x_0,r_1},\; \koba_{D_{T,x_0,r_1}}(z_0,z)\leq 
	\koba_{D(p;2s)\cap D_{T,x_0,r_1}}(z_0,z) \leq \frac{\alpha}{2}
	\log\left(\frac{1}{\dtb{D_{T,x_0,r_1}}(z)}\right)+C.
\end{equation*}
This last inequality says 
(once again after the appropriate translation and rotation) 
that $D_{T,x_0,r_1}$ satisfies $\alpha$-log-growth near the boundary point 
$x_0+iT(x_0)$. But this contradicts 
Proposition~\ref{prp:trunc_Taka_dom_not_alph_grwth}! Finally, this contradiction
tells us that $V_T$ does not have $\alpha$-log-growth near any boundary point and
that, consequently, it is not locally Goldilocks at {\em any} boundary point.

We now construct an example of an unbounded hyperbolic domain 
in $\C^2$ with infinitely many ends that satisfies the visibility property but such
that there is a large (in particular, non-totally-disconnected) subset of the
boundary each point of which is, {\em in all likelihood}, not a local Goldilocks
point.

We start with the planar domain $U_T$ given in \eqref{eqn:def_U_T}.
Since $U_T$ is a hyperbolic simply connected domain, there exists a Riemann map $\phi:
\unitdisk\to U_T$. Since $U_T$ clearly satisfies Condition~1, by Theorem~\ref{thm:ext_biholo}
$\phi$ extends to a homeomorphism from $\clos{\unitdisk}$ to $\clos{U}_T^{End}$, which we 
shall continue to denote by $\phi$. 

Now define $\Phi:\ball_2\to\C^2$ by
\[ \forall\,z=(z_1,z_2)\in\ball_2,\;\Phi(z)\defeq (\phi(z_1),z_2). \]
Since, for $(z_1,z_2)\in\ball_2$, $|z_1|^2+|z_2|^2<1$, it follows that the mapping
above is well-defined, holomorphic, and indeed a biholomorphism from $\ball_2$ 
onto some domain $\OM\subset \C^2$ whose intersection with $\C\times \{0\}$ is
precisely $U_T\times \{0\}$. In other words,
\begin{equation} \label{eqn:Phi_img_slice_by_0} 
	\Phi\big( \ball_2\cap (\C\times\{0\}) \big) = \Phi(\unitdisk\times\{0\}) =
	U_T\times\{0\}. 
\end{equation}
More generally, for $z_2\in\unitdisk$ arbitrary,
\[ \Phi\big(\ball_2\cap(\C\times\{z_2\})\big) = 
\phi\big(D(0;\sqrt{1-|z_2|^2})\big)\times\{z_2\}. \]
In particular, for every $z_2\in \unitdisk\setminus\{0\}$, 
$\Phi\big(\ball_2\cap(\C\times\{z_2\})\big)$ is bounded. By 
\eqref{eqn:Phi_img_slice_by_0}, it follows that $\OM$ has infinitely many ends
and that the ends of $\OM$ are in one-one correspondence with those of $U_T$. It 
also follows from the fact that $\phi$ is a homeomorphism from $\clos{\unitdisk}$
to $\clos{U}_T^{\text{End}}$ that $\Phi$ is a homeomorphism from $\clos{\ball}_2$
to $\clos{\OM}^{\text{End}}$. It now follows from \Cref{thm:cont_surj_im_vis_dom_vis}
that $\OM$ is a visibility domain.

Note that, by \eqref{eqn:Phi_img_slice_by_0},
\[ \bdy\OM\cap (\C\times\{0\}) = \bdy U_T\times \{0\}. \]
Thus $\bdy\OM$ is highly irregular and it is {\em very unlikely}, given that
$\bdy U_T$ is nowhere differentiable, that $\OM$ is locally Goldilocks at any 
point of $\bdy\OM\cap (\C\times\{0\})$. 


We present some examples of visibility and non-visibilityplanar domains. The first 
example we present is that of a planar domain that satisfies Condition~2 but not Condition~1.
Consequently, it possesses the visibility property, but the fact that it does so can {\em only} be
deduced using Theorem~\ref{T:Visibility_Thm_in_plane}, and not any weaker theorem (in particular, 
it cannot be deduced using \cite[Theorem~1.4]{BZ2023}).

\begin{example}
	Let
	\begin{equation*}
		D' \defeq \{ z\in\C\mid -1<\rprt(z)<1, \; -2+T(\rprt(z))<\iprt(z)<2-T(\rprt(z)) \}. \\
	\end{equation*}
	$D'$ is the rectangle $[-1,1]\times [-2,2]$ with its top and bottom 
	sides replaced with a part of the
	graph of the Takagi function \eqref{eqn:Takagi_fn_def}. 
	Note that
	\[ \{z\in\C\mid -1<\rprt(z)<1,\; \iprt(z)=0\}\subset D' \]
	and that $1,-1\in \bdy D'$. 
	Consider the usual middle-thirds Cantor set $C'\subset [0,1]$. Let
	$C\defeq C'-(1/2)$ ($C'$ translated to the left by $1/2$) and let
	$D\defeq D' \setminus C$. 
	We claim that $D$ satisfies Condition~2. To see this, first note that $\bdy D = \bdy D' \cup C$.
	Therefore, $C$ is a closed, totally-disconnected subset of $\bdy D$. It is easy to see that, for
	every 
	\[ p\in \bdy D \setminus \big( C \cup \{1+2i,1-2i,-1-2i,-1+2i\} \big), \]
	there exists a neighbourhood $U$ of $p$ such that $U\cap D$ is connected, simply connected, and has
	a boundary that is a Jordan curve. Consequently, by \Cref{prp:cond2-sat-xtrnl},
	$D$ satisfies Condition~2. But it does not satisfy 
	Condition~1 because, for every $p\in C$ and {\em every} neighbourhood $U$ of $p$, $U\cap D$ is not
	simply connected (and in fact has a very complicated boundary). Finally, it can be shown in
	precisely the same manner in which Proposition~\ref{prp:trunc_Taka_dom_not_alph_grwth} is proved 
	that no point of 
	\begin{align*} 
		&\{z\in\C\mid -1\leq\rprt(z)\leq 1,\;\iprt(z)=-2+T(\rprt(z))\} \; \cup \\
		&\{z\in\C\mid -1\leq\rprt(z)\leq1,\;\iprt(z)=2-T(\rprt(z))\} 
	\end{align*}
	is a local Goldilocks point of $D$, whence one cannot 
	invoke \cite[Theorem~1.4]{BZ2023} to conclude that $D$ is a visibility domain.   
\end{example}

\begin{example}\label{ex-extenion1}
    Consider the Takagi function $T$ as in \eqref{eqn:Takagi_fn_def} and $D_T=\{x+iy\in \C:\;\; y>T(x)\}$. Consider the domain 
    \[
    \Omega:= D_T\setminus \bigcup_{k=4}^\infty \overline{D(ki,1/3)}
    \]
    Clearly, $\Omega$ satisfies Condition~1. Hence, it is a visibility domain. The domain $\OM$ is not locally Goldilocks near the boundary points of $D_T$. Hence, it does not satisfy the condition of \Cref{res-BZ2023}. Since the sequence of diameters of the components of $\bdy\OM$ is neither in $l^1$ nor in $l^2$. Hence, it does not satisfy the conditions of \Cref{res-LuoYao2022} and \Cref{res-Ntalampekos2023}. But this domain satisfy the both the conditions of \Cref{thm:ext_biholo}. Hence, any biholo $f:\OM_1\to\OM$ extends to a homeomorphism upto the boundary. In particular, any automorphism from $\OM$ extends as a homeomorphism upto the boundary.
\end{example}

We now present an example of a planar domain that shows that if visibility 
is not assumed then there may exist boundary points $\xi$ such that 
{\em any} sufficiently small ball centred at $\xi$ intersects the domain in 
an open set having infinitely many components {\em none of which} has $\xi$ 
as a boundary point.

\begin{example}
	Let
	\[ D_1 \defeq \{z\in\C\mid 0<\rprt(z)<1,\,0<\iprt(z)<1\} \setminus
	\bigcup_{j=2}^\infty \{z\in\C\mid \rprt(z)=1/j,\,0\leq\iprt(z)\leq 1/2\} \]
	and let $\xi$ be the boundary point $(0,1/4)$ of $D_1$. It is easy to see that all 
	disks of radius $< 1/4$ centred at $\xi$ intersect $D_1$ in an open set having
	infinitely many components, none of which has $\xi$ as a boundary point. It is well
	known that $D_1$ is a simply connected domain. Since $\bdy D_1$ is not locally connected
	(it is not locally connected at $\xi$, for instance), it follows from Corollary~3.4
	in \cite{BNT2022} that $D_1$ is not a visibility domain. One could also conclude 
	that $D_1$ is not a visibility domain using Theorem~\ref{Th:vis_bound_reg}.
	
	Now let us modify the definition of $D_1$ slightly and consider the domain
	
	\begin{align*} 
		D_2 \defeq &\{z\in\C\mid 0<\rprt(z),\iprt(z)<1\}  
		\setminus \bigcup_{j=2}^\infty \{z\in\C\mid \rprt(z)=1/j,\,1/4\leq\iprt(z)\leq 3/4\}. 
	\end{align*}
	Then this domain is not simply connected, and so we cannot invoke Corollary~3.4 in
	\cite{BNT2022} to conclude that it is not a visibility domain. However, we can still invoke 
	Theorem~\ref{Th:vis_bound_reg} to conclude that $D_2$ is not a weak visbility domain
	(and hence not a visibility domain).
\end{example}

\begin{example}
	Let
	\[ D \defeq \unitdisk \setminus \left( \bigcup_{\nu=1}^\infty 
	\{re^{i/\nu} \mid 0\leq r\leq 1/2\} \cup \{x+iy\mid y=0,\; 0\leq x\leq 1/2\} \right) . \]
	Then it is clear that if we consider $0\in\bdy D$, then for all $r\in (0,1/2)$,
	$D(0;r)\cap D$ has infinitely many components, each of which has $0$ as a boundary
	point. Consequently, by Theorem~\ref{Th:vis_bound_reg}, it follows that $D$ is not a
	visibility domain. To the best of our knowledge, one cannot directly conclude this
	using any previously known theorem. We also note another fact, which is that each
	connected component of $D(0;r)\cap D$ is a simply connected domain with boundary
	a Jordan curve, and so each component is a visibility domain in its own right.
\end{example}

\begin{example}
	Let 
	\[ S \defeq \left( \bigcup_{j=3}^\infty 
	\{re^{i\theta}\mid 1-1/j\leq r\leq 1,\;\theta=0,2\pi/j,\dots,2\pi(j-1)/j\} \right)
	\cup \bdy\unitdisk. \]
	Then it is clear that $S$ is a closed set (in $\C$, or in $\clos{\unitdisk}$) and it
	is also clear that $D\defeq \unitdisk\setminus (S\cup [-1/3,1/3])$ is a domain. 
	Note that $\bdy D$ has two components, one of which is just a closed line segment. 
	It is not too difficult to see that $\bdy D$ is locally connected. Using this same
	local connectedness, it is not difficult to see that $D$ satisfies Condition~2.
	Hence, by \Cref{T:Visibility_Thm_in_plane}, $D$ is a visibility domain.
	Note that $S^1\subset \bdy D$ and that for every $\xi\in S^1$ 
	there exists a neighbourhood $U$ of $\xi$ such that $U\cap D$ is simply connected
	and has a boundary that is locally connected. Also note that, for every $\xi\in S^1$,
	there does not exist {\em any} neighbourhood $V$ of $\xi$ such that $V\cap D$ is
	simply connected with its boundary a Jordan curve.

\end{example}
The next example gives examples of domains that satisfy Condition~2; 
equivalently, domains that have locally connected boundary outside a 
totally disconnected set. These domains fall in the category of {\bf slit 
domains}.

\begin{example}
Consider a sequence $(y_n)_{n\geq 1}$ in $\R$ such that 
$\inf_{j \neq n} | y_n - y_j| \geq \delta$ for some $\delta > 0$. Let 
$(L_n)_{n\geq 1}$ be a sequence of {\em proper} closed subsets of $\R$. 
Then $L \defeq 
\cup_n \{x + iy_n: x \in L_n \}$ is a closed subset of $\C$. Now, we 
define a domain
\[
\Omega \defeq \C \setminus L.
\]
Note that $L = \partial \Omega$.
We claim that $\Omega$ satisfies Condition~2. To prove this claim,
let $S \defeq L \setminus \cup_n (\text{int}(L_n) \times \{y_n\})  $,
where $\text{int}(L_n)$ is the set of all interior points of $L_n$ considered
as a subset of $\R$.
Let  $\zeta \in \partial \Omega \setminus S = L \setminus S$. Then there 
exists $n_0 \in \posint$ such that $\zeta = x_0 + iy_{n_0}$, where $x_0 
\in \text{int}(L_{n_0})$. Hence, $x_0 \in I$ where $I$ is a connected 
component of $L_{n_0}$. Since $L_{n_0}$ is a closed subset of $\R$, we 
have $I$ is also a closed subset of $\R$. Since $x_{n_0}$ is an interior 
point of $L_{n_0}$, hence, $x_{n_0}$ is also an interior point of $I$. Let 
$r> 0$ such that $(x_0 - r, x_0 + r) \subset I$. Now, let $\epsilon = 
\min \{r, \delta\}$. Then $B(\zeta, \epsilon) \cap L = (x_0 - r, x_0 + r) 
\times \{y_{n_0}\} $. This shows that $D^{+} \defeq \{z \in \C: |z - 
\zeta|< \epsilon \,\,\text{and}\,\, \iprt(z - \zeta) > 0 \}$ and $D^{-} 
\defeq \{z \in \C: |z - \zeta|< \epsilon\,\, \text{and}\,\, \iprt(z - 
\zeta) < 0 \}$ are subsets of $\Omega$. Now, if we can show that $S$ is a 
totally disconnected set, then 
it will follow that $\Omega$ satifies Condition~2. 
\smallskip

To see that $S$ is a totally disconnected set, 
let $P$ be a connected component of $S$. Assume, to get a contradiction, 
that the cardinality of $P$ is stricly greater than $1$ .
Now, note that since $P$ is a connected subset of $L$, there exists $j 
\in \posint $ such that $P \subset  L_j \times \{y_j\}$. This implies 
that $P = I\times \{y_j\}$ for some 
interval $I  \subset \R$. Since the cardinality of $P$ is strictly 
greater than $1$, the cardinality of $I$ is also strictly 
greater than $1$. Let $c, d\in I$ with $c<d$.  Then $(c + d)/2$ is an 
interior point of $L_j$, 
and so $\zeta_0 = \frac{c + d}{2} + i y_j \in L \setminus S$. This is a 
contradiction to $\zeta_0 \in S$. This proves our claim that the domain 
$\Omega$ satisfies Condition~2. As a consequence of this, $\Omega$ is a 
visibility domain.
\end{example}

Finally we provide an example illustrating the comment made after the 
proof of Theorem~\ref{Th:vis_bound_reg}.
\begin{example} \label{xmp:nbd_bss_intrsc_infnt}
	
	First, for every $n\in\posint$, choose $r_n>0$ such that $1/n-r_n>1/(n+1)$. For every
	$n\in\posint$, define
	\[ T_n \defeq \bigcup_{\nu=1}^\infty \{ z\in\C \mid \rprt(z)=(1/n)-(r_n/\nu),\, 0\leq\iprt(z)\leq \sqrt{(1/n)^2-\rprt(z)^2} \}, \]
	and then define 
	\[ D\defeq \mathbb{U}\setminus \bigcup_{n=1}^\infty T_n, \]
	where $\mathbb{U}$ denotes the open upper half plane. Note that $\cup_{n=1}^{\infty} T_n$ 
	is a closed subset of $\mathbb{U}$ (the only limit points of $\cup_{n=1}^{\infty}T_n$ that
	are not in it are the points $1/n$, $n\in\posint$). Therefore $D$ is an open set in $\C$,
	and it is easy to see that $D$ is a domain, is simply connected, and has
	locally connected boundary. Consequently, by Theorem~\ref{thm:vis_simp_conn_loc_conn},
	it follows that $D$ is a visibility domain. Now consider the boundary point $0$ of
	$D$ and consider the neighbourhood basis $D(0;1/n)$, $n\in\posint$, of $0$. Then 
	$D(0;1/n)\cap D$ has infinitely many connected components (although only one of them has 
	$0$ as a boundary point). 
	
\end{example}

\section{Appendix: $\clos{\OM}^{End}$ is sequentially compact}\label{S:end_seq_comp}
In this section, we present the proof of Result~\ref{res:end_seq_comp} as stated
in subsection~\ref{SS:endcpt}. To prove the result,
we need a lemma. 

\begin{lemma} \label{prp:well-bhvd-nbd-end-ntrscts-dom}
Suppose that $\OM\subset\C^d$ is an unbounded hyperbolic domain.
Suppose that, for every compact $K\subset\clos{\OM}$, there exists $R<\infty$ such
that $K\subset B(0;R)$ and such that there are only finitely many components
of $\clos{\OM}\setminus K$ that intersect $\C^d\setminus\clos{B}(0;R)$.
Then, for every end $\mathfrak{e}$ of $\clos{\OM}$ and for every neighbourhood
$U$ of $\mathfrak{e}$ in $\clos{\OM}^{End}$, $U\cap\OM\neq\emptyset$.
\end{lemma}

\begin{proof}
Since $U$ is a neighbourhood of $\mathfrak{e}$ in $\clos{\OM}^{End}$,
therefore, by definition, there exist a compact exhaustion $(K_j)_{j\geq 1}$ of
$\clos{\OM}$ and a decreasing sequence $(F_j)_{j\geq 1}$, where for each
$j$, $F_j$ is a connected component of $\clos{\OM}\setminus K_j$, such that
$\wh{F}_{j_0}\subset U$ for some $j_0\in\Z_{+}$. 
By hypothesis, there exists $R<\infty$ such that $K\subset B(0;R)$ and such that only 
finitely many components of $\clos{\OM}\setminus K_{j_0}$ intersect
$\C^d\setminus\clos{B}(0;R)$. $F_{j_0}$, being unbounded, is one of 
these components. Let the other components
of $\clos{\OM}\setminus K_{j_0}$ that intersect $\C^d\setminus\clos{B}(0;R)$ be
$G_1,\dots,G_m$. Note that $F_{j_0}, G_1, \dots, G_m$,
being connected components of $\clos{\OM}\setminus K_{j_0}$, are disjoint
closed subsets of $\clos{\OM}\setminus K_{j_0}$; therefore, 
given $x_0\in F_{j_0}\setminus\clos{B}(0;R)$ we can choose 
	$r>0$ so small that 
	\begin{equation} \label{eqn:small-ball-x0-sep-cond} 
	B(x_0;r)\cap\clos{B}(0;R) = B(x_0;r)\cap G_1 = \dots = B(x_0;r)\cap G_m = \emptyset. 
	\end{equation}
	Since $x_0\in F_{j_0}\subset\clos{\OM}$, $B(x_0;r)\cap\OM\neq\emptyset$. Now, 
	\[ B(x_0;r)\cap\OM \subset \clos{\OM}\setminus\clos{B}(0;R) = \big(F_{j_0}\setminus\clos{B}(0;R)\big) \cup \bigcup_{i=1}^m \left(
	G_i\setminus\clos{B}(0;R) \right), \]
	so, from \eqref{eqn:small-ball-x0-sep-cond} it follows that $B(x_0;r)\cap\OM\subset F_{j_0}$. Thus, since $B(x_0;r)\cap\OM\neq\emptyset$, $F_{j_0}
	\cap\OM\neq\emptyset$, and hence $U\cap\OM\neq\emptyset$ as well.
\end{proof}
We now present the proof of the aforementioned result.  
\subsection{Proof of the Result~\ref{res:end_seq_comp}}
\begin{proof}
	We need to prove that any sequence $(z_n)_{n\geq 1}$ in $\clos{\OM}^{End}$ has a subsequence that converges to some point $z^0\in\clos{\OM}^{End}$.
	It is clear that we may focus separately on sequences in $\clos{\OM}$ and on sequences in $\clos{\OM}^{End}\setminus\clos{\OM}$. Let us first deal
	with a sequence $(z_n)_{n\geq 1}$ in $\clos{\OM}$.
	Suppose first that there exists $M<\infty$ such that, for infinitely many $n\in\posint$, $\|z_n\|\leq M$. Then it is clear that there is some 
	subsequence of $(z_n)_{n\geq 1}$ that converges to some point $z^0\in\clos{\OM}$. Therefore we
	may assume that, for every $M<\infty$, there exist only finitely many $n$ such that $\|z_n\|\leq M$. This means precisely that $\|z_n\|\to\infty$.
	Now consider the compact exhaustion $(\clos{B}(0;j)\cap\clos{\OM})_{j\geq 1}$ of $\clos{\OM}$. By hypothesis, for every $j\in\posint$, there exists
	$R_j$, $j<R_j<\infty$, such that there are only finitely many components of $\clos{\OM}\setminus\clos{B}(0;j)$ that intersect 
	$\C^d\setminus\clos{B}(0;R_j)$.
	Also, using the fact that $\|z_n\|\to\infty$, choose, for every $j\in\posint$, $N_j\in\posint$ such that, for all $n\geq N_j$, $\|z_n\|>R_j$. 
	Now we will define a subsequence of $(z_n)_{n\geq 1}$ inductively and show that it converges to an end of $\clos{\OM}$. We proceed as follows.
	For every $j\in\posint$, let the finitely many components of $\clos{\OM}\setminus
	\clos{B}(0;j)$ that intersect $\C^d\setminus\clos{B}(0;R_j)$ be $W^j_1,\dots,W^j_{k_j}$ (it follows that, for every $j$, $k_j\geq 1$, i.e., there
	is at least one component of $\clos{\OM}\setminus\clos{B}(0;j)$ that intersects $\C^d\setminus\clos{B}(0;R_j)$ because, otherwise, one would have
	$\clos{\OM}\subset\clos{B}(0;R_j)$, contrary to hypothesis). Now $(z_n)_{n\geq N_1}$ is a sequence in $\clos{\OM}\setminus\clos{B}(0;1)$ that 
	is contained in $\C^d\setminus\clos{B}(0;R_1)$; therefore, it is contained in
	\[ \bigcup_{\nu=1}^{k_1} \big( W^1_\nu \setminus \clos{B}(0;R_1) \big); \]
	therefore, there exists $\nu_1\in\{1,\dots,k_1\}$ such that $W^1_{\nu_1}\setminus\clos{B}(0;R_1)$ contains $z_n$ for infinitely many $n$. Let $A_1
	\subset\posint$ be an infinite set such that $z_n\in W^1_{\nu_1}$ for all $n\in A_1$. Now assuming that $\mu\in\posint$ and that we have obtained:
	(1) $\nu_1\in\{1,\dots,k_1\},\dots,\nu_\mu\in\{1,\dots,k_{\mu}\}$, (2) infinite subsets $A_1,\dots,A_\mu$ of $\posint$ such that $(1')$ 
	$W^\mu_{\nu_\mu}\subset\dots\subset W^1_{\nu_1}$, $(2')$ $A_\mu\subset\dots\subset A_1$, $(3)$ for every $i\in\{1,\dots,\mu\}$, $z_n\in W^i_{\nu_i}$
	for all $n\in A_i$, we attempt to find $\nu_{\mu+1}\in \{1,\dots,k_{\mu+1}\}$ and an infinite subset $A_{\mu+1}$ of $A_\mu$ such that 
	$W^{\mu+1}_{\nu_{\mu+1}}\subset W^\mu_{\nu_\mu}$ and such that $z_n\in W^{\mu+1}_{\nu_{\mu+1}}$ for every $n\in A_{\mu+1}$. Now note that 
	$W^\mu_{\nu_\mu}$ is a component of $\clos{\OM}\setminus\clos{B}(0,\mu)$ that intersects $\C^d\setminus\clos{B}(0;R_\mu)$ and that $z_n\in 
	W^\mu_{\nu_\mu}$ for every $n\in A_\mu$, which is an infinite subset of $\posint$. Note that $W^\mu_{\nu_\mu}\not\subset\clos{B}(0;R_{\mu+1})$ 
	because, if it is, the fact that $\|z_n\|\to\infty$ as $n\to\infty$ will be contradicted. Thus, $W^\mu_{\nu_\mu}\setminus\clos{B}(0;R_{\mu+1})\neq
	\emptyset$. For all $n\in A_\mu$ with $n\geq N_{\mu+1}$, $z_n\in W^\mu_{\nu_\mu}\setminus\clos{B}(0;R_{\mu+1})$. Now since
	\[ \clos{\OM}\setminus\clos{B}(0;R_{\mu+1}) = \bigcup_{i=1}^{k_{\mu+1}} W^{\mu+1}_i \setminus \clos{B}(0;R_{\mu+1}), \]
	it follows that
	\[ W^\mu_{\nu_\mu}\setminus\clos{B}(0;R_{\mu+1}) = \bigcup_{i=1}^{k_{\mu+1}} \big( W^\mu_{\nu_\mu}\cap W^{\mu+1}_i \big) \setminus 
	\clos{B}(0;R_{\mu+1}). \]
	From the above two facts it follows that there exists $i\in \{1,\dots,k_{\mu+1}\}$ such that $z_n\in \big( W^\mu_{\nu_\mu}\cap W^{\mu+1}_i \big) 
	\setminus \clos{B}(0;R_{\mu+1})$ for infinitely many $n\in A_\mu$. Let us choose such an $i$ and call it $\nu_{\mu+1}$. 
	The foregoing implies that there
	exists an infinite subset $A_{\mu+1}$ of $A_\mu$ such that, for all $n\in A_{\mu+1}$, $z_n\in \big( W^\mu_{\nu_\mu}\cap W^{\mu+1}_{\nu_{\mu+1}} \big) 
	\setminus \clos{B}(0;R_{\mu+1})$. Now $W^{\mu+1}_{\nu_{\mu+1}}$ is a connected component of $\clos{\OM}\setminus\clos{B}(0;\mu+1)$, hence a connected
	subset of $\clos{\OM}\setminus\clos{B}(0;\mu)$; furthermore, it intersects the connected component $W^\mu_{\nu_\mu}$ of $\clos{\OM}
	\setminus\clos{B}(0;\mu)$, and so must be included in it: $W^{\mu+1}_{\nu_{\mu+1}}\subset W^\mu_{\nu_\mu}$. And so the induction step is complete.
	Therefore we have sequences $(\nu_\mu)_{\mu\geq 1}$ and $(A_\mu)_{\mu\geq 1}$ such that, for every $\mu$, $\nu_\mu\in\{1,\dots,k_\mu\}$; such that,
	for every $\mu$, $W^{\mu+1}_{\nu_{\mu+1}}\subset W^\mu_{\nu_\mu}$; such that, for every $\mu$, $A_\mu$ is an infinite subset of $\posint$; such that,
	for every $\mu$, $A_{\mu+1}\subset A_\mu$; and such that, for every $n\in A_\mu$, $z_n\in W^\mu_{\nu_\mu}$. Now note that the sequence 
	$(A_\mu)_{\mu\geq 1}$ of nested infinite subsets of $\posint$ determines a subsequence $(z_{m_n})_{n\geq 1}$ of $(z_n)_{n\geq 1}$ with the following
	property: for every $\mu\in\posint$ and every $n\geq\mu$, $z_{m_n}\in W^\mu_{\nu_\mu}$. From this, from the fact that the sequence 
	$(W^\mu_{\nu_\mu})_{\mu\geq 1}$ is nested, and from the fact that $W^\mu_{\nu_\mu}$ is a connected component of $\clos{\OM}\setminus
	\clos{B}(0;\mu)$, it follows immediately that $(z_{m_n})_{n\geq 1}$ converges to the end of $\clos{\OM}$ that $(W^{\mu}_{\nu_\mu})_{\mu\geq 1}$
	defines. Therefore we have proved that any sequence $(z_n)_{n\geq 1}$ in $\OM$ converges to some point of $\clos{\OM}^{End}$.
	
	It only remains to prove that an arbitrary sequence in $\clos{\OM}^{End}\setminus\clos{\OM}$ (i.e., a sequence of ends of $\clos{\OM}$) has a 
	subsequence that converges to
	a point of $\clos{\OM}^{End}$. But, by Lemma~\ref{prp:well-bhvd-nbd-end-ntrscts-dom}, this is now easy to prove. Assume that $(z_n)_{n\geq 1}$
	is a sequence in $\clos{\OM}^{End}\setminus\clos{\OM}$. We work with the same compact exhaustion as before. For every $n\in\posint$, there exists a
	sequence $(F^n_j)_{j\geq 1}$ such that, for every $j$, $F^n_j$ is a connected component of $\clos{\OM}\setminus\clos{B}(0;j)$, such that 
	$F^n_{j+1}\subset F^n_{j}$, and such that $(\wh{F}^n_j)_{j\geq 1}$ is a neighbourhood basis for the topology of $\clos{\OM}^{End}$ at $z_n$. By 
 Lemma~\ref{prp:well-bhvd-nbd-end-ntrscts-dom}, we know that,
	for every $n\in\posint$ and every $j\in\posint$, $F^n_j\cap\OM\neq\emptyset$. So, for every $n\in\posint$, choose $w_n\in F^n_n\cap\OM$. Now 
	$(w_n)_{n\geq 1}$ is a sequence in $\OM$ and therefore, by what we have proved (note that $\|w_n\|\to\infty$), there exists an end $\mathfrak{e}$ of
	$\clos{\OM}$ to which some sub-sequence of $(w_n)_{n\geq 1}$, say $(w_{\tau(n)})_{n\geq 1}$, converges. We want to show that $(z_{\tau(n)})_{n\geq 1}$
	converges to $\mathfrak{e}$. To do this, let $\mathfrak{e} = (H_j)_{j\geq 1}$ where
 $(H_j)_{j\geq 1}$ is a decreasing sequence of connected components of 
 $\clos{\OM}\setminus\clos{B}(0;j)$, and let us
 consider an arbitrary neighbourhood $U$ of $\mathfrak{e}$ in $\clos{\OM}^{End}$. By the definition of the end compactification, there exists $j_0\in\posint$ such that 
 $\wh{H}_{j_0}\subset U$. Since
	$(w_{\tau(n)})_{n\geq 1}$ converges to $\mathfrak{e}$, there exists $n_0\in\posint$ (and we may suppose that $n_0\geq j_0$) such that, for every $n\geq
	n_0$, $w_{\tau(n)}\in H_{j_0}$. But by definition $w_{\tau(n)}\in F^{\tau(n)}_{\tau(n)}$. Now $F^{\tau(n)}_{\tau(n)}$ is a connected component of
	$\clos{\OM}\setminus\clos{B}(0;\tau(n))$. Since $\tau(n)\geq n\geq n_0\geq j_0$, $\clos{\OM}\setminus\clos{B}(0;\tau(n)) \subset \clos{\OM}
	\setminus \clos{B}(0;j_0)$. So $F^{\tau(n)}_{\tau(n)}$, being a connected subset of $\clos{\OM}\setminus\clos{B}(0;\tau(n))$, is also a connected
	subset of $\clos{\OM}\setminus\clos{B}(0;j_0)$ and, therefore, is included in a connected component of $\clos{\OM}\setminus\clos{B}(0;j_0)$. But
	$F^{\tau(n)}_{\tau(n)}\cap H_{j_0}\neq\emptyset$ (since $w_{\tau(n)}\in F^{\tau(n)}_{\tau(n)}\cap H_{j_0}$) and so $F^{\tau(n)}_{\tau(n)}\subset H_{j_0}$.
 Therefore $\wh{F}^{\tau(n)}_{\tau(n)}\subset \wh{H}_{j_0}$ and, since $z_{\tau(n)}\in
\wh{F}^{\tau(n)}_{\tau(n)}$, $z_{\tau(n)}\in \wh{H}_{j_0}$.
This holds for all $n\geq n_0$. Since $\wh{H}_{j_0}\subset U$ and the neighbourhood $U$ of 
$\mathfrak{e}$ was arbitrary, this shows that $(z_{\tau(n)})_{n\geq 1}$ converges to 
$\mathfrak{e}$. This completes the proof.
\end{proof}

 {\bf Acknowledgements:} Sushil Gorai is partially supported by a Core Research Grant
(CRG/2022/003560) from Science and Engineering Research Board, Department of Science and 
Technology, Government of India. Anwoy Maitra is supported by an INSPIRE Faculty Fellowship
(DST/INSPIRE/04/2021/000262) from the Department of Science and Technology, Government of 
India.

\end{document}